\documentclass[11pt,a4paper]{amsbook}
\usepackage[latin1]{inputenc}
\usepackage{amsmath}
\usepackage{amsfonts}
\usepackage{amssymb}
\usepackage{amsthm}
\usepackage{xfrac}
\usepackage{graphicx}
\usepackage{enumitem}
\usepackage{nicefrac}

\usepackage{color,hyperref}
\definecolor{darkblue}{rgb}{0.0,0.0,0.3}
\hypersetup{colorlinks,breaklinks,
            linkcolor=darkblue,urlcolor=darkblue,
            anchorcolor=darkblue,citecolor=darkblue}

\usepackage{afterpage}

\numberwithin{section}{chapter}

\theoremstyle{plain}
\newtheorem{theorem}{Theorem}[chapter]
\newtheorem{prop}{Proposition}[section]
\newtheorem*{theorem*}{Theorem}
\newtheorem*{prop*}{Proposition}
\newtheorem{lem}[prop]{Lemma}
\newtheorem{lemma}[prop]{Lemma}
\theoremstyle{definition}
\newtheorem{definition}[prop]{Definition}

\theoremstyle{remark}
\newtheorem{remark}[prop]{Remark}
\newtheorem*{remark*}{Remark}
\newtheorem{corollary}[prop]{Corollary}
\newtheorem*{corollary*}{Corollary}

\newtheoremstyle{named}{}{}{\itshape}{1.5em}{\scshape}{.}{5pt plus 1pt minus 1pt}{\thmnote{#3}}
\theoremstyle{named}
\newtheorem*{namedtheorem}{}

\numberwithin{equation}{chapter}
\numberwithin{figure}{chapter}

\newcommand{\Cc}{{\mathcal C}}
\newcommand{\Ee}{{\mathcal E}}
\newcommand{\Hh}{{\mathcal H}}

\author{Latif Eliaz}
\title[The Essential Spectrum on Trees]{On the Essential Spectrum of Schr\"odinger Operators on Graphs \vspace{160pt} \\ 
\fontsize{12pt}{16pt}\selectfont\textsc{Thesis for the degree of \\ ``Doctor of Philosophy" \vspace{12pt} \\ By} \vspace{-12pt}}

\begin{document}
\sloppy

\begin{titlepage}
    \begin{center}
        \vspace*{1cm}
 
        \Huge
        \textbf{On the Essential Spectrum of Schr\"odinger Operators on Graphs}
 
        \vspace{1.5cm}
        \Large
        Thesis for the degree of \\ ``Doctor of Philosophy" \vspace{12pt} \\ By
 
        \vspace{0.5cm}
 
        \textbf{Latif Eliaz}
 
        \vfill

        \vspace{0.8cm}

        \large
        Submitted to the Senate of the Hebrew University of Jerusalem \\
        February 2019\\

        \vspace{0.2cm}

    \end{center}

\end{titlepage}

\thispagestyle{plain}

\begin{center}
        \large
        \vspace{1.2cm}
        This work was carried out under the supervision of\\
        Prof.~Jonathan Breuer
\end{center}
\newpage

\begin{center}

        \Large\bf
        Acknowledgements
        \vspace{1.2cm}
\end{center}
        
\normalsize {\setlength{\parindent}{0pt}
I would like to thank my advisor Jonathan Breuer, who has been a source of inspiration for me in his mathematical thought. I am grateful for his patience and attention. This thesis would not have been possible without his dedication.

\vspace{6pt}
I am indebted to my colleagues for supporting me. 
I would like to thank Siegfried Beckus and Mathias Keller for interesting discussions and inspiration. Parts of the last chapter of this work are due to the collaboration with Siegfried, for which I am grateful.
Many thanks to my officemates for the enjoyable time together. Especially I would like to thank Amitai Yuval for the helpful conversations.

\vspace{6pt}
I would like to thank Yoram Last and Yehuda Pinchover for their involvement and useful suggestions.
Also, I would like to thank the anonymous referees for the useful comments and especially for pointing out the reference \cite{LenzTeplyaev}.

\vspace{6pt}
It is an honor to thank those who have established the roots of my mathematical education and research: my teachers, and the mathematicians whose research is at the foundation of this project. It is also a pleasure to thank the Hebrew University's Mathematics Department for providing me appropriate conditions for research.

\vspace{6pt}
As for the roots of my existence, I am grateful to my family, parents, sisters and friends.
I owe my deepest gratitude to my wife Yael,  for supporting me in various aspects of life during this period. 
And of course, I owe many thanks to my son Michael Nur, whose existence brought me a new perspective on my time, my past and my future.
}

\newpage        

\thispagestyle{plain}
\begin{center}
    \Large
    \textbf{On the Essential Spectrum of Schr\"odinger Operators on Graphs}
 
    \vspace{0.4cm}
    \large
  	Thesis for the degree of ``Doctor of Philosophy" \vspace{4pt} \\ By \\
    \vspace{0.4cm}
    \textbf{Latif Eliaz}\\
 	\vspace{0.4cm}
	 This work was carried out under the supervision of\\
    \textbf{Prof.~Jonathan Breuer} \\
    \vspace{0.9cm}
    \textbf{Abstract}
\end{center}

This work studies geometrical characterizations of the essential spectrum $\sigma_{\text ess}$ of Schr\"odinger operators on graphs.
Especially we focus on generalizing characterizations which are given in terms of the concept of right limits.
Intuitively the set of right limits of a Schr\"odinger operator $H$ on $\ell^2(\mathbb{N})$ includes the limit operators which are obtained by a sequence of left-shifts (moving away to infinity) of $H$.
One characterization, which is known for such operators is that $\sigma_{\text ess}(H)$ is equal to the union over the spectra of right limits of $H$, i.e.\ 
\begin{equation} \label{eq:abstractLastSimon}
\sigma_{\text ess}(H)=\bigcup_{H^{(r)}\textrm{ is a right limit of }H}\sigma(H^{(r)}).\tag{$*$}
\end{equation}
Additionally, the essential spectrum equals to the union  over the sets of ``eigenvalues'' corresponding to bounded eigenfunctions of the right limits of $H$. The natural generalizations of the definition and the above relations to $\mathbb{Z}^{n}$ are known to hold as well. The first characterization above is essentially due to a work by Last-Simon from 2006, in which they prove \eqref{eq:abstractLastSimon} with a closure on the right hand side. The second characterization has been shown independently by  Simon and Chandler-Wilde--Lindner in works from 2011. In this work we study the possibility of generalizing these characterizations of $\sigma_{ess}(H)$ to Schr\"odinger operators on graphs.

In Chapter \ref{chapter2} we focus on the first characterization on graphs of uniform polynomial growth. We show first the validity of the argument of Last-Simon in this case. On the second hand we study the limitation of this argument, and show that it can not be directly generalized to graphs of exponential growth. Moreover, we give an example of a graph of non-uniform polynomial growth on which this characterization fails. 

In Chapter \ref{chapter3} of the work we focus on trees. We review an argument for extending the Last-Simon method to regular trees, and give a constructive proof of \eqref{eq:abstractLastSimon} on the family of operators with a spherically symmetric potential on regular trees, in which also a better understanding of the spectral properties of the problem is obtained. Finally in Section \ref{sec:sparsePot} we implement the results and calculate the essential spectrum for an example of a Schr\"odinger operator with a sparse spherically symmetric potential. 

In Chapter \ref{chapter4} we study the possible generalization of the second characterization to general graphs, and show its validity on graphs of uniform sub-exponential growth. As a consequence we also get \eqref{eq:abstractLastSimon} on these graphs. In the course of this study we encounter a necessity for a reverse Shnol's type result on graphs, which we give here as well. 
In the last section of Chapter \ref{chapter4} we develop some examples for the implementation of the above mentioned results obtained on graphs of sub-exponential growth.

Parts of the content of this work have been published in \cite{BDE}.

\setcounter{tocdepth}{2}
\tableofcontents

\chapter{Overview}\label{chapter1}

\section{Introduction}

\subsection{Motivation and background}
The dynamics of a quantum physical system is described by the Schr\"odinger equation
\begin{equation} \nonumber
i\frac{\partial}{\partial t}\psi=H\psi
\end{equation}
where $H$ is a self-adjoint operator acting on the Hilbert space of the system (the ``Hamiltonian''). Thus, understanding the spectra of such operators is a topic of central importance in mathematical physics. Hamiltonians of single-particle quantum systems are often comprised of a sum of an operator corresponding to the kinetic energy (the Laplacian, denoted by $\Delta$) and an operator corresponding to the potential energy (the `potential', denoted by $V$). We refer to an operator of the form $$
H=\Delta+V
$$
as a Schr\"odinger operator. In this work we consider discrete Schr\"odinger operators defined over graphs. In particular we shall be interested in the \emph{essential spectrum} of such operators and various possible characterizations.

The spectrum of a bounded, self-adjoint operator, $H$ acting on a Hilbert space, $\mathcal{H}$ is the set
\begin{equation} \label{eq:spectrumDef}
\sigma(H)=\left\{\lambda\in\mathbb{C} \mid (H-\lambda I) \textrm{ does not have a bounded inverse} \right\}.
\end{equation}
In the case that $\mathcal{H}$ is finite dimensional, $\sigma(H)$ is simply a set of eigenvalues of finite multiplicity. In the case that $\mathcal{H}$ is not finite dimensional, $\sigma(H)$ is more complicated. It is therefore natural to separate $\sigma(H)$ into two sets. One, the \emph{discrete spectrum}, $\sigma_{\textrm{disc}}(H)$, consisting of isolated eigenvalues of finite multiplicity, and its complement, the \emph{essential spectrum}, $\sigma_{\textrm{ess}}(H)$. In a sense, $\sigma_{\textrm{ess}}(H)$ is the part of the spectrum that results from the infinite dimensionality of $\mathcal{H}$. It was given its name by Weyl in his consideration of the part of the spectrum of a half-line Schr\"odinger operator that is independent of the boundary conditions (see \cite{Weyl} and the discussion in \cite[notes to Section VII.3]{ReedSimonI}).

To make this invariance property of $\sigma_{\textrm{ess}}$ precise, recall that an operator, $K$, defined over $\mathcal{H}$ is called \emph{compact} if it is the limit (in norm) of finite rank operators. Weyl's Theorem (see e.g.\ \cite[Theorem S.13]{ReedSimonI}) states that for $H$ as above $\sigma_{\textrm{ess}}(H)=\sigma_{\textrm{ess}}(H+K)$ for any compact $K$. It follows that for a discrete Schr\"odinger operator $\sigma_{\textrm ess}$ is independent of the information contained in any finite part of the underlying space, and thus, in particular, of boundary conditions.

Compactness has a role in another equivalent definition of the essential spectrum. A bounded operator $A$ is said to be Fredholm if $A$ is invertible modulo compact operators, i.e.\ there exist bounded operators $S,T$ so that $(AS-I)$ and $(TA-I)$ are compact. An equivalent definition for the essential spectrum is given by
\begin{equation} \label{eq:sigmaEssAndFred}
\sigma_{\textrm ess}(H)=\{\lambda\in\mathbb{C} \mid (H-\lambda I) \textrm{ is not Fredholm}\}.
\end{equation}
That is, $(H-\lambda I)$ does not have a bounded inverse also modulo compact operators.

An equivalent description of $\sigma_{\textrm ess}$ that will be very useful for us is known as Weyl's Criterion (see, e.g.\ \cite[Theorem VII.12]{ReedSimonI}).
\begin{theorem}[Weyl's Criterion]\label{thm:Weyl}
Let $H$ be a bounded self-adjoint operator on a separable Hilbert space. Then ${\lambda\in\sigma\left(H\right)}$ iff there exists a sequence $\left\{\psi_n\right\}_{n=1}^\infty\subset\mathcal{H}$ of approximate eigenfunctions for $\lambda$ with unit norm, i.e.\ functions satisfying $\left\Vert\psi_n\right\Vert=1$ and 
\begin{equation}\label{eq:WeylAppEF}
\left\Vert\left(H-\lambda\right)\psi_n\right\Vert\xrightarrow{n\to\infty}0.
\end{equation}
Moreover, ${\lambda\in\sigma_{\text{ess}}\left(H\right)}$ iff there exists an orthonormal sequence $\left\{\psi_n\right\}_{n=1}^\infty\subset\mathcal{H}$ of approximate eigenfunctions for $\lambda$, i.e.\ functions satisfying $\left(\psi_n,\psi_k\right)=\delta_{n,k}$ and \eqref{eq:WeylAppEF}.
\end{theorem}

When considering a Schr\"odinger operator defined over some underlying space, it is natural to wonder whether the independence of the essential spectrum on compact perturbations is expressed through the relation between the operator and the geometry of the underlying space. A concept that has proven useful for studying this type of questions (and whose generalization is central to this thesis) is the concept of `right limit'. Right limits were originally introduced by Last and Simon \cite{LastSimonAC} in their study of the absolutely continuous spectrum of Jacobi matrices (see \eqref{eq:rightlimJacobi} below). There are two equivalent concepts that have been introduced in slightly different settings. One is that of `limit operator' introduced by Muhamadiev \cite{Muham1} in the study of the inversion of  almost periodic differential operators on the real axes and the other is that of `localization at infinity' introduced by Georgescu and Iftimovici \cite{GI1} in the study of spectral properties of quantum Hamiltonians through abelian C$^*$-algebras. While the concept of `right limits' emphasizes the geometry of the underlying space, the concepts of `limit operator' and  `localization at infinity' center on operator theoretic aspects. All these concepts are essentially equivalent in the setting relevant for us. However, in order to explain the relevance to the essential spectrum we use the concept of right limit.

Consider a bounded Schr\"odinger operator $J=\Delta+Q$ acting
on $\ell^{2}\left(\mathbb{N}\right)$. An operator $J^{\left(r\right)}=\Delta+Q^{\left(r\right)}$,
acting on $\ell^{2}\left(\mathbb{Z}\right)$, is a right limit of
$J$ if there exists a sequence of indices $\left\{ n_{j}\right\}_{j=1}^\infty \subseteq\mathbb{N}$
such that for every fixed $l\in\mathbb{Z}$,
\begin{equation} \label{eq:rightlimdef}
Q_{l+n_j}\stackrel{j\to\infty}{\longrightarrow}Q^{(r)}_l.
\end{equation}
Equivalently, if we expand $J$ to an operator $\tilde{J}$ on $\ell^2\left(\mathbb{Z}\right)$ (e.g.\ by putting $Q$ equal to zero on the left half line), we will get $J^{(r)}$ as a strong limit of a sequence of left-shifts of $\tilde{J}$ (corresponding to the sequence $\left\{ n_{j}\right\}_{j=1}^{\infty}$). Note that, by compactness, one can always find a right limit along a subsequence of any such sequence of shifts.

One can consider somewhat more general operators by replacing $\Delta$ with a weighted version: a Jacobi matrix on $\ell^2(\mathbb{N})$ is a tridiagonal matrix with real diagonal entries $\{b_k\}_{k=1}^{\infty}$ and nonnegative off-diagonal entries $\{a_k\}_{k=1}^{\infty}$. A Jacobi matrix on $\ell^2(\mathbb{Z})$ is defined analogously.
A Jacobi matrix $J^{\left(r\right)}$,
acting on $\ell^{2}\left(\mathbb{Z}\right)$, is a right limit of a (bounded)
Jacobi matrix on $\ell^2(\mathbb{N})$ if there exists a sequence of indices $\left\{ n_{j}\right\}_{j=1}^\infty \subseteq\mathbb{N}$
so that for every fixed $l\in\mathbb{Z}$,
\begin{equation} \label{eq:rightlimJacobi}
a_{l+n_j,l+1+n_j}\stackrel{j\to\infty}{\longrightarrow}a^{(r)}_{l,l+1}, \quad b_{l+n_j}\stackrel{j\to\infty}{\longrightarrow}b^{(r)}_{l}.
\end{equation}

The concept of right limits has been extended also to operators on $\mathbb{Z}^n$ and $\mathbb{R}^n$.

Remarkably, given a Schr\"odinger operator $H$ defined on any of the spaces $\ell^2(\mathbb{N})$, $\ell^2(\mathbb{Z}^n)$ and $L^2(\mathbb{R}^n)$,
\begin{equation}\label{eq:sigmaEssChar}
\sigma_{\textrm ess}(H)=\bigcup_{H^{(r)} \textrm{ is a right limit of } H}\sigma\left(H^{(r)}\right).
\end{equation}
Notice that, by \eqref{eq:spectrumDef} and \eqref{eq:sigmaEssAndFred}, this characterization is equivalent to a relation between Fredholmness of the operator $(H-\lambda I)$ to the existence of bounded inverses to all the right limits of it.
This characterization is essentially due to Last--Simon \cite{LastSimonEss} (for related results see \cite{AmreinMP, Anselone, CWL, GI1,GI2,GI3,Kurbatov, LanRab, M1, MPR,Muham1,Muham2, RRS,RRSband, SeidelSil, Shubin1,Shubin2,SimonSz}; comprehensive reviews and further references on the subject can be found in \cite{CWL, LastSimonEss, SimonSz}).
Additionally, Simon \cite{SimonSz} and Chandler-Wilde--Lindner \cite{CWL} show, independently, that for a Schr\"odinger operator on $\ell^2(\mathbb{N})$ and on $\ell^2(\mathbb{Z}^n)$
\begin{equation}\label{eq:sigmaInfChar}
\sigma_{ess}(H) =  \bigcup_{H^{\left(r\right)}\text{ is a right limit of \ensuremath{H}}}\sigma_{\infty}\left(H^{\left(r\right)}\right),
\end{equation}
where, for an operator $K$ on a discrete space $X$, $\sigma_{\infty}\left(K\right)$ denotes the pure point spectrum of $K$ in $\ell^{\infty}\left(X\right)$, i.e.\ \[ \sigma_{\infty}\left(K\right)=\left\{ \lambda\,|\,\exists\psi\in\ell^{\infty}\left(X\right)\,\text{so that }K\psi=\lambda\psi\right\} . \]
\begin{remark*}
The equality \eqref{eq:sigmaEssChar} was stated in \cite{LastSimonEss} with $\overline{\bigcup_{r}\sigma\left(H^{\left(r\right)}\right)}$ on the right hand side. However, $\bigcup_{r}\sigma\left(H^{\left(r\right)}\right)$ is in fact closed (see e.g.\ \cite{RRS} and \cite{SimonSz} for details).
\end{remark*}

Historically, these results were developed along a few different threads. One thread goes back to Favard \cite{Favard} and continues with Muhamadiev \cite{Muham1,Muham2}, Lange--Rabinovich \cite{LanRab}, Rabinovich--Roch--Silbermann \cite{RRS,RRSband} and Chandler-Wilde--Lindner \cite{CWL}. This thread involves the concept of `limit operators' which was originally defined by Muhmadiev in \cite{Muham1} as the family of limits of shifts of certain differential operators. In this thread, results similar to \eqref{eq:sigmaEssChar} and \eqref{eq:sigmaInfChar} relating Fredholmess of the original operator and invertibility of the limit operators were obtained for various families of operators (including certain discrete Schr\"odinger operators).

Another thread involves $C^*$ algebras and was pursued by Georgescu--Iftimovici \cite{GI1,GI2,GI3} and Mantoiu \cite{M1} who define the concept of `localization at infinity' which also coincides in the Schr\"odinger operator case with right limits. They obtain a result similar to \eqref{eq:sigmaEssChar} (in their case, with the closure on the right hand side) for operators on locally compact, non-compact abelian groups.

As mentioned above, the concept of `right limit' was introduced by Last-Simon \cite{LastSimonAC} in their study of absolutely continuous spectrum of Schr\"odinger operators and later used by them in their study of the essential spectrum \cite{LastSimonEss}. Remling \cite{Remling} obtained a remarkable characterization of the right limits of Jacobi matrices and Schr\"odinger operators with absolutely continuous spectrum that has many important implications.

For a detailed review of the rich history of the subject see \cite{CWL}.

The proof of \eqref{eq:sigmaEssChar} and \eqref{eq:sigmaInfChar} in \cite{SimonSz} involves a growth estimate for generalized eigenfunctions corresponding to points in the spectrum. According to a result known as Shnol's Theorem, for certain Schr\"odinger operators, $H$, given $\lambda\in\mathbb{C}$, if there exists a polynomially growing generalized eigenfunction for $\lambda$ then $\lambda\in\sigma(H)$. This result has been originally developed by Shnol \cite{Shnol} (and rediscovered in \cite{SimonEF}, see also \cite[Section 2.4]{CFKS}) in the context of Schr\"odinger operators on the real line, with some restrictions on the growth rate of the potential. A converse of Shnol's Theorem is also known to hold in this context, and is sometimes referred to as an expansion theorem (see \cite[Section C5]{SimonSG} and references therein). According to it, for spectrally almost every point $\lambda\in\sigma(H)$ (as we will define below) there exists a corresponding polynomially growing generalized eigenfunction. Both directions have been developed further to the multidimensional and discrete settings, including $\mathbb{R}^n$, $\mathbb{Z}^n$ and other graphs, and were also studied beyond the scope of Schr\"odinger operators (see \cite{BdmLS09,BdmSt03,FLW14,HK2011,Han, LenzTeplyaev}).

\vspace{0.5cm}

This work is concerned with examining the above problems on graphs. Namely, our aim is to understand the relation between the essential spectrum of a Schr\"odinger operator, $H$, defined over a graph and the limits of $H$ `at infinity', where now `infinity' is approached along paths on the graph.
Let $G$ be a graph with vertices $V\left(G\right)$ and edges $E\left(G\right)$. A Schr\"odinger operator on $G$ is an operator, $H$, acting on $\psi\in\ell^2\left(V\left(G\right)\right)\cong\ell^2\left(G\right)$ by
\begin{equation}
\left(H\psi\right)\left(v\right)=\sum_{u\sim v}\left(\psi\left(u\right)-\psi\left(v\right)\right)+Q\left(v\right)\psi\left(v\right),\label{eq:SchrOp}
\end{equation}
where we denote $u\sim v$ for vertices $u,v\in G$ if $\left(u,v\right)\in E\left(G\right)$ and $Q: V\left(G\right) \rightarrow \mathbb{R}$ is a function (which we take to be bounded throughout the thesis).
Denoting the graph Laplacian by $\Delta$,
\begin{equation} \nonumber
\Delta \psi (v)=\sum_{u\sim v}\left(\psi\left(u\right)-\psi\left(v\right)\right),
\end{equation}
and using $Q$ to denote the multiplication operator by the function $Q$, we write $H=\Delta+Q$.

Analogously to the one dimensional case one can define Jacobi operators on $G$ by replacing $\Delta$ with a weighted version (acting on nearest neighbours with nonnegative weights).

The essential spectrum of Schr\"odinger operators on infinite graphs (other than $\mathbb{Z}^n$) has been studied mostly on trees, where in the context of regular trees, Golenia \cite{Golenia} and Golenia-Georgescu \cite{GolGeor}, have shown the analog of \eqref{eq:sigmaEssChar} when $Q$ has a limit (in the usual sense) along every path to infinity. Fujiwara \cite{Fujiwara} has shown that for rapidly branching trees, the essential spectrum of $\Delta$ consists of a single point. There are, in addition, several works studying the minimum of the essential spectrum on graphs (see, e.g.\ \cite{BoGo} and references therein).

In order to extend \eqref{eq:sigmaEssChar} to general graphs, one needs first to extend the notion of right limit to that setting. While the notion of `limits at infinity' is intuitively clear, it is not immediately obvious how this should be done formally. We do this in Section \ref{sec:def}. We denote by $\mathcal{R}$-limit the notion analogous to right limit in graphs, and by $\mathcal{R}(H)$ the set of $\mathcal{R}$-limits of the operator $H$.

The first result we state is almost an immediate consequence of the definition of the notion of $\mathcal{R}$-limit and the characterization of the essential spectrum via an orthogonal sequence of approximate eigenfunctions (Theorem \ref{thm:Weyl}). Nevertheless, we give a proof of this theorem in Section~\ref{sec:thm1} below for completeness.

\begin{theorem}\label{thm:Thm1}Assume $H$ is a bounded Schr\"odinger operator on
$\ell^{2}\left(G\right)$ where $G$ is a graph of bounded degree, then
\[
\bigcup_{L\in\mathcal{R}(H)}\sigma\left(L\right)\subseteq\sigma_{\text{ess}}\left(H\right).
\]
\end{theorem}

Chapter 2 is devoted to studying the limitations of generalizing the argument of Last--Simon \cite{LastSimonEss} to graphs. As we show there, the results extend to graphs of uniform polynomial growth, i.e.\ in the case that the number of points in each ball is uniformly bounded by a polynomial in the radius, we have that
\begin{equation}\label{eq:sigmaEssLastSimonChar}
\overline{\bigcup_{L\in\mathcal{R}(H)}\sigma\left(L\right)}=\sigma_{\text{ess}}\left(H\right).
\end{equation}
On the other hand, by a closer analysis of the proof, we show that this argument can not simply be generalized to graphs of exponential growth.

The final section of Chapter 2 contains an example of a graph of non-uniform polynomial growth on which \eqref{eq:sigmaEssChar} fails.
The graph in this counterexample is not a tree, but its construction involves the use of a sequence of regular graphs with girth growing to infinity. Such a sequence can be thought of as an approximation of a regular tree.

Regular trees are in a sense canonical examples of graphs of exponential growth which have a simple structure. It is therefore natural to ask whether  \eqref{eq:sigmaEssChar} holds on regular trees, and generally on trees. We study this case in Chapter \ref{chapter3}.
While the results from Chapter \ref{chapter2} may suggest that the answer is negative, it is in fact positive. In our work we have initially considered the case of Schr\"odinger operators $H=\Delta+Q$ on regular trees, for which $Q$ has a spherical symmetry around some fixed root. We shall refer to this case as the spherically symmetric case. In this case we have obtained a constructive proof of \eqref{eq:sigmaEssChar}, where the essential spectrum can also be studied using a one-dimensional Jacobi matrix associated with $H$.

After having proved our results, S.~Denisov informed us of an argument which proves \eqref{eq:sigmaEssLastSimonChar} generally for any Schr\"odinger operator on regular trees. Note the closure in \eqref{eq:sigmaEssLastSimonChar}. The paper \cite{BDE} contains both our and Denisov's results, which we briefly describe in Section~3.2 for completeness. Denisov's proof overcomes the restrictions we have seen in Chapter \ref{chapter2} by using trial functions which are supported on annuli around a fixed origin, instead of functions supported on balls.

Finally, in Chapter \ref{chapter4} we turn to the characterization \eqref{eq:sigmaInfChar}. By adapting and genralizing the work of Simon \cite{SimonSz} to graphs we obtain both \eqref{eq:sigmaEssChar} and \eqref{eq:sigmaInfChar} for graphs of subexponential growth. Note that these results are stronger than the results obtained in Chapter \ref{chapter2}.

In the course of the generalization of Simon's work to graphs we encounter a necessity for a Shnol's type result for graphs. A direct Shnol's result on graphs of sub-exponential growth was developed by \cite{HK2011} (see also \cite{BP2018,BdmLS09,FLW14}). Regarding the inverse direction there exist results in other general settings, including \cite{BdmLS09,BdmSt03}. A recent work of Lenz and Teplyaev \cite{LenzTeplyaev} contains the result relevant to the current setting. We present this result together with a proof in our particular case in Chapter~4.

In the last section of Chapter \ref{chapter4} we develop some examples for the implementation of the above mentioned results obtained on graphs of sub-exponential growth.

\section{Preliminaries}

In this subsection we discuss some preliminaries from spectral theory that are needed in our analysis below.

\subsection{Spectral measures}

There are several notions of spectral measures that we use in this thesis.

Let $H$ be a bounded self-adjoint operator on a separable Hilbert space $\mathcal{H}$.
According to the spectral theorem (see e.g.\ \cite[Theorem VII.8]{ReedSimonI}) there exists a projection valued measure $E(\cdot)$ corresponding to $H$, satisfying for any $\psi\in\mathcal{H}$
\[
\langle\psi,H\psi\rangle=\int_{\mathbb R}x \,d\langle \psi, E(x)\psi\rangle,
\]
where $\langle\cdot,\cdot\rangle$ is the inner product in $\mathcal{H}$.

We call a Borel Measure $\mu$ on $\mathbb{R}$ a \emph{spectral measure} for $H$ if\ \ $\forall S\subseteq \mathbb{R}$ a Borel Set (see \cite[pg.~503]{SimonSG}),
\begin{equation} \label{eq:SpectralMeasure}
E(S)=0 \iff \mu(S)=0.
\end{equation}

Assume now $\psi\in\mathcal{H}$. The \emph{spectral measure of $H$ with respect to $\psi$} is the unique measure on $\mathbb{R}$ satisfying for $z\in\mathbb{C}\backslash\mathbb{R}$ (see e.g.~\cite{ReedSimonI}):
\[
\int_{\mathbb{R}} \frac{d\mu_{\psi}(x)}{x-z}=\langle\psi , (H-z)^{-1}\psi\rangle.
\]
The \emph{cyclic subspace} spanned by $H$ and $\psi$ is given by
\[
\mathcal{H}_\psi=\overline{\textrm{span}\left\{H^n\psi \mid n=0,1,\ldots\right\}}.
\]
Note that the spectral measure of $H$ with respect to a vector $\psi$ is not necessarily a spectral measure in the sense of \eqref{eq:SpectralMeasure}. However, if $\mathcal{H}_\psi=\mathcal{H}$ then the measure $\mu_\psi$ is a spectral measure as defined above. Alternatively, by taking an orthonormal sequence $\{\psi_n\}_{n\in\mathbb{N}}$, a positive sequence $\{a_n\}_{n\in\mathbb{N}}\in\ell^1(\mathbb{N})$, and defining $\mu=\sum_{n\in\mathbb{N}} a_n \mu_{\psi_n}$ we can get a spectral measure satisfying \eqref{eq:SpectralMeasure}.

\subsection{The Borel transform}
Let $\mu$ be a measure on $\mathbb{R}$ satisfying
\[
\int_{\mathbb R}\frac{d\mu(x)}{1+|x|}<\infty.
\]
The \emph{Borel transform} (also known as the \emph{Steijles transform} or the \emph{$m$-function}) of $\mu$ is defined by
\[
F_\mu(z)=\int_{\mathbb R}\frac{d\mu(x)}{x-z},
\]
where $z\in\mathbb{C}_+$.
The function $F_\mu$ captures important properties of the measure $\mu$. The following theorem is a consequence of \cite[Theorem 1.6]{GSrankone}.
\begin{theorem}\label{thm:spectralMeasureProperties}
Given $F_\mu$ and $\mu$ as above. Then,
\begin{enumerate}
\item $\mu(\{E_0\})=\lim_{\varepsilon\to0}\varepsilon\textrm{Im}F_\mu(E+i\varepsilon)$.

\item Assume further that $\mu$ is a spectral measure of a self-adjoint operator $H$. If there exists $\delta>0$ such that $\lim_{\varepsilon\to0}\textrm{Im}F_\mu (E+i\varepsilon)$ is nonzero for all $E$ in an interval $(E_0-\delta,E_0+\delta)$ then $E_0\in\sigma(H)$.

\end{enumerate}

\end{theorem}
\subsection{The resolvent identity}
Assume that $H, H_0$ are (bounded) Schr\"odinger operators. 
Denote by $R,R_0$ the operators, $R(z)=(H-z)^{-1}$ and $R_0(z)=(H_0-z)^{-1}$, defined on the complement of the spectrum in $\mathbb{C}$ (the resolvent set).
The following identity is known as the \emph{resolvent identity}
\[
R_{0}-R=R_{0}\left(H-H_{0}\right)R.
\]

Let $\varphi\in\mathcal{H}$, $\alpha\in\mathbb{R}$, and define $H_\alpha=H_0+\alpha\langle\varphi,\cdot\rangle\varphi$.
As a consequence of the resolvent identity one could get the following relation, which is sometimes known as the basic formula of rank-one perturbations (see e.g.\ \cite{SimonRankOne})
\[
F_\alpha(z)=\frac{F_0(z)}{1+\alpha F_0(z)},
\]
where 
\[
F_\alpha(z)=\langle\varphi , (H_\alpha-z)^{-1}\varphi\rangle.
\]
 
\section{$\mathcal{R}$-limits on graphs}

\subsection{Definitions and notations} \label{sec:def}

This section deals with the extension of the concept of right limits to general graphs with bounded degree. There are two issues that make the analogous notion of right limit for general graphs more complex than that of the one-dimensional object. The first (minor) one is the fact that general graphs may have multiple paths to infinity. This is true already in the case of $\mathbb{Z}^d$ and is the main reason why we refrain from using the name `right limit' in this case and use $\mathcal{R}$-limit instead. The second issue is that with a general graph the absence of homogeneity means that the different $\mathcal{R}$-limits of an operator might be defined on various different graphs which are not necessarily related in a simple way to the graph over which the original operator was defined. Thus, one is faced with the requirement to compare operators defined over different graphs. In order to deal with the first issue, one has to specify a path to infinity. In order to deal with the second one, we need to introduce local mappings to finite dimensional vector spaces which will satisfy a certain compatibility condition with each other.

Let $H$ be a Schr\"odinger operator on a graph $G$ with bounded degree. For any vertex $v\in G$ and $r\in\mathbb{N}$ denote the ball
\[
B_r\left(v\right)=\left\{u\in G\,|\,\text{dist}(u,v)\leq r\right\},
\]
where $\text{dist}(u,v)$ denotes the distance in the graph between $u$ and $v$, which is defined by the length (number of edges) of the shortest path between them.
Denote further by $N_{v,r}$ the number of vertices in this ball, i.e.\ $N_{v,r}=\left|B_r\left(v\right)\right|$. Let $H_r^{(v)}=H|_{B_r(v)}$(=the restriction of $H$ to $\ell^2 \left(B_r(v) \right)$).

Let $\eta$ be an indexing of the vertices of this ball,
\begin{equation} \nonumber
\eta:B_r(v)\to\left\{1,2,\ldots,N_{v,r}\right\}
\end{equation}
and define the corresponding unitary mapping $\mathcal{I}_\eta:\ell^2\left(B_r\left(v\right)\right)\to \mathbb{C}^{N_{v,r}}$ by
\begin{equation} \nonumber
\mathcal{I}_\eta \left(\delta_u\right)=e_{\eta(u)},
\end{equation}
for any vertex $u\in B_r(v)$, where $\delta_u$ is the delta function at $u$ and $\left \{e_1,e_2,\ldots,e_{N_{v,r}} \right \}$ is the standard basis in $\mathbb{C}^{N_{v,r}}$ (i.e.\ $e_i\left(j\right)=\delta_{i,j}$).
Let $M^{(v)}_{\eta, r}\in\mathcal{M}_{N_{v,r},N_{v,r}}$ be the matrix defined by
\begin{equation} \nonumber
M^{(v)}_{\eta,r}=\mathcal{I}_\eta H_r^{(v)} \mathcal{I}_\eta^{-1}.
\end{equation}

\begin{definition}
Fix a vertex $v\in G$ and for any $r \in \mathbb{N}$, let
\begin{equation} \nonumber
\eta_r :B_r(v)\to\left\{1,2,\ldots,N_{v,r}\right\}
\end{equation}
be an enumeration as above, and $\mathcal{I}_r=\mathcal{I}_{\eta_r}$ be the corresponding isomorphism.
We say that the sequence of isomorphisms $\left\{\mathcal{I}_r \right\}_{r=1}^\infty$ is \emph{coherent} if for any $r<s$ and any $u \in B_r(v)$
\begin{equation} \nonumber
\eta_s(u)=\eta_r(u).
\end{equation}
When we want to emphasize the dependence on $v$, we say that  $\left\{\mathcal{I}_r \right\}_{r=1}^\infty$ is a coherent sequence at $v$.
\end{definition}

Note that, if $\{\mathcal{I}_r\}_{r=1}^\infty$ is a coherent sequence of isomorphisms at $v \in G$, then for any $r$, the corresponding matrix $M^{(v)}_{\eta_r,r}$ is the $N_{v,r} \times N_{v,r}$ upper left corner of the matrix $M^{(v)}_{\eta_s,s}$ for any $r \leq s$. Thus, in what follows, when the coherent sequence is clear, we omit the $\eta_r$ and write simply $M^{(v)}_r=M^{(v)}_{\eta_r,r}$.

We say that a sequence of vertices $\left\{v_n\right\}_{n=0}^\infty$ is a path to infinity in $G$ if $v_{n+1}\sim v_n\ \forall n\in\mathbb{N}$, and $|v_n|=\text{dist}\left(v_n,v_0\right)\raisebox{\dimexpr-0.2\height}{\rotatebox{0}{$\xrightarrow{\makebox[0.5cm]{$\rotatebox{-0}{$\scriptstyle{n\to\infty}$}$}}$}}\infty$ monotonically.
\begin{definition} \label{def:rlimdef}
Given a graph $G'$, a vertex $v_0'\in G'$ and a Schr\"odinger operator $H'$ on $G'$, we say that $\left\{H',G',v_0'\right\}$ is an $\mathcal{R}$-limit of $H$ along the path to infinity $\left\{v_n\right\}_{n=0}^\infty$ if there exists a sequence of indices $\left\{n_j\right\}_{j=1}^\infty$, such that
\begin{enumerate}[leftmargin=*]
\item[(i)]For any $j\in\mathbb{N}$ there exists a coherent sequence of isomorphisms $\left\{\mathcal{I}^{(j)}_{k}\right\}_{k=1}^\infty$ at $v_{n_j}$.
\item[(ii)]There exists a coherent sequence of isomorphisms $\left\{\mathcal{I}'_{k}\right\}_{k=1}^\infty$ at ${v_0'}$.
\item[(iii)]For any $r\in\mathbb{N}$ $N_{v_{n_j},r}=N_{v_0',r}$ for all sufficiently large $j$, and
\begin{equation} \label{eq:rlimdef}
\lim_{j\to\infty}M^{\left(v_{n_j}\right)}_r=M^{(v_0')}_{r}.
\end{equation}
\end{enumerate}
\end{definition}

In the one dimensional case, the matrices $M^{(v_j)}_r$ are simply truncated Jacobi matrices and \eqref{eq:rlimdef} translates to the condition \eqref{eq:rightlimdef}. Thus, the definition of $\mathcal{R}$-limits is a direct generalization of the definition of right limits in the one dimensional case.

Note that, as in the one dimensional case, one can always find an $\mathcal{R}$-limit along a subsequence of any given sequence of vertices that move away to infinity.

\begin{lem}\label{lemma:rlimexists}
Let $H=\Delta+Q$ be a bounded Schr\"odinger operator on an infinite graph $\{G,v_0\}$ of a bounded degree. Assume $\left\{u_j\right\}_{j=1}^\infty$ is a sequence of vertices such that $\text{dist}(v_0,u_j)\raisebox{\dimexpr-0.2\height}{\rotatebox{0}{$\xrightarrow{\makebox[0.5cm]{$\rotatebox{-0}{$\scriptstyle{j\to\infty}$}$}}$}}\infty$ monotonically. Then there exists an $\mathcal{R}$-limit of $H$ which is obtained along a subsequence of $\left\{u_j\right\}_{j=1}^\infty$.
\end{lem}
\begin{proof}
Consider the sequence $U_0=\left\{u_j\right\}_{j=1}^\infty$. Since the set $\text{deg}\left(U_0\right)$ of possible vertices degree is finite some degree repeats infinitely many times. Restrict to this subsequence $U_1$ and repeat the argument with the set of neighbours of vertices of $U_1$. Inductively define for every $k\in\mathbb{N}$ such a subsequence $U_k$ for which the vertices degree agree in balls of radius $k$  (under a corresponding isomorphism of these sub-graphs).
We take the diagonal over the resulting subsequences of vertices $\subseteq U_k$ to
define a subsequence $U\subseteq U_0$ for which the vertices degree agree in balls of any radius for index large enough.

Next, since the potential is bounded we can restrict to a subsequence $U'_1\subseteq U$ on which the sequence $Q(U'_1)$ is converging. Again repeat the argument over neighbours of vertices of $U'_1$, and similarly continue inductively and take the diagonal to construct a subsequence on which $H$ converges to an $\mathcal{R}$-limit. 
\end{proof}

\begin{remark}
It is natural to define a topology on the set of weighted, rooted graphs (of bounded degree) by comparing the edge weights on growing spheres around the root. This space is metrizable (see e.g.\ \cite{Lavosz} for details) and it is not hard to see that the convergence we describe to $\mathcal{R}$-limits is the same as convergence of shifts of the graph in that topology. In particular, convergence of the Laplacian to an $\mathcal{R}$-limit can be thought of as a particular case of Benjamini-Schramm convergence (first defined in \cite{BS}), where the root is shifted along a fixed infinite graph. The case of a Schr\"odinger operator (i.e.\ with an added potential) is of course somewhat more general.
\end{remark}

\subsection{Proof of Theorem \ref{thm:Thm1}}\label{sec:thm1}
The proof of the first theorem is straightforward and thus we include it already here.

\begin{proof} [Proof of Theorem~\ref{thm:Thm1}]
Assume $\{H',G',v_0'\}$ is an $\mathcal{R}$-limit of $H$ along a path to infinity $\left\{v_j\right\}_{j=0}^\infty$, and $\lambda\in\sigma\left(H'\right)$.
Given $\varepsilon>0$ define
\[
\varepsilon'=\min\left(\frac{2\varepsilon}{1+\Vert H'\Vert+|\lambda|},\frac{1}{2}\ \right).
\]
According to Weyl's Criterion (Theorem \ref{thm:Weyl}, for $\sigma\left(H'\right)$) there exists ${\psi\in\ell^{2}\left(G'\right)}$
such that $\left\Vert \left(H'-\lambda\right)\psi\right\Vert <\varepsilon'$ and $\left\Vert \psi\right\Vert =1$.
Additionally, since $\psi\in\ell^2\left(G'\right)$, there exists $R>0$ such that
\[
\left\Vert \psi|_{G'\backslash B_R(v_0')}\right\Vert <\varepsilon'.
\]
Thus by defining for every $w\in V(G')$
\[
\varphi(w)=\begin{cases}
\psi(w)/{K} & w\in B_R(v_0') \\
0 & \text{otherwise},
\end{cases}
\]
with
\[
K=\left\Vert \psi|_{B_R(v_0')}\right\Vert>\frac{1}{2},
\]
we get an approximate eigenfunction for $H'$, supported on $B_R\left(v_0'\right)$, and satisfying $\left\Vert \varphi\right\Vert =1$. Indeed
\[
\left\Vert\left(H'-\lambda\right)\varphi\right\Vert=
\left\Vert\left(H'-\lambda\right)\frac{\psi-(\psi-\varphi)}{K}\right\Vert< {{1+\left\Vert H'\right\Vert+\left|\lambda\right|}\over{K}} \varepsilon'\leq\varepsilon.
\]
Since $H'$ is an $\mathcal{R}$-limit of $H$ there exists some $u=v_{n_j}\in G$ so that the corresponding matrices satisfy
\[
\left\Vert M^{(u)}_{R+2} - M^{(v_0')}_{R+2} \right\Vert < \varepsilon.
\]
Let $\mathcal{I}:\ell^2\left(B_{R+2}(u)\right)\to\mathbb{C}^{N_{u,R+2}}$ and $\mathcal{I'}:\ell^2\left(B_{R+2}(v_0')\right)\to\mathbb{C}^{N_{v_0',R+2}}$ be the isomorphisms from Definition~\ref{def:rlimdef}.
Denote by $\chi'$ the function $\chi'=\mathcal{I}^{-1}\mathcal{I}'\widetilde{\varphi} \in\ell^2\left(B_{R+2}(u)\right)$,
where $\widetilde{\varphi}=\varphi|_{B_{R+2}(v'_0)}$. Additionally define,
\[
\chi(w)=\begin{cases}
\chi'(w) & w\in B_{R+2}\left(u\right)\\
0 & \text{otherwise}.
\end{cases}
\]
Then
\[
\left(H-\lambda\right)\chi=\left(H|_{B_{R+2}(u)}-\lambda\right)\chi'
\]
and thus,
\begin{eqnarray}
&\left\Vert \left(H-\lambda\right)\chi\right\Vert =
\left\Vert \left(M^{(u)}_{R+2}-\lambda\right)\mathcal{I}'\widetilde{\varphi}\right\Vert < \nonumber\\
&\left\Vert \left(M^{(u)}_{R+2}-M^{(v_0')}_{R+2}\right)\mathcal{I}'\widetilde{\varphi}\right\Vert +
\left\Vert \left(M^{(v_0')}_{R+2}-\lambda\right)\mathcal{I}'\widetilde{\varphi}\right\Vert <\nonumber\\
&\varepsilon + \left\Vert \left(H'-\lambda\right)\varphi\right\Vert < 2\varepsilon \nonumber
\end{eqnarray}
We can now repeat this argument for a subsequence of vertices along the sequence $\left\{v_{n_j}\right\}$ from Definition~\ref{def:rlimdef},  such that $\text{dist}\left(u_1,u_2\right)>R+2$ for any two vertices $u_1, u_2$ on this subsequence.
As a result we get for any $\varepsilon>0$ an orthonormal sequence of (compactly supported) functions $\left\{\varphi_k\right\}_{k=1}^\infty$ satisfying,
\[
\left\Vert \left(H-\lambda\right)\varphi_k\right\Vert <\varepsilon.
\]
Thus, by taking e.g.\ $\varepsilon_n=\frac{1}{n}$, we can choose an orthonormal sequence of approximate eigenfunctions for $H$.
Hence by Weyl's Criterion for the essential spectrum (Theorem \ref{thm:Weyl}) $\lambda\in\sigma_\text{ess}\left(H\right)$.
\end{proof}

\section{Summary and further directions}

As we show in this work, the problem of characterizing the essential spectrum of Schr\"odinger operators on general graphs is non-trivial and raises various interesting questions. While in the case of $\mathbb{Z}^d$, the characterization is given completely in terms of the associated $\mathcal{R}$-limits, this is not true for general graphs. Thus, a natural problem that arises naturally from this work is that of characterizing graphs for which \eqref{eq:sigmaEssChar} holds (e.g.\ in terms of their geometric properties).

As we show in this work, \eqref{eq:sigmaEssChar} does hold for graphs of uniform sub-exponential growth and for trees with a spherically homogeneous potential. It seems likely that this class could be generalized to other graphs with spherical symmetry (such as those described in \cite{BK} for example). We leave this for future work.

In the course of the proof of \eqref{eq:sigmaEssChar} on graphs we encounter a necessity for a Shnol's type result for graphs. We prove a reverse result of this type, i.e.\ the existence of a generalized eigenfunction of specific growth rate for each point in the spectrum. 
As we describe in Section \ref{sec:PfonRegTrees}, for general Schr\"odinger operators on trees \eqref{eq:sigmaEssLastSimonChar} holds. The problem of removing the closure from the left hand side is still open. We believe that a better understanding of the growth properties of generalized eigenfunctions for operators on trees would be of use in studying this problem.

\chapter[Extending the localization method]{Extending the Last-Simon localization method to general graphs}\label{chapter2}

In this chapter we study possible ``direct" generalizations of the method of \cite{LastSimonEss} to general graphs.
In the first two sections we adapt the argument to general graphs by using trial functions which are supported on balls instead of intervals. This method enables us to obtain positive results on graphs with uniform polynomial growth rate.
Next, in Section~2.3 we study the restrictions of applying this method to more general graphs. Finally, in Section~2.4 we demonstrate, by a counterexample, that in fact the statement is false on general graphs. 
The example given in Section~2.4 is based on Section~4 of \cite{BDE}. 

\section{Graphs of uniform polynomial growth}

Let $G$ be an infinite graph with vertex degree bounded by $d$. 
Given a vertex $v\in G$, recall the notation (identifying the graph $G$ with the set of vertices $V(G)$)
\[
B_r(u)=\left\{v\in G\,|\,\text{dist}(u,v)\leq r\right\},
\]
and define $N_u(r):\mathbb{N}\to\mathbb{N}$ by
\[
N_{u}(r)= \left\vert B_{r}(u)\right\vert.
\]

Given functions $f,g:\mathbb{N}\to\mathbb{R}$ (and similarly for functions defined on the graph) we write $f\lesssim g$ if there exists a constant $c>0$ so that for any  $n\in\mathbb{N}$,
\[
f(n)\leq c g(n).
\]
Consider a function $\gamma:\mathbb{N}\to\mathbb{N}$ and $\alpha>0$, we say that $\gamma$ is of upper (lower) $\alpha$-polynomial growth if $\gamma\lesssim n^\alpha$ ($\gamma\gtrsim n^\alpha$).
In short we will say that $\gamma$ is of upper (lower) polynomial growth.

We say that $G$ is of (\{$\alpha,\beta$\}-)polynomial growth if for some $\alpha, \beta>0$, and a vertex $v_0\in G$,
\[
r^\alpha \lesssim N_{v_0}(r) \lesssim r^\beta.
\]
\begin{remark*}
If $\text{dist}(u,v)=\rho$ then $N_u(r)\leq N_v(r+\rho)$. Thus, if $G$ is of \{$\alpha,\beta$\}-polynomial growth, then for any $u\in G$ also
\[
r^\alpha \lesssim N_{u}(r) \lesssim r^\beta.
\]
\end{remark*}
We say that $G$ is of {\bf uniform} \{$\alpha,\beta$\}-polynomial growth if $r^\alpha \lesssim N_{u}(r) \lesssim r^\beta$ for any $u\in G$ with the same constants, i.e.\ there exist $c,c'>0$ such that for any $u\in G$, $r\in\mathbb{N}$,
\[
c r^\alpha \leq N_u(r)\leq c'r^\beta.
\]

Let $H$ be a bounded Schr\"odinger operator on $G$ (see \eqref{eq:SchrOp}). As in  Chapter~1 we denote by $\mathcal{R}(H)$ the set of $\mathcal{R}$-limits of $H$.

\begin{theorem} \label{thm:Thm4}
Assume $G$ is a graph of \{$\alpha,\beta$\}-uniform polynomial growth so that $\beta-\alpha<1$, and assume $H$ is a bounded Schr\"odinger operator on $G$, then
\begin{equation} \label{eq:mainsta}
\sigma_{\text ess}(H)=\overline{\bigcup_{L\in\mathcal{R}(H)}\sigma(L)}.
\end{equation}
\end{theorem}
The inclusion of $\sigma(L)$ in $\sigma_{\text ess}(H)$ is proven generally in Chapter 1. Here we will prove the opposite inclusion.
\begin{remark}
In Chapter~\ref{chapter4}, we give a stronger version of this characterization using bounded generalized  eigenfunctions of $\mathcal{R}$-limits of $H$.
As we show there the set $\cup\sigma(L)$ is actually already closed, and \eqref{eq:mainsta} holds also without the assumption of lower polynomial growth and the condition $\alpha>\beta-1$. The purpose of the current part is to investigate the possibility to generalize the Last-Simon method to general graphs.
\end{remark}
Requiring simply polynomial growth (not uniform) of $G$ is not enough!
We show in  Section~\ref{sec:counterexample} an example of a graph of non-uniform polynomial growth on which $\sigma_{\text ess}(H)\backslash \overline{\cup \sigma(L)}$ is non-empty. Indeed,
\begin{theorem}
\label{thm:counterExm} There exists a graph $G$ of polynomial growth so that the adjacency operator on $G$, $A_G$, satisfies
\[
\sigma_{\text{ess}}\left(A_G\right)\Bigg\backslash\overline{\bigcup_{L\in\mathcal{R}(A_G)}\sigma\left(L\right)}
\]
is nonempty.
\end{theorem}

\begin{remark}
By a simple adaptation everything holds also for Jacobi operators on the graph. 
For simplicity we treat Schr\"odinger operators.
\end{remark}

\section[Proof of Theorem~2.1]{Proof of Theorem~\ref{thm:Thm4}}\label{sec:pfofThm4}

First we cite from Last-Simon \cite{LastSimonEss} a proposition which will be useful in the proof:
\begin{prop}
[{\cite[Theorem 2.2]{LastSimonEss}}] \label{prop:LS2.2}
Let $\mathcal{H}$ be a separable Hilbert space, and $A$ a bounded selfadjoint operator on $\mathcal{H}$. Assume $\left\{j_\alpha\right\}_{\alpha\in S}$ is a set of bounded selfadjoint operators $($indexed by a discrete set $S)$, which is a partition of unity, namely, $\sum_\alpha j^2_\alpha =1$. Let $\varphi\in\mathcal{H}$. Then there exists $\alpha\in S$ such that $j_\alpha\varphi\neq0$, and
\[
\left\Vert A j_\alpha\varphi\right\Vert^2\leq \left\{2\left(\frac{\left\Vert A\varphi\right\Vert}{\Vert \varphi\Vert}\right)^2+\Vert C\Vert\right\}\left\Vert j_\alpha\varphi\right\Vert^2,
\]
where
\[
C=-\Sigma_\alpha 2[A,j_\alpha]^2.
\]
\end{prop}

\begin{remark}
Last-Simon do not require the boundedness of $A$. Since in our applications $A$ is bounded we formulate the theorem in this, slightly simpler, case.
\end{remark}

The proof proceeds by using this proposition in order to uniformly truncate a sequence of approximate eigenfunctions for $\lambda\in\sigma_{\text ess}(H)$. We will use a specific partition of unity which we describe next.

\begin{proof}[Proof of Theorem~\ref{thm:Thm4}]

Define for any $u\in G, r\in \mathbb{N}$ the ``pyramid" function
\[\chi_{u,r}(v)=\begin{cases}
\frac{r-k}{r}\,\,\,\,\,\,\,\,\,\,\,\,\,\, & k=\text{dist}(v,u)\leq r\\
0 & \text{dist}(v,u)>r.
\end{cases}\]
Let $c,c'>0$ be such that $c r^\alpha \leq N_u(r)\leq c'r^\beta$ for any $u\in G$. Let \[c_{r}^{2}(u)=\sum_v\chi^2_{u,r}(v)\] then
\[
c_r^2(u)=1+\sum_{k=1}^{r-1} \left(\frac{r-k}{r}\right)^2\left(N_u(k)-N_u(k-1)\right).
\]
For any $1\leq k\leq r/2$ we have that $\nicefrac{(r-k)}{r}\geq\nicefrac{1}{2}$, and thus
\[
c_r^2(u)\geq 1+\frac{1}{4}\sum_{k=1}^{r/2}\left(N_u(k)-N_u(k-1)\right)= \frac{1}{4} N_u\left(\nicefrac{r}{2}\right).
\]
On the other hand,
\[
c_r^2(u)\leq  1+\sum_{k=1}^{r-1}\left(N_u(k)-N_u(k-1)\right)=N_u(r-1).
\]
Asymptotically we get that there exist $C_1,C_2>0$ such that for every $u\in G$ and $r\in\mathbb{N}$
\[
C_1 r^\alpha\leq c_{r}^{2}(u)\leq C_2 r^\beta.
\]

Additionally, for any $v\in G$  define
\[\varphi_{u,r}^2(v)=c_{r}^{-2}(u)\chi_{u,r}^2(v),\]
and (notice that this time the sum is over  the lower index of $\varphi$)
\[\eta_r^2(v)=\sum_u \varphi_{u,r}^2(v)=\sum_u \frac{\chi_{u,r}^2(v)}{c_r^2(u)}.\]
Then
\[ \eta_r^2(v)\leq C_1^{-1} r^{-\alpha}\left(1+\sum_{k=1}^{r-1} \left(\frac{r-k}{r}\right)^2 \left(N_v(k)-N_v(k-1)\right)\right)\leq \frac{C_2}{C_1} r^{\beta-\alpha}
\]
and similarly
\[
\eta_r^2(v)\geq \frac{C_1}{C_2} r^{\alpha-\beta}.
\]
We now define the function
\[\psi_{u,r}^2(v)=\eta^{-2}_r(v)c_{r}^{-2}(u)\chi_{u,r}^2(v)\]
which satisfies $\sum_{u}\psi_{u,r}^{2}(v)=1$ for any $v\in G$.

For the rest of the proof we denote by ($*$) the following condition on a given $u,v,w\in G$ and $r\in\mathbb{N}$:
\[
(*)\  \max\left(\text{dist}(u,v),\text{dist}(u,w)\right)\leq r \text{ and } \text{dist}(u,v)\neq\text{dist}(u,w).
\]
Notice that, given $u\in G$ and $w\sim v\in G$, then
\[
\left|\chi_{u,r}(v)-\chi_{u,r}(w)\right|=\begin{cases}\frac{1}{r} & (*)\ \text{holds}\\ 0 & \text{otherwise}.
\end{cases}
\]
Hence, uniformly in $u,v,w\in G$,
\begin{equation} \label{eq:bnd1}
\left|\left\langle \delta_{v},\left[H,\psi_{u,r}\right]\delta_{w}\right\rangle \right|\lesssim\begin{cases}
\frac{1}{\eta_r\cdot c_{r}\cdot r}\,\,\,\, & \text{if }w\sim v\text{ and }(*)\text{ holds}\\
0 & \text{otherwise}.
\end{cases}
\end{equation}
Define $C^{(r)}=-\sum_{u}2[H,\psi_{u,r}]^{2}$. We have for any $v,w\in G$, the following expression for the matrix elements of $C^{(r)}$
\[
C^{(r)}_{v,w}=\langle\delta_v, C^{(r)} \delta_w\rangle= -2\sum_{x,u\in G}\left\langle \delta_{v},\left[H,\psi_{u,r}\right]\delta_{x}\right\rangle \left\langle \delta_{x},\left[H,\psi_{u,r}\right]\delta_{w}\right\rangle.
\]
Each term in the sum is either bounded by $\left(\eta_r \cdot c_r\cdot r\right)^{-2}$ if ($*$) is satisfied or is zero otherwise. 
Accordingly, for a term to be nonzero it is necessary that $u\in B_r(v)$ and $x\sim v$. The number of such terms is bounded by $c'r^{\beta}\cdot d$. 
Consequently, the matrix elements of $C^{(r)}$ satisfy the following upper bound (uniformly in $v,w\in G$):
\[
\left|C_{v,w}^{(r)}\right|\lesssim 2\cdot d\cdot c' r^{\beta}(\eta_r\cdot c_{r}\cdot r)^{-2}\lesssim\frac{r^\beta}{r^{\alpha-\beta}\cdot r^{\alpha}\cdot r^2}\lesssim\frac{1}{r^{2-2(\beta-\alpha)}}.
\]
Thus, for some constant $K>0$,
\[\left\Vert C^{(r)}\right\Vert \leq K\cdot r^{-2+2(\beta-\alpha)}.\]

Fix an $\nicefrac{1}{2}>\varepsilon>0$. Since $\beta-\alpha<1$, we can now fix an $r>0$ such that $\left\Vert C^{(r)}\right\Vert<\varepsilon^{2}$.

Fix a vertex $v_0\in G$. Assume $\lambda\in\sigma_{ess}(H)$. By Weyl's Criterion (Theorem \ref{thm:Weyl}) there exists a sequence of unit vectors $\left\{ \phi^{(m)}\right\} _{m=1}^{\infty}\subset\ell^2(G)$, such that $\left\Vert \left(H-\lambda\right)\phi^{(m)}\right\Vert \to0$ and $\phi^{(m)}\overset{w\,\,}{\to}0$ (where $\overset{w\,\,}\to$ indicates weak convergence).
Consequently,
\[\sum_{v\in B_{r}(v_0)}\left|\phi^{(m)}(v)\right|^{2}\underset{m\to\infty}{\to}0,\]
for any fixed $r\in\mathbb{N}$.
Thus, by restricting $\left\{\phi^{(m)}\right\}_{m=1}^{\infty}$ to a subsequence if necessary (which we denote the same), 
\[
\varphi^{(m)}(v)=
\begin{cases}
\phi^{(m)}(v) & v\in G\backslash B_{m}(v_0) \\
0 & \text{otherwise}
\end{cases}
\]
satisfies both $\left\Vert \varphi^{(m)}\right\Vert >1-\varepsilon$ and
\begin{eqnarray*}
\left\Vert \left(H-\lambda\right)\varphi^{(m)}\right\Vert 	\leq	\left\Vert \left(H-\lambda\right)\phi^{(m)}\right\Vert +\left\Vert \left(H-\lambda\right)\left(\varphi^{(m)}-\phi^{(m)}\right)\right\Vert \leq \\
\varepsilon/2	+	\left\Vert H-\lambda\right\Vert \left\Vert \varphi^{(m)}-\phi^{(m)}\right\Vert <\varepsilon/2+C\cdot\varepsilon/2C=\varepsilon.
\end{eqnarray*}

Next we apply Proposition~\ref{prop:LS2.2}, with the operator $A=H-\lambda$, the set $\left\{j_\alpha\right\}_\alpha=\left\{\psi_{u,r-1}\right\}_{u\in G}$ (treated as multiplication operators) and each time with a function $\varphi=\varphi^{(m)}$ from the sequence of approximate eigenfunctions found above. Thus, for any $m\in\mathbb{N}$ there exists $\psi_{m}=\psi_{u_{m},r-1}$ such that $\psi_{m}\varphi^{(m)} \neq0$ and (for $\varepsilon\leq \nicefrac{1}{2}$)
\begin{eqnarray*}
\left\Vert \left(H-\lambda\right)\psi_{m}\varphi^{(m)}\right\Vert^2 \leq& \left(2 \frac{\left\Vert \left(H-\lambda\right)\varphi^{(m)}\right\Vert^2}{\left\Vert\varphi^{(m)}\right\Vert^2} + \left\Vert C^{(r)}\right\Vert \right) \left\Vert \psi_{m}\varphi^{(m)}\right\Vert^2\\
\leq & \left( 2\left(\frac{\varepsilon}{1-\varepsilon}\right)^2 + \varepsilon^2 \right)\left\Vert \psi_{m}\varphi^{(m)}\right\Vert^2 \\
\leq & 9\varepsilon^2  \left\Vert \psi_{m}\varphi^{(m)}\right\Vert^2.
\end{eqnarray*}
Note that $\psi_{m}\varphi^{(m)}$ is supported in $B_{r-1}(u_{m})\subseteq G\backslash B_{m-r}(v_0)$, which moves out to infinity.
By Lemma~\ref{lemma:rlimexists} there exists a sequence of indices $\left\{ m_{j}\right\} _{j=1}^{\infty}$ of $\left\{ u_{m}\right\} _{m=1}^{\infty}$ on which $H$ approaches an $\mathcal{R}$-limit $\left\{H',G',v_0'\right\}$. Denote by $M_r^{(n)},M'_r$ the matrices corresponding to $H|_{B_{r}\left(u_{m_n}\right)}$ and $H'|_{B_r\left(v_0'\right)}$.
Further denote by $\mathcal{I}_r'$, and $\mathcal{I}_r^{(n)}$ the corresponding isomorphisms (acting on $\ell^2\left(B_r\left(v_0'\right)\right)$ and $\ell^2\left(B_r\left(u_{m_n}\right)\right)$). Then there exists $N\in\mathbb{N}$, so that for any $n>N$
\[
\left\Vert M_r^{(n)}-M'_r\right\Vert <\varepsilon.
\]
Define $\zeta_{n}=\mathcal{I}_r'^{-1}\mathcal{I}_r^{(n)}\psi_{m_{n}}\varphi_{m_{n}}$ (completed with zeros on $G'\backslash B_r\left(v_0'\right)$). Then for any $n>N$
\begin{eqnarray*}
&\left\Vert \left(H'-\lambda\right)\zeta_{n}\right\Vert 	=	\left\Vert \left(H'|_{B_r(v_0')}-\lambda\right)\zeta_{n}\right\Vert =\\
&=\left\Vert \left(M'_r-\lambda\right)\mathcal{I}_r^{(n)}\psi_{m_{n}}\varphi_{m_{n}}\right\Vert\leq\\
&\left\Vert \left(M_r^{(n)}-M'_r\right)\mathcal{I}_r^{(n)}\psi_{m_{n}}\varphi_{m_{n}}\right\Vert 	+	\left\Vert \left(M_r^{(n)}-\lambda\right)\mathcal{I}_r^{(n)}\psi_{m_{n}}\varphi_{m_{n}}\right\Vert < \\
&\varepsilon\left\Vert \mathcal{I}_r^{(n)}\psi_{m_{n}}\varphi_{m_{n}}\right\Vert + \left\Vert \left(H|_{B_r\left(u_{m_n}\right)}-\lambda\right)\psi_{m_{n}}\varphi_{m_{n}}\right\Vert=\\
&\varepsilon\left\Vert\zeta_n \right\Vert + \left\Vert \left(H-\lambda\right)\psi_{m_{n}}\varphi_{m_{n}}\right\Vert < 4\varepsilon\left\Vert \zeta_{n}\right\Vert.
\end{eqnarray*}
Thus
\[
\lim_{n\to\infty}\frac{\|\left(H'-\lambda\right)\zeta_{n}\|}{\|\zeta_{n}\|}\leq 4\varepsilon,
\]
which implies (by a variant of Weyl's Criterion, Theorem \ref{thm:Weyl}, see e.g.\ \cite[Proposition 7.2.2]{SimonSz}) that $\text{dist}\left(\lambda,\sigma\left(H'\right)\right)\leq 4\varepsilon$.
Since $\varepsilon$ is arbitrary, we can conclude
\[\lambda\in\overline{\bigcup_{H'\in\mathcal{R}(H)}\sigma\left(H'\right)}.\]
\end{proof}

\section{On extending the proof to graphs of exponential growth}

In order to avoid cumbersome calculations we shall assume in this section that the graph $G$ is uniform in the sense that $N_u(k)=N_v(k)\equiv N(k)$ for every pair of vertices $u,v\in G$ and every $k\in\mathbb{N}$.
Examples of such graphs include $\mathbb{Z}^n$ and a $d$-regular tree $T_d$.

Let $\left\{f_r:\mathbb{N}\cup\{0\}\to[0,1]\right\}_{r\in\mathbb{N}}$ be a sequence of functions, such that each $f_r(k)$ is supported (and is non-zero) on $[0,r]$. We shall assume (without loss of generality) that $f_r(0)=1$ and that $f_r(k)$ is monotonically decreasing with $k$. Given such a sequence we can define trial functions on the graph
\begin{equation}\label{eq:newchi}
\chi_{u,r}(v)=f_r \left(\text{dist}(v,u)\right).
\end{equation}

We choose a root $v_0$ for the graph $G$, and denote
\[
\mathcal{S}(k)=\mathcal{S}_{v_0}(k)=\left\{v\in G\,|\,\text{dist}\left(v_0,v\right)=k\right\},
\]
\[
S(k)=\left|\mathcal{S}(k)\right|=N(k)-N(k-1).
\]
We shall prove the following:
\begin{theorem} \label{thm:Thm5}
Let $H$ be a bounded Schr\"odinger operator on a uniform infinite graph $G$. Assume there exists $\alpha>0$ so that
\begin{equation}
\sup_{r\in\mathbb{N}} \,\,\,\,\, \frac{N(r)}{S(r)}\leq \alpha.
\end{equation}
Let $\left\{\chi_{u,r}\right\}_{r\in\mathbb{N},\, u\in G}$ be a set of trial functions satisfying \eqref{eq:newchi} and let $($similarly to  Section \ref{sec:pfofThm4}$)$
\[
C^{(r)}=-\sum_{u}2[H,\chi_{u,r}]^{2}.
\]
Then $\forall r\in\mathbb{N}$ and $\forall v\in G$, the diagonal matrix element,  $C_{v,v}^{(r)}$, satisfies
\begin{equation} \label{eq:Cbound}
\left|C_{v,v}^{(r)}\right|>{1\over 2\alpha}.
\end{equation}
\end{theorem}

\begin{remark}

Functions of the type $\chi_{u,r}$ defined in \eqref{eq:newchi} are natural candidates for trial functions on graphs.
Thus, one might try to follow the argument of the proof of Theorem~\ref{thm:Thm4} using such functions and the operator $C^{(r)}$ as defined above.
It is then necessary to obtain an upper bound on $\left\Vert C^{(r)} \right\Vert$ which is asymptotically vanishing as $r\to\infty$,  in order to implement Proposition~\ref{prop:LS2.2} and find approximate eigenfunctions. 
Under the conditions of Theorem~\ref{thm:Thm5}, the diagonal matrix elements of $C^{(r)}$ are uniformly bounded from below. Thus, the bound \eqref{eq:Cbound} shows that such an attempt is doomed to fail for any non polynomially growing graph.
\end{remark}

\begin{remark} This section is complemented by Section~\ref{sec:counterexample}, in which we present an example for a graph of non-uniform polynomial growth on which the characterization fails. As opposed to that counterexample, we focus here on graphs with a strong regularity condition, for which one might have hoped that the argument could still hold.
\end{remark}

The proof relies on the following weighted discrete Hardy Inequality, originally proven by Leindler:
\begin{prop}[\cite{Leindler}, (1)] \label{prop:leindler}
Let $N\in\mathbb{N}$, $a_n\geq0,\, \lambda_n\geq0$ $(n=1,2,\ldots N)$, $p\geq1$, then
\begin{equation} \label{eq:whardy}
\sum_{n=1}^N \lambda_n \left(\sum_{k=1}^n a_k\right)^p \leq p^p \sum_{n=1}^N \lambda_n^{1-p}\left(\sum_{k=n}^N\lambda_k\right)^p a_n^p.
\end{equation}
\end{prop}

\begin{proof}[Proof of Theorem~\ref{thm:Thm5}]

For any $r\in\mathbb{N}$ let $\left\{\chi_{u,r}\right\}_{u\in G}$ be a sequence of trial functions as in \eqref{eq:newchi}.
Recall the definitions,
\begin{eqnarray*}
c_{r}^{2}(u)=&\sum_v\chi^2_{u,r}(v),\\
\varphi_{u,r}^2(v)=&c_{r}^{-2}(u)\chi_{u,r}^2(v).
\end{eqnarray*}
Notice that, since the graph is uniform,
\[ \sum_u \varphi_{u,r}^2(v)\equiv 1, \]
and thus we can skip the normalization by $\eta_r$. We have
\[
c_r^2(u)\equiv c_r^2=\sum_{k=1}^{r} f_r^2(k) S(k).
\]
Also, in this case,
\begin{equation*}
\left|\left\langle \delta_{v},\left[H,\psi_{u,r}\right]\delta_{w}\right\rangle \right|=\begin{cases}
\frac{1}{ c_{r}}\cdot\left(f_r(k)-f_r(k-1)\right)\,\,\,\, & \text{if }w\sim v\text{ and }(*)\text{ holds}\\
0 & \text{otherwise},
\end{cases}
\end{equation*}
with,
\[
(*)\  k=\max\left(\text{dist}(u,v),\text{dist}(u,w)\right)\leq r \text{ and } \text{dist}(u,v)\neq\text{dist}(u,w).
\]
Thus,
\begin{eqnarray*}
|C_{v,u}^{(r)}|=|\langle\delta_v, C^{(r)} \delta_u\rangle|=2\left|\sum_{y,z}\langle\delta_v, [\psi_{y,r},H] \delta_z\rangle\langle\delta_z, [\psi_{y,r},H] \delta_u\rangle\right|.
\end{eqnarray*}
The sum on $z$ contributes $\left|\left\{z\,|\,z\sim v\, \wedge\, z\sim u\right\}\right|$ terms, which in total is between $1$ and $d$. The sum over $y$ gives a non-zero contribution if either $\text{dist}(y,z)\leq r$ or both $\text{dist}(y,v)\leq r$ and $\text{dist}(y,u)\leq r$. Thus we get the following bound for appropriate $u$ and $v$ (e.g.\ $u=v$),
\begin{eqnarray*}
|C_{v,u}^{(r)}| \geq 2\sum_{k=1}^{r} S(k)\, \left|\frac{1}{ c_{r}}\cdot\left(f_r(k)-f_r(k-1)\right)\right|^2.
\end{eqnarray*}
Define $g_r(k)=f_r(r-k)$, and $x_r(k)=g_r(k)-g_r(k-1)$. Since by assumption $f_r$ is monotonically decreasing $x_r(k)\geq0, \forall k$, and we have $g_r(k)=\sum_{j=1}^k x_r(j)$. By a change of indices we get
\begin{eqnarray*}
|C_{v,u}^{(r)}|\geq 2\frac{\sum_{n=0}^{r-1} \, S(r-n) \left(x_r(k)\right)^2}{\sum_{n=0}^{r-1} S(r-n) \left( \sum_{k=0}^n x_r(k) \right)^2 }.
\end{eqnarray*}
Next we apply Proposition~\ref{prop:leindler}, with
\begin{eqnarray*}
a_n = & x_r (n) \\
\lambda_n = & S(r-n) \\
p=&2 \\
N=&r-1.
\end{eqnarray*}
Then, for any $r\in\mathbb{N}$, \eqref{eq:whardy} translates to:
\[
\sum_{n=0}^{r-1} S(r-n) \left(g_r(n)\right)^2 \leq 4 \sum_{n=1}^{r-1} S(r-n) \left(\sum_{k=n}^{r-1}\frac{S(r-k)}{S(r-n)}\right)^2 \left(x_r(n)\right)^2.
\]
By the assumption there exists $\alpha>0$ so that $\forall r\in\mathbb{N},\,\, N(r)\leq \alpha\cdot S(r)$. Thus,
\[
\sum_{k=n}^{r-1}\frac{S(r-k)}{S(r-n)}=\frac{\sum_{j=1}^{r-n} S(j)}{S(r-n)}= \frac{N(r-n)}{S(r-n)}\leq \alpha,
\]
and finally, $\forall r\in\mathbb{N}$,
\[
|C_{v,u}^{(r)}|\geq \frac{1}{2\alpha}>0.
\]

\end{proof}

\begin{remark}
On the $d$-regular tree
\begin{eqnarray*}
S_{T_d}(r)=d\cdot(d-1)^{r-1},\\
N_{T_d}(r)=d\cdot\frac{(d-1)^r-1}{d-2},\\
\lim_{r\to\infty}\frac{N_{T_d}(r)}{S_{T_d}(r)}=\frac{d-1}{d-2}.
\end{eqnarray*}
Thus from Theorem~\ref{thm:Thm5} we conclude that the argument of Theorem~\ref{thm:Thm4} does not hold.
We expect a similar behaviour on every graph for which the growth rate is beyond polynomial.
\end{remark}
\begin{remark}
On the other hand on $\mathbb{Z}^n$,
\begin{eqnarray*}
\frac{N_{\mathbb{Z}^n}(r)}{S_{\mathbb{Z}^n}(r)}\sim r,
\end{eqnarray*}
which as expected is not bounded.
\end{remark}

\section[A counterexample]{A counterexample with non-uniform polynomial growth}{} \label{sec:counterexample}
We shall prove in this section Theorem~\ref{thm:counterExm} which we repeat here for completeness. 

\begin{namedtheorem}[Theorem~\ref{thm:counterExm}]
\label{thm:repCounterExm}There exists a graph $G$ of polynomial growth so that the adjacency operator on $G$, $A_G$, satisfies
\[
\sigma_{\text{ess}}\left(A_G\right)\Bigg\backslash\overline{\bigcup_{L\in\mathcal{R}(A_G)}\sigma\left(L\right)}
\]
is nonempty.
\end{namedtheorem}

\begin{proof}
First, recall that the girth of a graph $\mathcal{G}$ is
\[
\text{girth}\left(\mathcal{G}\right)\equiv\min\left\{\text{length}(l)\,|\,l\text{ is a cycle in }\mathcal{G}\right\}.
\]
Fix $d>2$ and let $\left\{ G_{n_{i}},u^{(1)}_i,u^{(2)}_i \right\} _{i=1}^{\infty}$
be a sequence of $d$-regular graphs on $n_{i}$ vertices, each with two marked vertices $u^{(1)}_i,u^{(2)}_i \in G_{n_i}$, where, $\left\{n_{i}\right\}_{i=1}^{\infty}\subset\mathbb{N}$, $n_{i}\to\infty$ monotonically, and so that
\begin{align*}
&\text{girth}\left(G_{n_{i}}\right)\overset{i}{\longrightarrow}\infty,\\
&\text{dist}\left(u^{(1)}_i,u^{(2)}_i\right)\overset{i}{\longrightarrow}\infty.
\end{align*}
\begin{figure}[ht!]
\centering
\includegraphics[width=135mm]{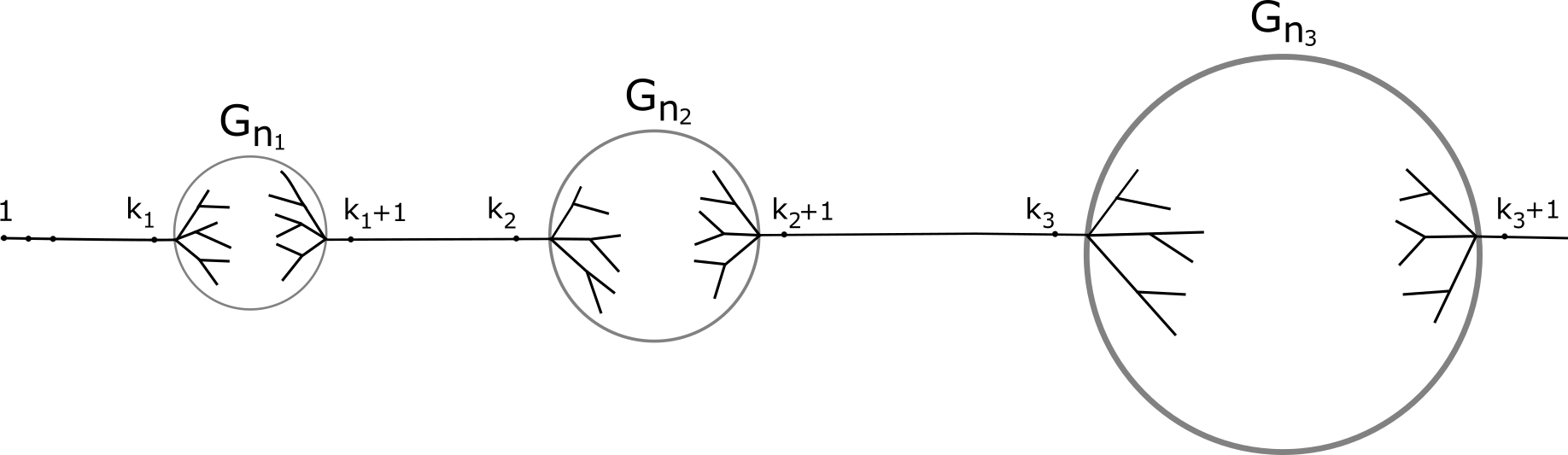}\\
\caption{The construction of the graph $G$ for the counterexample. \label{figure1}}
\end{figure}
By, e.g., \cite{LPS} such a sequence exists for $d=p+1$ for any prime $p\neq1$ satisfying $p\equiv1\left(\text{mod}\,4\right)$.
Additionally, let $\left\{ k_{i}\right\} _{i=1}^{\infty}\subset\mathbb{N}$, be an increasing sequence so that $k_{i+1}-k_{i}\to\infty$. We construct $G$ by `replacing' the edge $(k_i, k_i+1)$ in $\mathbb{N}$ by the graph $G_{n_{i}}$. This is done by cutting $(k_i, k_i+1)$ and attaching $u^{(1)}_i$ to $k_i$ and $u^{(2)}_i$ to $k_{i+1}$ (see Figure~\ref{figure1}). Formally
\begin{eqnarray*}
V(G)=&\mathbb{N}\cup \bigcup_{i=1}^\infty V\left(G_{n_i}\right),\\
E(G)=&\left(\bigcup_{i=1}^\infty E\left(G_{n_i}\right)\right) \cup \left(E\left(\mathbb{N}\right)\Big\backslash \cup_{i=1}^\infty \left\{(k_i,k_{i+1})\right\}\right)\cup \\
&\left(\bigcup_{i=1}^\infty \left\{(k_i,u^{(1)}_i)\right\}\right)\cup \left(\bigcup_{i=1}^\infty \left\{(u^{(2)}_i,k_{i+1})\right\}\right).
\end{eqnarray*}

In order to obtain polynomial growth of the graph we can choose for example $k_i=\sum_{j=1}^i n_j$. Consequently the growth of the graph, e.g.\ around $1\in\mathbb{N}$, satisfies $\forall k \in \mathbb{N}$,
\[
k\leq N_1(k)\leq 2k.
\]

Let $H=A_G$ on $G$, i.e.\ the potential is $Q(v)=\text{deg}\left(v\right)$. For each graph
$G_{n_i}$ the constant function $\varphi\left(v\right)=\frac{1}{\sqrt{n_i}}$ is an eigenfunction of $A_{G_i}$ with eigenvalue $\lambda=d$. Define
\[
\varphi_{i}\left(v\right)=\begin{cases}
\sfrac{1}{\sqrt{n_{i}}}\,\,\,\,\, & \text{if }v\in G_{n_{i}}\\
0 & \text{otherwise.}
\end{cases}
\]
Then summing over the boundary terms, we have for any $i\in\mathbb{N}$
\[
\left\Vert H\varphi_{i}-\lambda\varphi_{i}\right\Vert^{2}=\sfrac{2}{n_{i}}.
\]
Thus
\[
\left\Vert H\varphi_{i}-\lambda\varphi_{i}\right\Vert ^{2}\overset{i\to\infty}{\,\longrightarrow\,\,\,}0.
\]
Additionally, for any $i\neq j$ the functions $\varphi_{i}$ and
$\varphi_{j}$ are orthogonal. Thus $\left\{ \varphi_{i}\right\} _{i=1}^{\infty}$
is an orthonormal sequence of approximate eigenfunctions of $H$
for the value $\lambda=d$, and thus $d\in\sigma_{\text{ess}}\left(H\right).$
We claim that $d\notin\overline{\bigcup\sigma\left(L\right)}$. Indeed, it is easy to see that
the only $\mathcal{R}$-limits of $G$ are the following three objects:
\begin{enumerate}
\item The adjacency operator on the full line $\mathbb{Z}$, appearing when the limit is taken along a subsequence $\left\{ v_{n_j}\right\} _{j=1}^{\infty}$ of points (only) on $\mathbb{N}$, of increasing distance from the sequence $\left\{ k_{i}\right\} _{i=1}^{\infty}$, i.e.\
\[
\inf_{i}\left(\text{dist}\left(v_{n_j},k_{i}\right)\right)\underset{j}{\longrightarrow}\infty.
\]

\item  The adjacency operator on a $d$-regular tree $T_{d}$, appearing when $\left\{ v_{n_j}\right\} _{j=1}^{\infty}$
includes (only) points on $\left\{ G_{n_{i}}\right\}_{i=1}^{\infty} $, of increasing distance from both the sequences of vertices $\left\{u^{(1)}_i \right\}_{i=1}^{\infty}$
and $\left\{u^{(2)}_i\right\}_{i=1}^{\infty}$, i.e.\
\[
\inf_{i}\text{dist}\left(v_{n_j},u^{(\ell)}_i\right)\underset{j}{\longrightarrow}\infty
\]
for both $\ell=1,2$.
Since the girth of $G_{n_{i}}$ grows to infinity, we conclude that
for any $R>0$ the reduced graph of radius $R$ around $v_{n_j}$ will
be a tree for $j$ large enough.
\item  The adjacency operator on the tree, $\widetilde{T}=\widetilde{T}_d$, which is a half-line connected to a $d$-regular tree at the point $1\in\mathbb{N}$,
appearing when $\left\{v_{n_j}\right\}_{j=1}^{\infty}$ are points
on $\mathbb{N}$ of fixed distance from $\left\{k_{i}\right\}_{i=1}^{\infty}$,
or when $\left\{ v_{n_j}\right\} _{j=1}^{\infty}$ are points from $\left\{ G_{n_{i}}\right\}_{i=1}^{\infty}$
and are of a fixed distance from either $\left\{u^{(1)}_i\right\}_{i=1}^{\infty}$
or $\left\{u^{(2)}_i \right\}_{i=1}^{\infty}$.
\end{enumerate}
The corresponding spectra for the first two operators are:
\begin{enumerate}
\item $\sigma\left(A_{\mathbb{Z}}\right)=\left[-2,2\right]$.
\item $\sigma\left(A_{T_{d}}\right)=\left[-2\sqrt{d-1},2\sqrt{d-1}\right]$.
\end{enumerate}
Both of them do not contain the point $\lambda=d$. The following lemma
completes the argument.
\begin{lem}\label{lem:Lemma1}
$d\notin\sigma\left(A_{\widetilde{T}}\right)$
\end{lem}
Before proving the lemma we conclude that this example satisfies $d\notin\overline{\bigcup_{L}\sigma\left(L\right)}$,
while $d\in\sigma_{ess}(H)$. This completes the proof of Theorem~\ref{thm:counterExm}.
\end{proof}

\begin{proof}[Proof of Lemma~\ref{lem:Lemma1}]
The tree $\widetilde{T}=\widetilde{T}_d$ is composed of a $d$-regular tree $T=T_{d}$
and a line, such that the point $1\in\mathbb{N}$ is connected to a point $0\in T$. Thus $A_{\widetilde{T}}$ is a finite rank perturbation of $A_T\oplus A_\mathbb{N}$, and so
\[\sigma_\text{ess}\left(A_{\widetilde{T}}\right)=\sigma_\text{ess}\left(A_T\right)\cup\sigma_\text{ess}\left(A_\mathbb{N}\right)=\sigma_\text{ess}\left(A_T\right)=[-2\sqrt{d-1},2\sqrt{d-1}].
\]
Therefore $d\notin\sigma_\text{ess}\left(A_{\widetilde{T}}\right)$. We want to exclude the possibility that ${d \in \sigma_{\text disc}\left(A_{\widetilde{T}} \right)=
\sigma\left(A_{\widetilde{T}} \right) \setminus \sigma_{\textrm{ess}}\left(A_{\widetilde{T}} \right)}$.

Using Dirac's bra-ket notation, define $A_{0}=A_{\widetilde{T}}-\left|\delta_{0}\right\rangle \left\langle \delta_{1}\right|-\left|\delta_{1}\right\rangle \left\langle
\delta_{0}\right|$
and $R\left(z\right)=\left(A_{\widetilde{T}}-z\right)^{-1}$, $R_{0}\left(z\right)=\left(A_{0}-z\right)^{-1}$.
Recall the resolvent identity (we omit the dependence on $z$),
\begin{equation} \label{eq:resid}
R_{0}-R=R_{0}\left(A_{\widetilde{T}}-A_{0}\right)R=R_{0}\left(\left|\delta_{0}\right\rangle \left\langle
\delta_{1}\right|+\left|\delta_{1}\right\rangle \left\langle \delta_{0}\right|\right)R.
\end{equation}
Multiplying by $\delta_{0}$ on both sides we have
\[
m_{T}(z)-m(z)=m_{T}(z)\left\langle \delta_{1}, R(z) \delta_{0}\right\rangle,
\]
where
\[
m_{T}\left(z\right)=\left\langle \delta_{0},R_{0}\left(z\right)\delta_{0}\right\rangle
\]
\[
m\left(z\right)=\left\langle \delta_{0},R\left(z\right)\delta_{0}\right\rangle.
\]
Additionally, by multiplying the identity \eqref{eq:resid} by $\delta_{1}$ on the left and by $\delta_{0}$ on the right we have
\[
0-\left\langle \delta_{1}\right|R(z)\left|\delta_{0}\right\rangle =m_{\mathbb{N}}(z)m(z)
\]
with
\[
m_{\mathbb{N}}\left(z\right)=\left\langle \delta_{1},R_{0}\left(z\right)\delta_{1}\right\rangle.
\]
Combining we get,
\[
m_{T}(z)-m(z)=m_{T}(z)\left(-m_{\mathbb{N}}(z)m(z)\right),
\]
which implies
\begin{equation} \label{eq:appendix1}
m(z)=\frac{m_{T}(z)}{1-m_{T}(z)m_{\mathbb{N}}(z)}.
\end{equation}

Now, if $\lambda\in\sigma_\text{disc}\left(A_{\tilde{T}}\right)$, then
$\lim_{\varepsilon\to0}\text{Im}\left(m\left(\lambda+i\varepsilon\right)\right)\neq0$ (see Theorem \ref{thm:spectralMeasureProperties}).
We will consider this expression for $\lambda=d$. It follows from \eqref{eq:appendix1} that
\[
\text{Im}\left(m\right)=\frac{\text{Im}\left(m_{T}\left(1+\overline{m_{T}m_{\mathbb{N}}}\right)\right)}{\left|1-m_{T}m_{\mathbb{N}}\right|^{2}}=\frac{\text{Im}\left(m_{T}\right)-\left|m_{T}\right|^{2}\text{Im}\left(m_{\mathbb{N}}\right)}{\left|1-m_{T}m_{\mathbb{N}}\right|^{2}}.
\]

It is known (see, e.g., \cite{SimonSz}) that
\begin{equation}\label{eq:mN}
m_{\mathbb{N}}\left(z\right)=\frac{-z+\sqrt{z^{2}-4}}{2}
\end{equation}
\[
m_{T}\left(z\right)=\frac{-2\left(d-1\right)}{\left(d-2\right)z+d\sqrt{z^{2}-4\left(d-1\right)}}.
\]
Thus
\[
\lim_{\varepsilon\to0}\text{Im}\left(m_{T}\left(d+i\varepsilon\right)\right)=\lim_{\varepsilon\to0}\text{Im}\left(m_{\mathbb{N}}\left(d+i\varepsilon\right)\right)=0.
\]
Additionally the denominator of $\text{Im}\left(m\right)$ satisfies,
\begin{align*}
1-m_{T}\left(d+i0\right)m_{\mathbb{N}}\left(d+i0\right)=&1-\left(d-1\right)\frac{d-\sqrt{d^{2}-4}}{\left(d-2\right)d+d\sqrt{d^{2}-4\left(d-1\right)}}=\\
&1-\frac{\left(d-1\right)\left(d-\sqrt{d^{2}-4}\right)}{2d\left(d-2\right)} \neq 0
\end{align*}
since the number $\sqrt{d^{2}-4}$ is
irrational for every $2<d\in\mathbb{N}$ ($d^2-4$ is not a perfect square). Thus
\begin{equation} \nonumber
\lim_{\varepsilon\to0}\left|1-m_{T}\left(d+i\varepsilon\right)m_{\mathbb{N}}\left(d+i\varepsilon\right)\right|^{2}>0,
\end{equation}
and we get
\[
\lim_{\varepsilon\to0}\text{Im}\left(m\left(d+i\varepsilon\right)\right)=0.
\]
This implies that $d\notin\lambda\in\sigma_\text{disc}\left(A_{\tilde{T}}\right)$,
and we can conclude that $d\notin\sigma\left(A_{\widetilde{T}}\right)$.
\end{proof}

\begin{remark*} In fact $\sigma\left(A_{\widetilde{T}}\right)=\sigma\left(A_T\right)$.
The inclusion $\sigma\left(A_{T}\right)\subseteq\sigma\left(A_{\widetilde{T}}\right)$ is clear.
Additionally, it is not hard, but is a bit cumbersome to see that for any $\lambda\notin\sigma\left(A_T\right)$ the expression
$1-m_{T}\left(\lambda\right)m_{\mathbb{N}}\left(\lambda\right)$ is nonzero, and thus in this case also
$\lambda\notin\sigma\left(A_{\widetilde{T}}\right)$.
\end{remark*}

\chapter{Characterizing $\sigma_{\text ess}(H)$ on regular trees}\label{chapter3}
After studying in the previous chapter the limitation of the method of Last-Simon \cite{LastSimonEss} on general graphs, we now turn to the natural problem of generalizing the first characterization of the essential spectrum \eqref{eq:sigmaEssChar} to infinite trees and especially to regular trees.
First, we review in a sketch an argument which overcomes the limitations and produces a positive result of the form \eqref{eq:sigmaEssLastSimonChar} on regular trees. In the rest of this chapter we use a different method to study a special case, of Schr\"odinger operators $H=\Delta+Q$ on a regular tree, for which $Q$ has a spherical symmetry around a fixed root. In this case we present a constructive proof of \eqref{eq:sigmaEssChar}, which also enables us to obtain a better understanding of the spectral properties of the problem. 
Finally in Section \ref{sec:sparsePot} we implement the results and calculate the essential spectrum for an example of a Schr\"odinger operator with sparse spherically symmetric potential.
The content of this chapter is based on  \cite{BDE}.

\section{Introduction}
We first recall the 1-dimensional characterization of the essential spectrum \eqref{eq:sigmaEssChar}. Let $J$ be a one sided Jacobi matrix. Then
\begin{equation}
\sigma_{ess}(J) = \bigcup_{J^{\left(r\right)}\text{ is a right limit of \ensuremath{J}}}\sigma\left(J^{\left(r\right)}\right)\label{eq:repMainResult}.
\end{equation}
As we know from Chapter \ref{chapter2} the method of Last--Simon \cite{LastSimonEss} for proving this characterization (with a closure on the right hand side) can be generalized to a family of graphs of uniform polynomial growth. On the other hand, we have seen an example for a graph (which is not included in this family) for which \eqref{eq:repMainResult} fails. As we have seen, the right limits of this example were all trees. We have additionally seen that it is impossible to generalize this method by considering balls on graphs of exponential growth.

It is therefore somewhat surprising that on regular trees, which are in some sense the canonical example of exponentially growing graphs, \eqref{eq:sigmaEssLastSimonChar} still holds.

We shall focus here on the case of Schr\"odinger operators $H=\Delta+Q$ on regular trees, for which $Q$ has a spherical symmetry around some fixed root (see Definition \ref{def:sphSymOp} below). We shall refer to this case as the spherically symmetric case.
\begin{theorem}
\label{thm:spSymChar}Assume $H$ is a bounded and spherically symmetric Schr\"odinger operator on $\ell^{2}\left(T\right)$ where $T$ is a regular tree, then
\[
\sigma_{\text{ess}}\left(H\right)=\bigcup_{L\in\mathcal{R}(H)}\sigma\left(L\right).
\]
\end{theorem}
As before, here $\mathcal{R}(H)$ denotes the set of $\mathcal{R}$-limits of $H$.

We give a constructive proof for this theorem in Section~3.3 below. The argument involves a correspondence between the essential spectrum of $H$ and the spectra of specific 1-dimensional Jacobi operators. This result might be of interest on its own and is useful in calculating $\sigma_{\text ess}$ in some cases.
Denote by $V_0\subset\ell^2(T)$ the subspace which is spanned by $\delta_{v_0}$ and $H$. By the Gram-Schmidt process on the sequence of vectors $H^n\delta_{v_0}$, $n\in\mathbb{N}$ we get an orthonormal sequence which we will use as a complete orthonormal set for $V_0$. Clearly $HV_0\subset V_0$ and the restriction of $H$ to this subspace is unitarily equivalent to a Jacobi matrix. We denote by $J_H$ the corresponding matrix.
\begin{prop}\label{prop:spSymJacobi1} Assume $H$ is a bounded and spherically symmetric Schr\"odinger operator on $\ell^{2}\left(T\right)$ where $T$ is a regular tree.
Then
\[
\sigma_{\text{ess}}\left(H\right)\subseteq{\left(\bigcup_{r}\sigma\left(J^{\left(r\right)}\right)\right)\cup\left(\bigcup_{s}\sigma\left(J_{\left(s\right)}\right)\right)}
\]
where $\left\{ J^{\left(r\right)}\right\} $ is the set of right limits
of $J_H$, and $\left\{ J_{\left(s\right)}\right\} $ is the set of
strong limits of the sequence $\left\{ J_{n}\right\} _{n=1}^{\infty}$ of tails of $J_H$.
\end{prop}
\begin{prop}\label{prop:spSymJacobi2}
In the setting of Proposition~\ref{prop:spSymJacobi1},
\[
{\left(\bigcup_{r}\sigma\left(J^{\left(r\right)}\right)\right)\cup\left(\bigcup_{s}\sigma\left(J_{\left(s\right)}\right)\right)}\subseteq\bigcup_{L\in\mathcal{R}(H)}\sigma\left(L\right)
\]
\end{prop}

Theorem~\ref{thm:spSymChar} immediately follows from Theorem~\ref{thm:Thm1} and these two propositions. Moreover, we get an additional characterization of the essential spectrum for such operators.

\begin{theorem}\label{thm:spSymJacobi}
Assume $H$ is a bounded and spherically symmetric Schr\"odinger operator on $\ell^{2}\left(T\right)$ where $T$ is a regular tree.
Then
\[
\sigma_{\text{ess}}\left(H\right)={\left(\bigcup_{r}\sigma\left(J^{\left(r\right)}\right)\right)\cup\left(\bigcup_{s}\sigma\left(J_{\left(s\right)}\right)\right)}
\]
where $\left\{ J^{\left(r\right)}\right\}$ and  $\left\{ J_{\left(s\right)}\right\}$ are as in Proposition~\ref{prop:spSymJacobi1}.
\end{theorem}

After we prove Theorem \ref{thm:spSymChar}, S.~Denisov has shown us an argument which we describe in \cite{BDE}, and proves generally that
\begin{equation}\label{eq:onTrees}
\sigma_{\text{ess}}\left(H\right)=\overline{\bigcup_{L\in\mathcal{R}(H)}\sigma\left(L\right)}.
\end{equation}
for any Schr\"odinger operator on a regular tree.
As we sketch in Section~3.2 the proof overcomes the restrictions we have seen in Chapter \ref{chapter2} by using trial functions which are supported on annuli around a fixed origin, instead of functions supported on balls.

\begin{remark}
We expect this result to carry over to the case of non-regular trees as well. We restricted our attention to regular trees in \cite{BDE}, for simplicity. Further, it is an interesting open problem to study whether this argument can be generalized to a larger family of graphs (e.g.\ by applying some restrictions on their connectivity).
\end{remark}

\begin{remark}
The general result \eqref{eq:onTrees} includes a closure on the right hand side. Since it is an adaptation of the argument of \cite{LastSimonEss} and Section 2.2, we cannot do better with this method. However, as mentioned in Chapter 1, in the one dimensional case the result \eqref{eq:sigmaEssChar} is known to hold without the closure. For now it is still an open problem whether this result can be extended to trees.
In the spherically symmetric case we have Theorem 3.1, which holds without the closure on the right hand side. 
\end{remark}
\begin{remark}
By a simple adaptation everything holds also for Jacobi operators on the tree. For simplicity we treat Schr\"odinger operators.
\end{remark}

\begin{remark}
We expect the argument used in the proof of Theorem \ref{thm:spSymChar} (and Theorem~\ref{thm:spSymJacobi})
to carry over to more general graphs which have a spherical symmetry around some fixed root. For example, consider a spherically symmetric tree with cycles (an example  is given in Section~\ref{sec:example2} in which we additionally assume sparseness of this graph). Notice that this type of examples cannot be treated with the argument used on trees. We leave this direction to a future work.
\end{remark}

\section{Overcoming the obstacles in the argument on trees}\label{sec:PfonRegTrees}
For completeness, we sketch here Denisov's proof of \eqref{eq:onTrees} for general operators on regular trees.
The full proof is included in \cite[Section 4]{BDE}. Given a regular tree $T$, we fix an origin $v_0$ and define trial functions of the form:
\[
\chi_{k,r}(v)=\left\{
\begin{array}{cc}
1-\frac{||v|-k|}{r}   & {\rm if} \quad ||v|-k|<R  \\
0 & {\rm otherwise,}
\end{array}
\right.\quad
\]
where $k,r\in\mathbb{N}$, and as usual $|v|=\text{dist}\left(v,v_0\right)$. Define further
\[\psi_{k,r}(v)=\nicefrac{\chi_{k,r}(v)}{\sqrt{\sum_k \chi^2_{k,r}(v)}}.\]
We continue by defining $C^{(r)}$ as in the proof of Theorem \ref{thm:Thm4}, i.e.\
\[
C^{(r)}=-2\sum_k \left[H,\psi_{k,r}\right]^2.
\]
The number of nonzero terms in this sum  is linear with $r$, in contrast to the exponential number of terms we had in the corresponding sum in Chapter \ref{chapter2}. This fact allows a better bound on $\left\Vert C^{(r)} \right\Vert$,
by which we can conclude that $\lim_{r\to\infty}\left\Vert C^{(r)}\right\Vert=0$.
We continue tracing  the argument of Last-Simon and take a Weyl's sequence of orthonormal approximate eigenfunctions $\left\{\varphi_n\right\}_{n\in\mathbb{N}}$ for $\lambda$.
By following the argument, we split each $\varphi_n$ into candidate approximate eigenfunctions $\varphi_{k,n}=\psi_{k,r}\varphi_n$ (for $k\in \mathbb{N}$) which are supported on annuli of diameter $2r$, moving away to infinity, and such that the sum
\[
\sum_k \left\Vert\left(H-\lambda\right)\varphi_{k,n}	\right\Vert^2
\]
is suitably small (for large enough  $r$ and $n$).
In the second part of the proof the structure of the tree is exploited, in order to split each function $\varphi_{k,n}$ (in the connections closest to the origin) into a sum of candidate approximate eigenfunctions which are locally supported in the graph (on balls of radius $r$). Next we conclude that for each $n\in\mathbb{N}$ there exists at least one such locally supported approximate eigenfunction which is nonzero.
Now we can find $\mathcal{R}$-limits and corresponding approximate eigenfunctions by an adaptation of the argument we used in Section 2.2, based on compactness.

\section{$\mathcal{R}$-Limits on regular trees}
As a preliminary for the proof of Propositions~\ref{prop:spSymJacobi1} and \ref{prop:spSymJacobi2} we discuss here some properties of $\mathcal{R}$-limits of regular trees and present some useful definitions and notation.
Let $H$ be a Schr\"odinger operator on a $d$-regular tree $T$ (with $d>2$) with root vertex $v_0\in T$.
In this case the Laplace operator is different from the adjacency operator by a constant, i.e.\ $\Delta_T=A_T- dI$. The influence of the term $-dI$ on the spectrum is just a shift, and thus, for convenience, we can ignore this term (or absorb it into the potential $Q$) and write $\Delta$ for the adjacency operator.
Let $\left\{H',T',v_0'\right\}$ be an $\mathcal{R}$-limit of $H$ along a path to infinity ${\left\{v_j\right\}_{j=0}^\infty\subset V\left(T\right)}$. The following are properties of $\mathcal{R}$-limits on regular trees.

First, since a regular tree is homogeneous, any $\mathcal{R}$-limit of $H$ is defined on the same regular tree. Thus we can assume that $T'$ is another copy of the $d$-regular tree.
Next, we will rely on this property in order to develop a more geometric formulation for the definition of $\mathcal{R}$-limit on regular trees.

\begin{definition}
Given an isometry between trees $f:T\to T'$, denote by $I_{f}:\ell^{2}\left(T\right)\to\ell^{2}\left(T'\right)$
the isometry operator: $\left(I_{f}\psi\right)\left(v\right)=\psi\left(f\left(v\right)\right)$.
\end{definition}
\begin{definition}
For any vertex $u\in T$, $R>0$, denote by $P_{u,R}$ the projection operator onto $\ell^{2}\left(B_{R}\left(u\right)\right)$.
Further, for any operator $X$ on $\ell^{2}\left(T\right)$, denote the operator $X_{u,R}=P_{u,R}XP_{u,R}$.
\end{definition}

\begin{prop} \label{prop:RlimOnRegT}
Let $H$ be a Schr\"odinger operator on a $d$-regular tree $T$, and assume $\left\{H',T',v_0'\right\}$ is an $\mathcal{R}$-limit of $H$  along a path to infinity $\left\{v_j\right\}_{j=0}^\infty$.
Then there exists a subsequence of vertices $\left\{u_j\right\}_{j=1}^\infty\subseteq\left\{v_j\right\}_{j=0}^\infty$ and a sequence of tree isometries $\left\{f_{j}:T\to T'\right\}_{j=1}^\infty$, with $f_{j}\left(u_{j}\right)=v_{0}'$ $($see Figure~\ref{figure2}$)$ satisfying, for every $R>0$,
\begin{equation} \label{eq:rlimonregT}
\left\Vert I_{f_{j}}H_{u_{j},R}I_{f_{j}}^{-1}-H'_{v_{0}',R}\right\Vert \underset{j\to\infty}{\longrightarrow}0.
\end{equation}
Moreover, if $H'$ is a Schr\"odinger operator on $\left\{T',v_0'\right\}$ and there exist sequences $\left\{u_j\right\}_{j=1}^\infty$, $\left|u_j\right|\to\infty$ and $\left\{f_j\right\}_{j=1}^\infty$ as above, such that \eqref{eq:rlimonregT} is satisfied for any $R>0$, then $\left\{H',T',v_0'\right\}$ is an $\mathcal{R}$-limit of $H$.
\end{prop}

\begin{figure}[ht!]
\centering
\includegraphics[width=130mm]{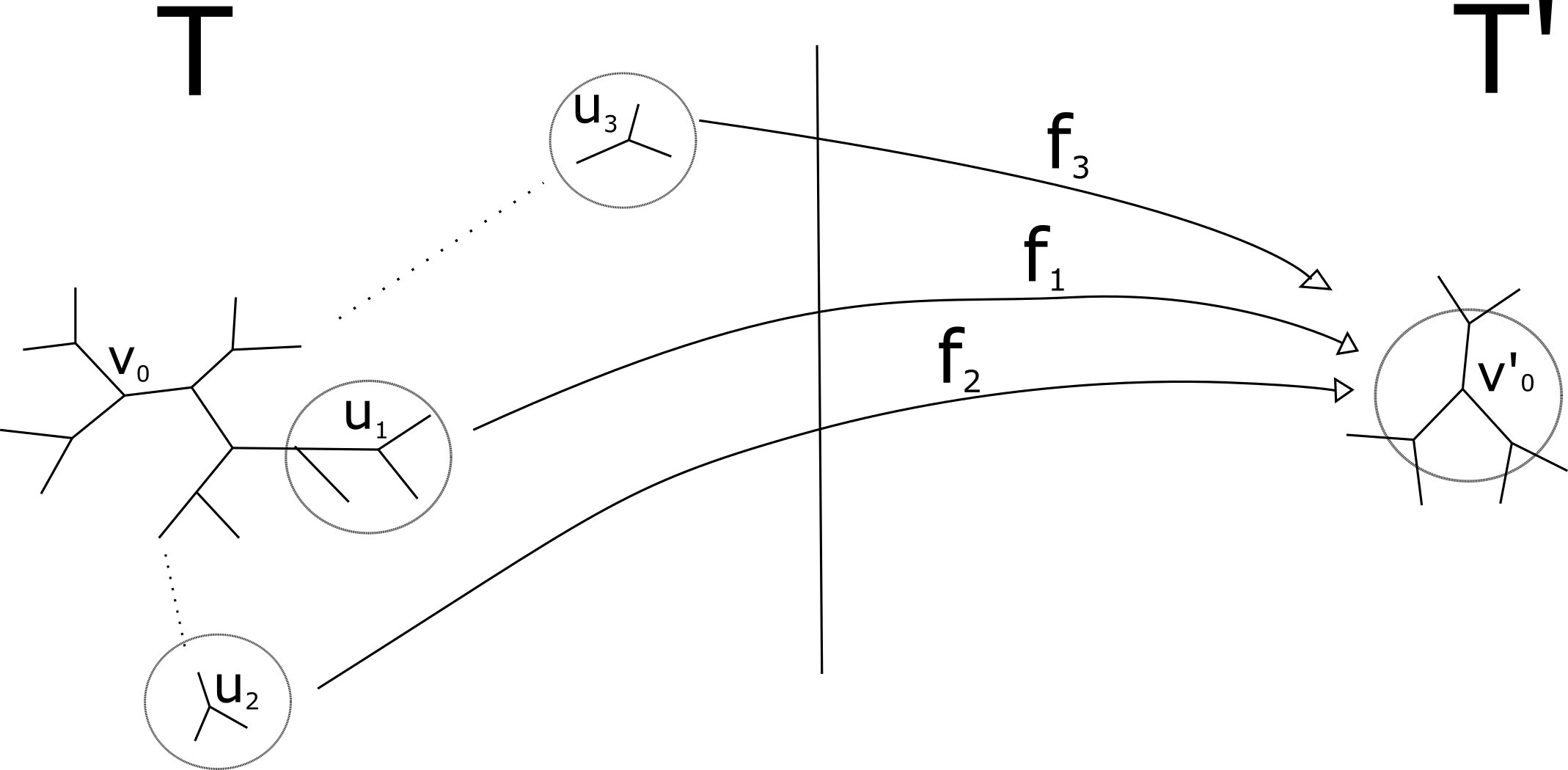}\\
\caption{The isometries $f_j$ between $T$ and $T'$. \label{figure2}}
\end{figure}

\begin{proof}
Assume the $\mathcal{R}$-limit $H'$ is obtained along the subsequence $\left\{v_{n_j}\right\}_{j=1}^\infty$ (as in Definition~\ref{def:rlimdef}), and define $u_j=v_{n_j}$. Note that for any $j\in\mathbb{N}$ the coherent isomorphisms sequence $\left\{\mathcal{I}^{(j)}_k \right\}_{k=1}^\infty$ can be extended to an isomorphism $\mathcal{I}_j:\ell^2\left(T\right)\to\ell^2\left(\mathbb{N}\right)$, that agrees on balls around $u_j$ with $\mathcal{I}^{(j)}_k$. Similarly, the sequence $\mathcal{I'}_k$ can be extended to an isomorphism $\mathcal{I'}:\ell^2\left(T'\right)\to\ell^2\left(\mathbb{N}\right)$.
Now we can define $I_{f_j}$ (and $f_j$) by $I_{f_j}={\mathcal{I'}^{-1}\mathcal{I}_{j}}$. The convergence \eqref{eq:rlimonregT} then follows directly from \eqref{eq:rlimdef}.

In the other direction, assume $\left\{u_j\right\}$ is a sequence of vertices and $\left\{f_j\right\}_{j=1}^\infty$ is a sequence of tree isometries as above.
By compactness there is a path to infinity, $\left\{v_j\right\}_{j=1}^\infty$, which contains a subsequence $\left\{u_j'\right\}_{j=1}^\infty\subseteq\left\{u_j\right\}_{j=1}^\infty$, i.e.\ $u_j'=v_{n_j}$, for a corresponding sequence $\left\{n_j\right\}\subseteq\mathbb{N}$. Let $\left\{\mathcal{I'}_k\right\}_{k=1}^\infty$ be any sequence of coherent isomorphisms of $T'$ around $v_0'$. We can now define for any $j\in\mathbb{N}$ a coherent sequence of isomorphisms $\left\{\mathcal{I}^{(j)}_{k}\right\}_{k=1}^\infty$ of $T$ around $u_j'$ by
\begin{equation} \nonumber
\mathcal{I}^{(j)}_{k}=\mathcal{I'}_k \mathcal{I}_{f_j}|_{B_k\left(u_j'\right)}.
\end{equation}
The convergence \eqref{eq:rlimdef} follows directly from  \eqref{eq:rlimonregT}.
\end{proof}

Next we present another property of $\mathcal{R}$-limits on regular trees, which is the possibility to choose the isometries $\left\{f_j\right\}$ such that the path from $u_j$ to $v_0$ is always mapped to the same sequence of vertices in $T'$ from $v_0'$ to infinity.
\begin{definition}
Denote by $N\left(u\right)$ the set of \textbf{neighbors} of the vertex $u$. Additionally, assuming $u\neq v_0$, denote by $A\left(u\right)=A^1(u)=A_{v_0}\left(u\right)$ the vertex ${w\in N(u)}$ on the (shortest) path from $v_0$ to $u$. For $n\in \mathbb{N}$, let $A^n(v)=A\left(A^{n-1}\left(v \right)\right)$.
\end{definition}

\begin{prop} \label{prop:ancestorsSeq}
Let $\left\{T,v_0\right\}$ and $\left\{T',v_0'\right\}$ be two copies of the $d$-regular tree, assume $\left\{v_j\right\}_{j=0}^\infty$ is a path to infinity in $T$, $\left\{w_j\right\}_{j=1}^\infty\subseteq\left\{v_j\right\}_{j=0}^\infty$ is a subsequence, and  $\left\{g_{j}:T\to T'\right\}_{j=1}^\infty$ is a sequence of tree isomtries, with $g_j\left(w_j\right)=v_0'$. Then there exist a subsequence of vertices $\left\{u_j\right\}_{j=1}^\infty\subseteq\left\{w_j\right\}_{j=1}^\infty$, a corresponding subsequence of tree isometries $\left\{f_{j}\right\}_{j=1}^\infty\subseteq\left\{g_{j}\right\}_{j=1}^\infty$, and a path to infinity ${\left\{v_k'\right\}_{k=0}^\infty\subset V\left(T'\right)}$, so that
\begin{equation} \label{eq:ancestorSeqCond}
\forall n\in\mathbb{N},\ f_{j}\left(A^{n}\left(u_{j}\right)\right)=v_{n}',
\end{equation}
for any $j\in\mathbb{N}$ such that $\left|u_{j}\right|>n$.
\end{prop}

\begin{proof}
By compactness, the sequence $\left\{ f_{j}\left(A\left(w_{j}\right)\right)\right\} _{j=1}^{\infty}\subset T'$ contains a vertex ${v\in N(v_{0}')}$ an infinite number of times, denote it by $v_{1}'=v$
and restrict to this sebsequence. Continue further inductively
to define the path to infinity $\{v_{k}'\}_{k=0}^{\infty}$. Finally take the diagonal over the resulting subsequences of vertices $\subseteq\left\{w_j\right\}_{j=1}^\infty$ and tree isometries $\subseteq\left\{g_j\right\}_{j=1}^\infty$ to define the subsequence $\left\{u_{j}\right\}_{j=1}^\infty$ and the subsequence $\left\{f_{j}\right\}_{j=1}^\infty$.
\end{proof}

We refer to the sequence $\{v_{k}'\}_{k=0}^{\infty}$ from Proposition~\ref{prop:ancestorsSeq} as an \emph{ancestors sequence}.

\begin{corollary}
Let $H$ be a Schr\"odinger operator on a $d$-regular tree $T$, and assume $\left\{H',T',v_0'\right\}$ is an $\mathcal{R}$-limit of $H$  along a path to infinity ${\left\{v_j\right\}_{j=0}^\infty}$.
Then there exists a subsequence of vertices $\left\{u_j\right\}_{j=1}^\infty\subseteq\left\{v_j\right\}_{j=0}^\infty$, a sequence of tree isometries $\left\{f_{j}\right\}_{j=1}^\infty$ as in Proposition~\ref{prop:RlimOnRegT}, and an ancestors sequence ${\left\{v_k'\right\}_{k=0}^\infty\subset V\left(T'\right)}$, such that \eqref{eq:rlimonregT} and \eqref{eq:ancestorSeqCond} are satisfied.
\end{corollary}

We conclude with some more definitions that will be useful below,
\begin{definition}[\textbf{Descendants of an Ancestors Sequence}]
Given an ancestors sequence $\left\{ v_{k}'\right\} _{k=0}^{\infty}$
we define a descendant (strict total) order relation $>_{D}$ on the neighboring
vertices of the tree, recursively, as follows:

\begin{enumerate}
\item for any $k\in\mathbb{N}\cup\{0\}$, $v_{k+1}'>_{D}v_{k}'$.
\item if $w>_{D}u$ then for any $z\in N\left(u\right)$ so that $z \neq w$,
$u>_{D}z$ (and so also $w>_D z$ is required, to ensure transitivity).
\end{enumerate}
\end{definition}
\begin{definition}\label{def:descSubtree}
 Given a vertex $v'\in T'$ denote by $\Gamma_{v'}$ the subtree of descendants of $v'$,
\[
\Gamma_{v'}=\left\{ u\in T'\,|\,v'>_{D}u\right\} \cup\left\{v'\right\}.
\]
\end{definition}

\section{Proof of Propositions 3.1.1 and 3.1.2}

\begin{definition}\label{def:sphSymOp} We say that an operator $H=\Delta+Q$ on a regular tree, $T$, is spherically symmetric (around some vertex, $v_0 \in T$) if the potential satisfies $Q(v)=q\left(\left|v\right|+1\right)$, where $\left|v\right|=\text{dist}\left(v,v_{0}\right)$  and $q: \mathbb{N} \rightarrow \mathbb{R}$.
\end{definition}

\begin{definition}
A Jacobi matrix with parameters $\{a_k\}_{k=1}^\infty$ and $\{b_k\}_{k=1}^\infty$ is a matrix of the form
\[
J=\left[\begin{array}{ccccc}
b_1 & a_1 & 0\\
a_1 & b_2 & a_2 & 0\\
0 & a_2 & b_3 & a_3 \\
 & 0 & a_3 & b_4\\
 &  &  &  & \ldots
\end{array}\right].
\]
\end{definition}
\begin{definition}
The $k$-th tail of a semi-infinite matrix $A$ is a semi-infinite matrix $A^{[k]}$ defined by $\left(A^{[k]}\right)_{i,j}=(A)_{i+k,j+k}$, where $i,j\in\mathbb{N}$.
\end{definition}

\begin{proof} [Proof of Proposition~\ref{prop:spSymJacobi1}]
We begin with exploring the essential spectrum of $H$ using the symmetry of the system.
The spherical symmetry implies that $H$ decomposes as a direct sum (see \cite{AF,Breuer}) of Jacobi matrices
\[H\cong\bigoplus_{n=1}^\infty\left(\oplus_{j=1}^{k_n} S_n\right),\]
where $S_n$ has the parameters
\begin{equation*}
a_k^{(n)}=\begin{cases} \sqrt{d}\ \ \ \ \ \ \ \ n=k=1 \\
\sqrt{d-1}\ \ \ \textrm{otherwise},
\end{cases}
\end{equation*}
$b_k^{(n)}=q_{k+n-1}$, and $k_n$ is some explicit function of $n$ and the degree of the tree (see \cite{Breuer}). Notice that the direct sum includes $k_n$ copies of $S_n$ for each $n\in\mathbb{N}$. Additionally, note that the matrix $S_n$ is the ($n-k$)'th tail of the matrix $S_k$ for any $n>k\in\mathbb{N}$. The matrix we denoted previously by $J_H$ is actually $S_1$, which is,
\[
S_{1}=J_H=\left[\begin{array}{ccccc}
q_{1} & \sqrt{d} & 0\\
\sqrt{d} & q_{2} & \sqrt{d-1}\\
0 & \sqrt{d-1} & q_{3} & \sqrt{d-1}\\
 &  & \sqrt{d-1} & q_{4}\\
 &  &  &  & \ldots
\end{array}\right]
\]
All other $S_n$'s are tails of this Jacobi matrix.

Proposition \ref{prop:spSymJacobi1} now follows from the next proposition, which we give in a more general form.
\end{proof}

\begin{prop}\label{prop:PropA} Assume $J$ is a bounded Jacobi matrix with parameters $\{a_j\}_{j=1}^\infty$ and $\{b_j\}_{j=1}^\infty$,
satisfying
\[
\sup_{j}\left(\left|a_{j}\right|+\left|b_{j}\right|+\left|a_{j}\right|^{-1}\right)=M<\infty.
\]
Let $\left\{ J_{n}\right\} _{n=1}^{\infty}$ be a subsequence of the
sequence of tails of $J$, $\left\{J^{[k]}\right\}_{k=1}^\infty$, let $\left\{ i_{n}\right\}_{n=1}^\infty \subset\mathbb{N}$, and
let $K=\bigoplus_{n=1}^{\infty}\left(\oplus_{j=1}^{i_{n}}J_{n}\right)$.
Then the essential spectrum of $K$ satisfies:
\[
\sigma_{\text{ess}}\left(K\right)\subseteq{\left(\bigcup_{r}\sigma\left(J^{\left(r\right)}\right)\right)\cup\left(\bigcup_{s}\sigma\left(J_{\left(s\right)}\right)\right)}
\]
where $\left\{ J^{\left(r\right)}\right\} $ is the set of right limits
of $J$, and $\left\{ J_{\left(s\right)}\right\} $ is the set of
strong limits of the sequence $\left\{ J_{n}\right\} _{n=1}^{\infty}.$
\end{prop}

\noindent Before proving Proposition~\ref{prop:PropA} we present another preliminary
proposition:
\begin{prop}\label{prop:PropB}
The essential spectrum of $K$ satisfies:
\[
\sigma_{\text{ess}}(K)=\sigma_{\text{ess}}\left(J\right)\cup\Sigma
\]
where,
\begin{align*}
\Sigma_{\text{ }}=&\ \Sigma_0\big\backslash\sigma_{\text{ess}}\left(J\right), \\
\Sigma_0= &\left\{ \vphantom{\int_t} E\in\mathbb{R}\,|\,\exists\left\{n_{k}\right\}_{k=1}^{\infty},n_{k+1}\geq n_{k}\in\mathbb{N},\,\left\{g_{k}\right\}_{k=1}^{\infty}\in\ell^{2}\left(\mathbb{N}\right),\,\lambda_{k}\in\mathbb{R}, \right.\\
&\, \left. \text{so that }   J_{n_{k}}g_{k}=\lambda_{k}g_{k},\text{ and }\lambda_{k}\underset{k\to\infty}{\longrightarrow}E\right\}
\end{align*}
$($i.e.\ $\Sigma$ is the set of limit points of eigenvalues of the $J_{n}$'s that are not in $\sigma_{\textrm{ess}})$.
\end{prop}
\begin{proof}[Proof of Proposition~\ref{prop:PropB}]

First, note that $J_{n}$ is a finite rank perturbation of $J$ and thus
$\sigma_{\text{ess}}\left(J\right)=\sigma_{\text{ess}}\left(J_{n}\right)$
for every $n\in\mathbb{N}$. Thus $\sigma_{\text{ess}}\left(J\right)\subseteq\sigma_{\text{ess}}\left(K\right)$.
Additionally by definition $\Sigma\subseteq\sigma_{\text{ess}}\left(K\right)$.
Thus,
\[
\sigma_{\text{ess}}(K)\supseteq\sigma_{\text{ess}}\left(J\right)\cup\Sigma.
\]
For the reverse inclusion, denote $\sigma_{n}=\sigma_{\text disc}(J_n)=\sigma\left(J_{n}\right)\backslash\sigma_{\text{ess}}\left(J_{n}\right)$, so,
\begin{align*}
\sigma\left(J_{n}\right)=\sigma_{\text{ess}}\left(J_{n}\right)\cup\sigma_{n}=\sigma_{\text{ess}}\left(J\right)\cup\sigma_{n}.
\end{align*}
Then,
\begin{align*}
\sigma\left(K\right)&=\overline{\bigcup_{n}\sigma\left(J_{n}\right)}=\overline{\bigcup_{n}\left(\sigma_{\text{ess}}\left(J\right)\cup\sigma_{n}\right)}=\\
&=\overline{\sigma_{\text{ess}}\left(J\right)\cup\left(\bigcup_{n}\sigma_{n}\right)}.
\end{align*}
The essential spectrum is closed and thus we can write
\[
\sigma\left(K\right)=\sigma_{\text{ess}}\left(J\right)\cup\overline{\bigcup_{n}\sigma_{n}},
\]
and we claim that this is exactly:
\[
=\left(\bigcup_{n}\sigma\left(J_{n}\right)\right)\cup\Sigma.
\]
Indeed, $\sigma_{\text{ess}}\left(J\right)\subseteq\sigma\left(J_{n}\right)$
for every $n$, and if $\lambda\in\left(\overline{\bigcup_{n}\sigma_{n}}\right)\Big\backslash\sigma_{\text{ess}}\left(J\right)$
then it is either an isolated eigenvalue of some $J_{n}$ (so
$\lambda\in\sigma\left(J_{n}\right)$), or an accumulation point
of eigenvalues of $J_{n}$'s, in which case $\lambda\in\Sigma$.
The opposite inclusion follows immediately. We now have,
\[
\sigma\left(K\right)\backslash\left(\sigma_{\text{ess}}\left(J\right)\cup\Sigma\right)=\left[\left(\bigcup_{n}\sigma\left(J_{n}\right)\right)\cup\Sigma\right]\bigg\backslash\left(\sigma_{\text{ess}}\left(J\right)\cup\Sigma\right)=
\]
\[
=\left[\bigcup_{n}\left(\sigma\left(J_{n}\right)\backslash\sigma_{\text{ess}}\left(J\right)\right)\right]\bigg\backslash\Sigma.
\]
Each term $\sigma_{n}=\sigma\left(J_{n}\right)\backslash\sigma_{\text{ess}}\left(J\right)=\sigma\left(J_{n}\right)\backslash\sigma_{\text{ess}}\left(J_{n}\right)=\sigma_{\text{disc}}\left(J_{n}\right)$
contains only isolated eigenvalues of finite multiplicity. Every accumulation
point of such points is contained in $\Sigma$. Hence, every point
in $\left(\cup_{n}\sigma_n\right)\backslash\Sigma$
is an isolated eigenvalue of finite multiplicity \textbf{of finitely
many} \textbf{$J_{n}$'s}, and thus it is also an isolated eigenvalue of finite multiplicity of $K$. We conclude that $\sigma\left(K\right)\backslash\left(\sigma_{\text{ess}}\left(J\right)\cup\Sigma\right)\subseteq\sigma_{\text{disc}}\left(K\right)$,
and thus $\sigma_{\text{ess}}\left(K\right)\subseteq\sigma_{\text{ess}}\left(J\right)\cup\Sigma$.
\end{proof}

\begin{proof}[Proof of Proposition~\ref{prop:PropA}]

By \eqref{eq:sigmaEssChar} we have that $\sigma_{\text{ess}}\left(J\right)\subseteq{\cup_{r}\sigma\left(J^{\left(r\right)}\right)}$.
Thus (using Proposition~\ref{prop:PropB}) it is sufficient to prove that $\Sigma\subseteq{\cup_{s}\sigma\left(J_{\left(s\right)}\right)}$.
Let $E\in\Sigma$. Assume $\left\{ E_{k}\right\}_{k=1}^{\infty} $ is a sequence of eigenvalues  of $\left\{ J_{n_{k}}\right\}_{k=1}^{\infty}$, with $n_{k+1}\geq n_{k}$, so that $E_{k}\to E$, and let $\psi_{k}$ be the corresponding eigenfunctions, satisfying $J_{n_{k}}\psi_{k}=E_{k}\psi_{k}$, $\left\Vert \psi_{k}\right\Vert =1$.
If $\left\{ n_{k}\right\} _{k=1}^{\infty}$
is bounded then $E$ is a limit point of eigenvalues of $J_{n_0}$, where $n_{0}=\max_{k}n_{k}$, which means that $E\in\sigma_{\text{ess}}\left(J\right)$, contradicting $E \in \Sigma$. Thus $n_{k}\to\infty$ and, by restricting to a subsequence if necessary, we may assume that $J_{n_{k}}$ converges strongly to some $J_{\left(s\right)}$, i.e.\ for any $\psi\in\ell^2\left(N\right)$, $\left\Vert J_{\left(s\right)}\psi-J_{n_{k}}\psi\right\Vert \to0$.

Denote by $\mu_{k}$ the spectral measure of $J_{n_{k}}$ with respect to $\delta_{1}=(1,0,0,0,\ldots)$.
Then,
\[
\mu_{k}\overset{\text{w}\,\,}\longrightarrow\mu_{s}
\]
where $\mu_{s}$ is the spectral measure of $J_{\left(s\right)}$
with respect to $\delta_{1}$, and $\overset{\text{w}\,\,}\to$ indicates weak convergence.

\noindent Assume first that $\lim_{k\to\infty}\mu_{k}\left(\left\{ E_{k}\right\} \right)=0$. We shall show that in this case $E \in \sigma_{\textrm{ess}}(J)$, in contradiction with $E \in \Sigma$. Note
\[
\left|\psi_{k}(1)\right|^{2}=\mu_{k}\left(\left\{ E_{k}\right\} \right)\overset{k\to\infty}{\longrightarrow}0.
\]
Define $\left\{ \widetilde{\psi}_{k}\right\} _{n=1}^{\infty}$ by
\begin{equation} \nonumber
\widetilde{\psi}_{k}(j)=\begin{cases}
0\,\,\, & j<n_{k}\\
\psi_{k}(j-n_{k}+1)\,\: & j\geq n_{k}
\end{cases}.
\end{equation}
Then $\widetilde{\psi}_{k}$ satisfies $\left\Vert \widetilde{\psi}_{k}\right\Vert =1$
and,
\[
\left(J\widetilde{\psi}_{k}\right)\left(j\right)=\begin{cases}
\left(J_{n_{k}}\psi_{k}\right)\left(j-n_{k}+1\right)\,\,\, & j\geq n_{k}\\
a_{n_{k}}\psi_{k}\left(1\right) & j=n_{k}-1\\
0 & j<n_{k}-1
\end{cases}.
\]
Thus,
\[
\left\Vert J\widetilde{\psi}_{k}-E\widetilde{\psi}_{k}\right\Vert \leq\left\Vert J\widetilde{\psi}_{k}-\widetilde{J_{n_{k}}\psi_{k}}\right\Vert +\left\Vert \widetilde{J_{n_{k}}\psi_{k}}-E\widetilde{\psi}_{k}\right\Vert \leq
\]
\[
\leq\left|a_{n_{k}}\psi_{k}\left(1\right)\right|+\left|E_{k}-E\right|\to0.
\]
In addition, it is clear that $\widetilde{\psi}_{k}\overset{\text{w}}{\longrightarrow}0$, which implies that $E \in \sigma_{\textrm{ess}}(J)$. Thus, we can conclude that
\[\overline{\lim_{k}}\mu_{k}\left(\left\{ E_{k}\right\} \right)>0,
\]
so by taking a subsequence of $\left\{E_{k}\right\}_{k=1}^{\infty}$ we can assume that
\[\nu=\lim_{k}\mu_{k}\left(\left\{E_{k}\right\}\right)>0\]
exists.
Let $\varepsilon>0$ and $f\in\mathcal{C}\left(\mathbb{R}\right)$
such that $\text{supp}\left(f\right)\subseteq\left(E-\varepsilon,E+\varepsilon\right)$,
$f\geq0$ and $f\left(E\right)>0$. Then there exists some $N\in\mathbb{N}$ such that for every $k>N$,
\[
\left|E_{k}-E\right|<\varepsilon,\ \mu_{k}\left(\left\{ E_{k}\right\} \right)>\sfrac{\nu}{2}\ \text{and } f\left(E_{k}\right)>\sfrac{f\left(E\right)}{2}.
\]
Now, for every $k>N$ we have that
\[
\int f\,d\mu_{k}\geq f\left(E_{k}\right)\mu_{k}\left(\left\{ E_{k}\right\} \right)\geq\sfrac{f(E)\cdot \nu}{4}>0.
\]
Hence by the weak convergence of the measures $\mu_{k}\overset{\text{w}}{\longrightarrow}\mu_{s}$
we conclude that $\int f\,d\mu_{s}\geq\sfrac{f(E)\cdot \nu}{4}>0$
for every such $f$, and thus (since the spectrum is a closed set)
$E\in\overline{\text{supp}\left(\mu_{s}\right)}\subseteq\sigma\left(J_{\left(s\right)}\right)$.
\end{proof}

\begin{proof}[Proof of Proposition \ref{prop:spSymJacobi2}]
We shall now prove that
\[
\left(\bigcup_{r}\sigma\left(J^{\left(r\right)}\right)\right)\cup\left(\bigcup_{s}\sigma\left(J_{\left(s\right)}\right)\right)\subseteq \bigcup_{L\text{ is an $\mathcal{R}$-limit of }H} \sigma(L).
\]
Assume $J^{(r)}$ is a right limit of $J=J_H$ along a sequence
$\left\{ l_{i}\right\} _{i=1}^{\infty}\subseteq\mathbb{N}$, i.e.
\begin{equation} \label{eq:Jrightlim}
\left\Vert J_{l_{i},R}-J_{0,R}^{\left(r\right)}\right\Vert \underset{i\to\infty}{\longrightarrow}0
\end{equation}
for any $R>0$. We claim that we can find a corresponding $\mathcal{R}$-limit
$\left\{L,T',v'_0\right\}$ of $H$,
such that $\sigma\left(J^{\left(r\right)}\right)\subseteq\sigma\left(L\right)$.
Indeed, for $i\in\mathbb{N}$ take some $u_{i}\in T$ so that $\text{\text{dist}}\left(u_{i},v_{0}\right)=l_{i}$,
and take any isomorphism of trees $f_{i}:T\to T'$ so that $f_{i}\left(u_{i}\right)=v_{0}'$.
Moreover, by Proposition~\ref{prop:ancestorsSeq} and restricting to a subsequence if necessary, we can assume the existence of an ancestors sequence $\left\{ v_{j}'\right\} _{j=0}^{\infty}$ such that \eqref{eq:ancestorSeqCond} is satisfied.
For any $j\in\mathbb{N}$ define a Schr\"odinger operator $L^{\left(j\right)}$ on $\ell^2\left(\Gamma_{v_j'}\right)$, with diagonal terms:
\[
\left(L^{\left(j\right)}\right)_{x,x}=\left(J^{\left(r\right)}\right)_{\left|x\right|-j,\left|x\right|-j},
\]
where for $x\in\Gamma_{v_j'}$, $\left|x\right|=\text{dist}\left(x,v_j'\right)$. Note that $L^{\left(j\right)}$ is spherically symmetric around $v_j'$, and that the sequence $\left\{L^{\left(j\right)}\right\}_{j=1}^\infty$ satisfies $L^{(j)}|_{\Gamma_{v'_{j}}}=L^{(k)}|_{\Gamma_{v'_{j}}}$ for $k\geq j$ (see Figure~\ref{figure3}).
\begin{figure}[ht!]
\centering
\includegraphics[width=130mm]{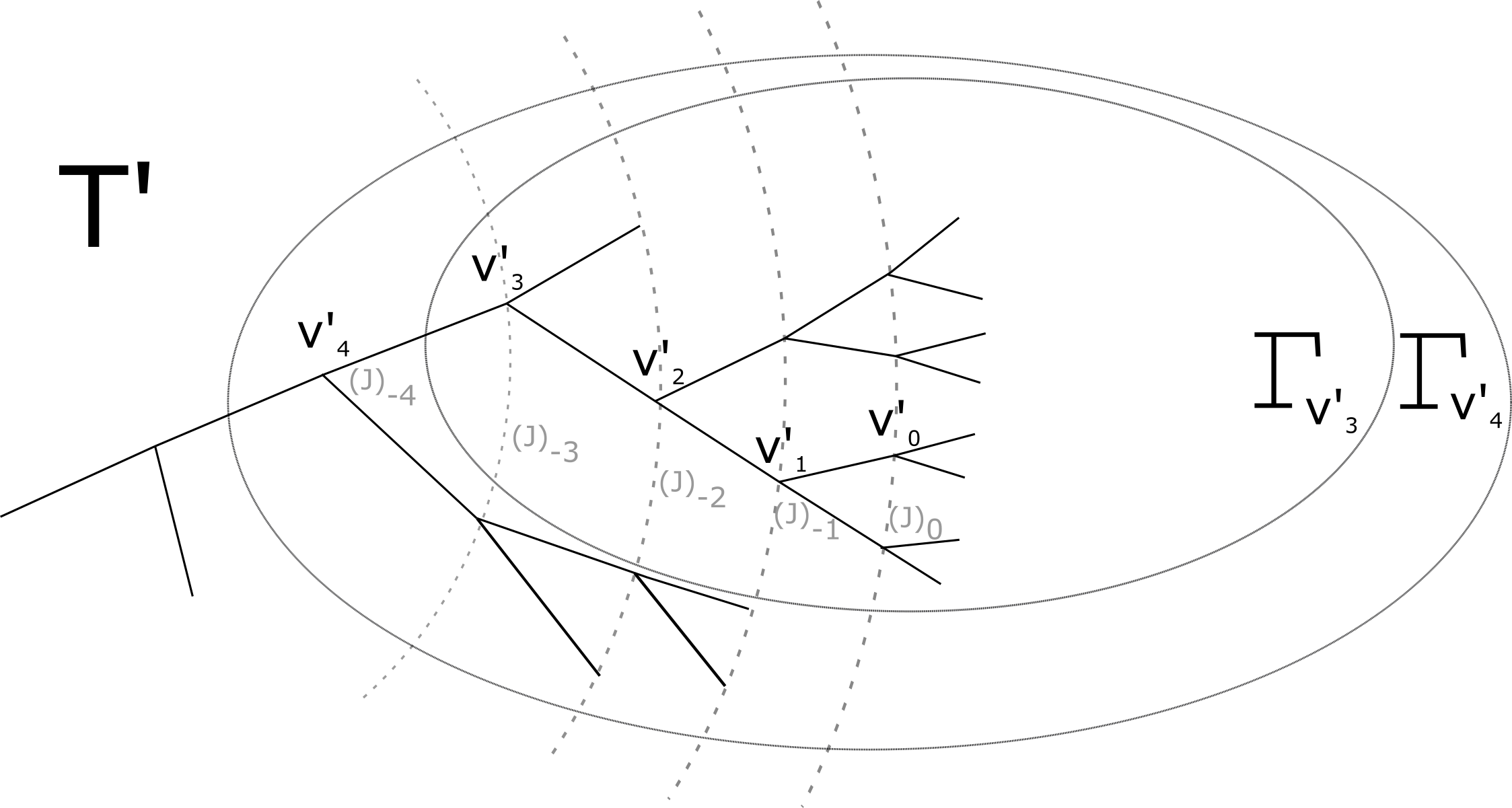}\\
\caption{The subtrees $\Gamma_{v'_j}$, and the diagonal terms of $L$, $(J)_k\equiv\left(J^{(r)}\right)_{k,k}$. \label{figure3}}
\end{figure}
Thus, we may define $L$ on $\ell^2\left(T'\right)$ by $L|_{\Gamma_{v'_{j}}}=L^{\left(j\right)}$. This defines an operator on $\ell^2\left(T'\right)$ since $\cup_{j=1}^\infty \Gamma_{v_j'}=T'$. The sequence $\left\{L^{\left(j\right)}\right\}_{j=1}^\infty$ converges strongly to $L$: Indeed, for any $\varepsilon>0$ and $g\in\ell^{2}\left(T'\right)$ we can find $R>0$ so that $\left\Vert g|_{T'\backslash B_{R}\left(v_{0}'\right)}\right\Vert <\varepsilon$,
and thus for any $j>R+1$,
\[
\left\Vert \left(L-L^{\left(j\right)}\right)g\right\Vert <\left\Vert L\right\Vert \left\Vert g|_{T'\backslash B_{R}\left(v_{0}'\right)}\right\Vert <\left\Vert L\right\Vert \varepsilon.
\]
Notice that by the spherical decomposition (and the symmetry) there exists for any $i\in\mathbb{N}$ a map $n_i(x):B_R\left(u_i\right)\to\mathbb{N}\cap[1,2R+1]$, so that for $x\in B_R\left(u_i\right)$
\[ \left(H_{u_i,R}\right)_{x,x}=\left(J_{l_i,R}\right)_{n_i(x),n_i(x)}.
\]
Similarly, each term $\left(L_{v_0',R}\right)_{y,y}$ is a diagonal term $\left(J^{(r)}_{0,R}\right)_{\widetilde{n}(y),\widetilde{n}(y)}$ (note that $L_{v_0',R}$ is spherically symmetric around $v_R'$, but not around $v_0'$). Moreover, by the construction of $L$, the maps $\widetilde{n}$ and $n_i$ are related by $\widetilde{n}\left(f_i(x)\right)=n_i(x)$ (for $i\in\mathbb{N},\ x\in B_R\left(u_i\right)$).
Now, since any diagonal term of $J_{l_i,R}$ and $J^{(r)}_{0,R}$ repeats at most $d^R$ times as a diagonal term of $H_{u_i,R}$ and $L_{v_0',R}$, and using \eqref{eq:Jrightlim}, it follows that for any $R\in\mathbb{N}$,
\[
\left\Vert I_{f_i}H_{u_i,R}I_{f_i}^{-1} - L_{v_0',R}\right\Vert<d^R\left\Vert J_{l_i,R}-J_{0,R}^{(r)}\right\Vert\underset{i\to\infty}\longrightarrow 0.
\]
Thus by Proposition~\ref{prop:RlimOnRegT}, $L$ is an $\mathcal{R}$-limit of $H$.

\qquad{}The spherical decomposition of $L^{\left(j\right)}$ produces
a direct sum of Jacobi matrices,
\begin{equation} \label{eq:Ljdecomp}
L^{\left(j\right)}\cong\bigoplus_{i=0}^{\infty}\left(\oplus L_{i}^{\left(j\right)}\right),
\end{equation}
where $L_{i}^{\left(j\right)}\in\mathcal{B}\left(\ell^{2}\left(\mathbb{N}\right)\right)$,
has diagonal terms
\[\left(L_{i}^{\left(j\right)}\right)_{n,n}=\left(J^{\left(r\right)}\right)_{n-j+i-1,n-j+i-1}.\] Now, from each approximate eigenfunction of $J^{\left(r\right)}$
we can produce approximate eigenfunctions of $L_{i}^{\left(j\right)}$ above, for any $j-i$ large enough: Assume $g$ is an approximate eigenfunction of $J^{\left(r\right)}$, satisfying $\left\Vert J^{\left(r\right)}g-\lambda g\right\Vert <\varepsilon$, since $g\in\ell^{2}\left(\mathbb{Z}\right)$ we can take $N$ large enough so that $\left\Vert g|_{\mathbb{Z}\backslash\left(-N,N\right)}\right\Vert <\varepsilon$.
For any $m\in\mathbb{Z}$ define $h_{m}\in\ell^{2}\left(\mathbb{N}\right)$
by
\begin{equation}
h_{m}(n)=\begin{cases}
g(n-m-1) & \left|n-m\right|<N\\
0 & \left|n-m\right|\geq N
\end{cases},\nonumber
\end{equation}
then for any $j,i\in\mathbb{N}$ such that $j-i>N$,
\[
\left(L_{i}^{\left(j\right)}h_{j-i}\right)\left(k\right)=\left(J^{\left(r\right)}g|_{[-N,N]}\right)\left(k-j+i-1\right)
\]
 for every $k\in\mathbb{N}\cap\left(-N-j+i-1,N-j+i-1\right)$. Thus,
\begin{eqnarray*}
\left\Vert L_{i}^{\left(j\right)}h_{j-i}-\lambda h_{j-i}\right\Vert  & \leq & \left\Vert J^{\left(r\right)}g|_{[-N,N]}-\lambda g|_{[-N,N]}\right\Vert +\varepsilon
\end{eqnarray*}
\[
\leq\left\Vert J^{\left(r\right)}g-\lambda g\right\Vert +\left(\left\Vert J^{\left(r\right)}\right\Vert +\left|\lambda\right|\right)\left\Vert g|_{\mathbb{Z}\backslash[-N,N]}\right\Vert +\varepsilon<C\cdot\varepsilon.
\]
By the unitary equivalence \eqref{eq:Ljdecomp}, an approximate eigenfuction of some $L_{i}^{\left(j\right)}$ will correspond to an approximate
eigenfunction of $L^{\left(j\right)}$, with the same eigenvalue.
Moreover, since the semi-infinite-matrix $L^{(j)}_i$ depends only on $j-i$, the same function is an approximate eigenfunction of $L^{(j)}$ for any $j$ large enough. Thus, using the strong convergence $L^{(j)}\to L$ we get an approximate eigenfunction of $L$. Therefore $\sigma\left(J^{(r)}\right)\subseteq\bigcup\sigma\left(L\right)$.

We now turn to the case
in which $J_{\left(s\right)}$
is a strong limit of the sequence $\left\{ J_{k}\right\} _{k=1}^{\infty}$.
Any such $J_{\left(s\right)}$ will appear as  the restriction to the half line $\ell^{2}\left(\mathbb{N}\right)$ of some right limit $J^{\left(r\right)}$.
Thus $J_{(s)}$ is contained in the set of matrices $\left\{ L_{i}^{\left(j\right)}\right\} _{i,j=0}^{\infty}$
above. Thus, again, any approximate eigenfunction of $J_{\left(s\right)}$
corresponds to an approximate eigenfunction of some $L_{i}^{\left(j\right)}$,
and thus also of $L$.  Hence $\sigma\left(J_{(s)}\right)\subseteq\bigcup\sigma\left(L\right)$.
\end{proof}

\section{Example: a sparse spherically symmetric potential} \label{sec:sparsePot}
Let $\alpha\in\mathbb{R}$ and define $q:\mathbb{N}\cup\{0\}\to\mathbb{R}$ by
\[
q(n)=
\begin{cases}
\alpha \ \ \ \ n\in\{k^2\,|\,k\in\mathbb{N}\} \\
0 \ \ \ \ \    \text{otherwise}.
\end{cases}
\]
We shall consider a spherically symmetric Schr\"odinger operator on the rooted $d$-regular tree $(T_d,v_0)$ defined by  $H=\Delta+Q$, where $Q(v)=q(|v|)$ (as usual $|v|=\text{dist}(v,v_0)$). We would like to calculate $\sigma_{\text ess}(H)$.

The matrix $J_H$ corresponding to $H$ is the Jacobi matrix with parameters
\[
a_n=\begin{cases}
\sqrt{d} & n=1 \\
\sqrt{d-1} & n>1,
\end{cases} 
\]
\[
b_n = q(n-1).
\]
As a consequence of Theorems \ref{thm:spSymChar} and \ref{thm:spSymJacobi} we know that
\[
\sigma_{\text ess}(H)=\bigcup_{L\in\mathcal{R}(H)}\sigma(L)=\left(\bigcup_{r}\sigma\left(J^{\left(r\right)}\right)\right)\cup\left(\bigcup_{s}\sigma\left(J_{\left(s\right)}\right)\right),
\]
where $\left\{ J^{\left(r\right)}\right\} $ and $\left\{ J_{\left(s\right)}\right\} $ are the sets of right limits of $J_H$ and of
strong limits of the sequence $\left\{ J_{n}\right\} _{n=1}^{\infty}$ of tails of $J_H$, respectively.
There are two right limits of $J_H$:
\[
J^{(0)}=\sqrt{d-1}\Delta_{\mathbb Z}=\left[\begin{array}{ccccc}
 & \ldots\\
\ldots & 0 & \sqrt{d-1}\\
 & \sqrt{d-1} & 0 & \sqrt{d-1}\\
 &  & \sqrt{d-1} & 0 & \ldots\\
 &  &  & \ldots
\end{array}\right],
\]
and 
\[
J^{(1)}=J^{(0)}+\alpha\langle\delta_1,\cdot\rangle\delta_1.
\]
There are infinitely many strong limits, which are:
\[
J_{(0)}= \sqrt{d-1}\Delta_{\mathbb N}=\left[\begin{array}{cccc}
 0 & \sqrt{d-1}\\
 \sqrt{d-1} & 0 & \sqrt{d-1}\\
  & \sqrt{d-1} & 0 & \ldots\\
  &  & \ldots
\end{array}\right],
\]
and, for every $k\in\mathbb{N}$, 
\[
J_{(k)}=J_{(0)}+\alpha\langle\delta_k,\cdot\rangle\delta_k.
\]
Obviously,
\[
\sigma\left(J^{(0)}\right)=\sigma\left(J_{(0)}\right)=\sigma_{\text ess}\left(J^{(1)}\right)=\sigma_{\text ess}\left(J_{(k)}\right)=\left[-2\sqrt{d-1},2\sqrt{d-1}\right].
\]
Thus we are left to calculate the discrete spectrum of $J^{(1)}$ and $J_{(k)}$. 
These matrices are rank one perturbations of $J^{(0)}$ and $J_{(0)}$.
Begin with $J^{(1)}$ and denote $R^{(1)}=\left(J^{(1)}-z\right)^{-1}$, $R^{(0)}(z)=\left(J^{(0)}-z\right)^{-1}$.
By the basic formula of rank-one perturbations we get that
\[
m^{(1)}(z)=\langle\delta_0,R^{(1)}(z)\delta_0\rangle= \frac{m^{(0)}(z)}{1+\alpha m^{(0)}(z)},
\]
where
\[
m^{(0)}(z)=\langle\delta_0,R^{(0)}(z)\delta_0\rangle. 
\]
Using the expression for the Borel transform corresponding to the adjacency operator on $\mathbb{Z}$ (see e.g\ \cite{SimonSz})
\[
m_{\mathbb Z}(z)=-\frac{1}{\sqrt{z^2-4}},
\]
and the relation $J^{(0)}=\sqrt{d-1}\Delta_{\mathbb Z}$ we conclude that
\[
m^{(0)}(z)=-\frac{1}{\sqrt{z^2-4(d-1)}}.
\]
We can get an additional point in the spectrum of $J^{(1)}$ if  $1+\alpha m^{(0)}(z)$ vanishes. Thus
\[
\sqrt{z^2-4(d-1)}=\alpha,
\]
so
\[
z_{\pm}= \pm\sqrt{\alpha^2+4(d-1)}.
\]
By requiring that $m(z)=-\nicefrac{1}{z}+O\left(\nicefrac{1}{z^2}\right)$ we get that the correct branch is 
\[
z_0=\text{sign}(\alpha)\sqrt{\alpha^2+4(d-1)}.
\]
The vector $\delta_0$ is not cyclic for $J^{(1)}$, but together with $\delta_1-\delta_{-1}$  we get a cyclic system. However, $J^{(1)}$ on the subspace spanned by $J^{(1)}$ and $\delta_1-\delta_{-1}$ is equivalent to $J_{(0)}$. Thus we can conclude that
\[
\sigma\left(J^{(1)}\right)=\left[-2\sqrt{d-1},2\sqrt{d-1}\right]\cup\{z_0\}.
\]

As for $J_{(1)}$, similarly, with $R_{(1)}=\left(J_{(1)}-z\right)^{-1}$ and $R_{(0)}=\left(J_{(0)}-z\right)^{-1}$, we have
\[
m_{(1)}(z)=\langle\delta_0,R_{(1)}(z)\delta_0\rangle= \frac{m_{(0)}(z)}{1+\alpha m_{(0)}(z)},
\]
where
\[
m_{(0)}(z)=\langle\delta_0,R_{(0)}(z)\delta_0\rangle. 
\]
Using the expression for the Borel transform corresponding to the adjacency operator on $\mathbb{N}$ (again see e.g\ \cite{SimonSz})
\[
m_{\mathbb N}(z)=\frac{-z+\sqrt{z^{2}-4}}{2},
\]
and the relation $J_{(0)}=\sqrt{d-1}\Delta_{\mathbb N}$ we conclude that
\[
m_{(0)}(z)=\frac{-z+\sqrt{z^{2}-4(d-1)}}{2(d-1)}.
\]
Again, we can get an additional point in the spectrum of $J_{(1)}$ if  $1+\alpha m_{(0)}(z)$ vanishes. Thus
\[
\alpha z - 2(d-1) =\alpha\sqrt{z^2-4(d-1)},
\]
so
\[
z_1= \alpha+\frac{d-1}{\alpha}
\]
and we can conclude that $\sigma\left(J_{(1)}\right)=\left[-2\sqrt{d-1},2\sqrt{d-1}\right]\cup\{z_1\}$.

Consider now the operator $J_{(k)}$. Since $J_{(k)}$ is a rank one perturbation of $J_{(0)}$ we know that $\sigma\left(J_{(k)}\right)$ consist of $\sigma\left(J_{(0)}\right)$ and a possibly additional point in the discrete spectrum. Thus it will be sufficient to find this point by studying the spectral measure of $J_{(k)}$ with respect to the vector $\delta_k$. Denote,
\[
G_{(k),n}(z)=\left\langle\delta_n,\left(J_{(k)}-z\right)^{-1}\delta_n\right\rangle.
\]
Again, by the basic formula of rank one perturbation we have
\[
G_{(k),k}(z)= \frac{G_{(0),k}(z)}{1+\alpha G_{(0),k}(z)},
\]
and we get an additional point in $\sigma\left(J_{(k)}\right)$ if
\[
G_{(0),k}(z)=-\frac{1}{\alpha}.
\]
Define further 
\[
G_{k}^{\mathbb N}(z)=\left\langle\delta_k,\left(\Delta_{\mathbb N}-z\right)^{-1}\delta_k\right\rangle.
\]
Since 
\[
G_{(0),k}(z)=\frac{1}{\sqrt{d-1}} G_{k}^{\mathbb N}\left(\frac{z}{\sqrt{d-1}}\right)
\]
we should solve the equation
\[
G_{k}^{\mathbb N}(u)=-\beta
\]
with $u=\nicefrac{z}{\sqrt{d-1}}$ and 
\[
\beta=\nicefrac{\sqrt{d-1}}{\alpha}.
\]
Using known expressions for $\Delta_{\mathbb N}$ (it follows e.g.\ from \cite[equation~(3.2.34) and example~3.7.3]{SimonSz}) we have
\[
G_{k}^{\mathbb N}\left(x+x^{-1}\right)=\frac{1-x^{2k}}{x-x^{-1}}.
\]
Denote
\[
f_k(x)=\frac{x^{2k}-1}{x-x^{-1}}.
\]
We have that, for any $x\neq\pm1$,
\[
f_k(x)=x\left(1+x^2+x^4+\ldots+x^{2n-2}\right),
\]
and thus $f'_k(x)>0$ for all $x\in\mathbb{R}\backslash\{0,\pm1\}$. Additionally $f_k(0)=0$, $\lim_{x\to\pm\infty}f_k(x)=\pm\infty$. Thus for a given $\beta\neq0$ there exists a single solution $x_k\in\mathbb{R}$ satisfying $f_k(x_k)=\beta$.
Now from $x_k$ we get a corresponding point
\[
z_k=\sqrt{d-1}\left(x_k+x_k^{-1}\right)
\]
in the spectrum of $J_{(k)}$.
Note that in the limit $k\to\infty$ we get, as expected, that $z_k\to z_0$.

We can now conclude that
\[
\sigma_{\text ess}(H)=\left[-2\sqrt{d-1},2\sqrt{d-1}\right]\cup\bigcup_{k=0}^{\infty}\{z_k\}.
\]
On the other hand, the set of $\mathcal{R}$-limits of $H$ is composed of two objects:
\begin{itemize}
\item The adjacency operator on $T$ the $d$-regular tree.
\item A Schr\"odinger operator on $T$ with a specific potential $Q$ that we will describe next. Let $v_0$ be a root for $T$, let $\{v_k\}_{k=0}^{\infty}$ be an ancestors sequence in $T$ (as defined in Proposition~\ref{prop:ancestorsSeq}), and let $\Gamma_{v_k}$ be the subtree of descendants of $v_k$ (as in Definition~\ref{def:descSubtree}). Then, 
\[
Q(v)= \begin{cases}
\alpha\,\,\,\text{ if for some } k\in\mathbb{N}, v\in \Gamma_{v_k} \text{ and } \text{dist}(v,v_k)=k \\
0 \,\,\, \text{ otherwise.}
\end{cases}
\]
The term `sphere at infinity' might be appropriate for the set of points on which $Q$ is nonzero. From the above computation the spectrum of this operator is exactly 
\[
\left[-2\sqrt{d-1},2\sqrt{d-1}\right]\cup\bigcup_{k=0}^{\infty}\{z_k\}.
\]
\end{itemize}

\chapter{Eigenfunction growth of $\mathcal{R}$-limits of $H$}\label{chapter4}

In this chapter we study possible generalizations of the characterization \eqref{eq:sigmaInfChar} of the essential spectrum $\sigma_{\text ess}$, which involves the set of bounded generalized eigenfunctions of right limits of $H$.
By a (non-trivial) adaptation to graphs of the proof given by Simon \cite[Section 7.2]{SimonSz} we produce positive results for graphs of sub-exponential growth.
As a consequence we get also a statement of \eqref{eq:sigmaEssChar} on graphs of uniform sub-exponential growth, which is stronger than the one that we obtained by generalizing Last-Simon.
The proof involves a correspondence between points in the spectrum and generalized eigenfunctions obeying a specific estimate on the growth rate. The statement that, for a given generalized eigenfunction of polynomial growth the corresponding energy is in the spectrum  is known as Shnol's Theorem. We study in more generality a ``reverse" Shnol's Theorem for graphs (following Lenz and Teplyaev \cite{LenzTeplyaev}) in Section~4.3. Finally, in Section~4.4 we show examples of applications of \eqref{eq:sigmaEssChar}.
The content of this chapter is the subject of a work in preparation, partly joint with S.~Beckus.
\section{Introduction}

We first recall \eqref{eq:sigmaInfChar} from Chapter~\ref{chapter1}: on $\mathbb{N}$ and on $\mathbb{Z}^n$ the essential spectrum of a Schr\"odinger operator $H$ can be characterized also in terms of the set of energies corresponding to bounded generalized eigenfunctions of right limits of $H$, i.e.
\begin{equation} \label{eq:1dsigmainfty}
\sigma_{\text ess}(H)=\bigcup_{H^{(r)}\text{ is a right limit of }H}\sigma_\infty\left(H^{(r)}\right),
\end{equation}
where, for a Schr\"odinger operator $J$
\[
\sigma_\infty(J)=\left\{\lambda\,\Big|\, \exists\psi\in\ell^\infty,\,J\psi=\lambda\psi\right\}.
\]
In the current chapter we study a possible generalization of this relation to graphs.

Let $H$ be a Schr\"odinger operator on a rooted graph $\left(G,v_0\right)$. Recall the notation $\mathcal{R}\left(H\right)$ for the set of $\mathcal{R}$-limits of $H$, and the definitions
\[
\left|v\right|=\text{dist}\left(v,v_{0}\right),
\]
\[
\mathcal{S}_{v_0}(r)=\left\{v\in G\,|\, \text{dist}(v,v_0)=r\right\},
\]
\[
S(r)=S_{v_0}(r)=\left|\mathcal{S}_{v_0}(r)\right|.
\]

In order to generalize \eqref{eq:1dsigmainfty} to general graphs we will have to assume a sub-exponential growth rate of the graph, i.e.\ $\forall \gamma>1, \exists C>0$, so that $\forall r\in\mathbb{N}$
\begin{equation}\label{eq:subexpgr}
S_{v_0}(r)< C\gamma^r.
\end{equation}

\begin{theorem}\label{thm:sigmainfty} Assume $\left(G,v_0\right)$ is an infinite graph of sub-exponential growth rate, and $H$ is a bounded Schr\"odinger operator on
$\ell^{2}\left(G\right)$, then
\begin{equation}\label{eq:essubsetinfty}
\sigma_{\text{ess}}\left(H\right)
\subseteq\bigcup_{H'\in\mathcal{R}(H)}\sigma_{\infty}\left(H'\right).
\end{equation} 
\end{theorem}

The proof depends on the existence of a \emph{generalized eigenfunction} for each point in the spectrum, by which we mean a function $\psi:G\to\mathbb{C}$ satisfying $H\psi=\lambda\psi$. The proof relies on the existence of such generalized eigenfunctions of specific growth rate, a property known as a ``reverse" Shnol's Theorem. We shall use the following form of a reverse Shnol's Theorem on graphs, which follows from \cite{LenzTeplyaev} (in which this property is given in a slightly broader context):
\begin{theorem}[reverse Shnol's Theorem, {\cite[Theorem 3]{LenzTeplyaev}}]\label{thm:revSch} Let $(G,v_0)$ be a rooted (infinite) graph of bounded degree and $H$ a bounded Schr\"odinger operator on $\ell^{2}\left(G\right)$. Assume $\omega\in\ell^2(G)$ is real and positive $($i.e.\ $\omega(v)>0\, \forall v\in G)$ and let $\mu$ be a spectral measure for $H$. Then, for $\mu$-a.e.\ $\lambda\in\sigma\left(H\right)$ there exists a generalized eigenfunction $\varphi=\varphi\left(v\right)$ satisfying $H\varphi=\lambda\varphi$,
and additionally
\begin{equation}\label{eq:phiomegainell2}
\varphi(\cdot)\omega(\cdot)\in\ell^2(G).
\end{equation}
\end{theorem}

\begin{corollary}
Let $(G,v_0)$ be a rooted (infinite) graph of bounded degree and $H$ a bounded Schr\"odinger operator on $\ell^{2}\left(G\right)$. Let $\omega\in\ell^2(G)$ be real and positive, then
\[
\overline{\left\{\lambda \in \sigma(H)\,\big|\,\exists\varphi:G\to\mathbb{C},\,H\varphi=\lambda\varphi \wedge \eqref{eq:phiomegainell2}\right\}}=\sigma(H).
\]

\end{corollary}

\begin{corollary}
For $\mu$-a.e.\ $\lambda\in\sigma\left(H\right)$ there exists a corresponding generalized eigenfunction satisfying for any $v\in G$,
\begin{equation} \label{eq:graphefGrowth}
\left|\varphi\left(v\right)\right|\leq \left(|v|+1\right)\cdot \sqrt{S_{v_0}\left(\left|v\right|\right)}.
\end{equation}
\end{corollary}
For completeness we shall include a proof of Theorem~\ref{thm:revSch} and its corollaries in Section~\ref{sec:RSproperty}.

In the case of graphs of sub-exponential growth the following additional property holds, which is a (direct) Shnol's theorem:

\begin{prop}[{\cite[Theorem 4.8]{HK2011}}]\label{prop:HK4.8}
Let $H$ be a bounded Schr\"odinger operator on a rooted graph $(G,v_0)$ of a bounded degree.  Assume {$w:V(G)\to\mathbb{C}$} is a $($non-zero$)$ generalized eigenfunction of $H$, satisfying $(H-\lambda)w=0$.
Assume that $w$ is sub-exponentially bounded with respect to the graph metric,
i.e.\ $e^{-\alpha |\cdot|}w\in\ell^2(G)$ for all $\alpha>0$.
Then $\lambda  \in \sigma(H)$.
\end{prop}

\begin{remark}
This result appears in \cite{HK2011} in more generality, allowing any (bounded) Jacobi operators on the graph.
\end{remark}

If we assume a {\bf uniform} sub-exponential growth rate of the graph, i.e.\ $\forall \gamma>1, \exists C>0$, so that $\forall u\in G, r\in\mathbb{N}$
\begin{equation}
S_{u}(r)< C\gamma^r,
\end{equation}
we get the following result:
\begin{theorem}\label{thm:fullresultsubexp}
Let $G$ be an infinite graph of uniform sub-exponential growth and let $H$ be a bounded Schr\"odinger operator $\ell^2(G)$. Then,
\begin{equation}
\sigma_{\text ess}(H)=\bigcup_{H'\in\mathcal{R}(H)}\sigma_\infty(H')=\bigcup_{H'\in\mathcal{R}(H)}\sigma(H').
\end{equation}
\end{theorem}
\begin{proof}
We already know from Theorems~\ref{thm:Thm1} and~\ref{thm:sigmainfty} that
\[
\bigcup_{H'\in\mathcal{R}(H)}\sigma(H')\subseteq\sigma_{\text ess}(H)\subseteq\bigcup_{H'\in\mathcal{R}(H)}\sigma_\infty(H').
\]
Thus, showing that, for any $H'\in\mathcal{R}(H)$,
\[
\sigma_\infty(H')\subseteq\sigma(H')
\]
will complete the proof. Actually, since we assume uniform sub-exponential growth also any $\mathcal{R}$-limit $(H',G')$ of $(H,G)$ is of (uniform) sub-exponential growth (for any $r>0$, the sphere $\mathcal{S}_{v_0'}(r)$ has a corresponding sphere in the original graph which satisfies \eqref{eq:subexpgr}, and thus also $S_{v_0'}(r)<C\gamma^r$).
In this case $e^{-\alpha|\cdot|}\in\ell^2(G')$ for any $\alpha>0$. Thus, given a  bounded generalized eigenfunction $\varphi$, satisfying $H'\varphi=\lambda'\varphi$, the conditions of Proposition~\ref{prop:HK4.8} are satisfied for $\varphi$, and we get that $\lambda\in\sigma(H')$.
\end{proof}

\begin{remark} On graphs of exponential growth
\[
\sigma_{\text{ess}}\left(H\right)\nsupseteq\bigcup_{H'\in\mathcal{R}(H)}\sigma_{\infty}\left(H'\right).
\]
A simple example is $H=\Delta$, the adjacency operator on the $d$-regular
tree. The only $\mathcal{R}$-limit is the same operator $(H'=\Delta$ on $T_d)$, for which, e.g.\ by taking the constant function $\psi\equiv1$, we get that $d\in\sigma_{\infty}(\Delta)$. Thus \[
d\in\sigma_{\infty}\left(\Delta\right)\backslash\sigma_{\text{ess}}\left(\Delta\right).
\]
In fact, $[-d,d]\subset\sigma_{\infty}(\Delta)$ (it follows e.g.\ from \cite[Equation~(7.7)]{MoharWoess}), and thus
\[
\left[-d,d\right]\big\backslash\left[-2\sqrt{d-1},2\sqrt{d-2}\right]\subset\sigma_{\infty}\left(\Delta\right)\backslash\sigma_{\text{ess}}\left(\Delta\right).
\] 
\end{remark}
\begin{remark}
As discussed in Chapter~2 the uniform growth of $G$ is crucial. See the example in Section~\ref{sec:counterexample}.
\end{remark}
\begin{remark}
The above mentioned example does not contradict \eqref{eq:essubsetinfty} on general graphs. It is still an open problem whether this inclusion holds or not on graphs of exponential growth.
\end{remark}
\begin{remark}
The set of graphs of uniform sub-exponential growth contains many different examples. For instance, one class of such graphs is the set of (infinite) penny graphs, i.e.\ graphs whose vertices can be represented by unit circles, with no two of these circles crossing each other, and with two adjacent vertices if and only if they are represented by tangent circles. Similarly one can consider any planar graph in which each vertex can be represented by a closed shape such that its girth and the surrounded area are bounded below by a constant greater than zero.
\end{remark}
\begin{remark}
By a direct adaptation everything holds also for Jacobi operators on the graph as in (1.2). For simplicity we stay here with Schr\"odinger operators.
\end{remark}

\section{Characterizing $\sigma_{\text ess}(H)$ using generalized eigenfunctions}

\subsection{Notations}
Throughout this section it will be more convenient to use notations different than the notations given in Chapter~1 for the definition of $\mathcal{R}$-limits. Assume $\left(H',G',v_0'\right)$ is an $\mathcal{R}$-limit of $(H,G,v_0)$, let $\left\{v_n\right\}_{n\in\mathbb{N}}$ and $\left\{n_j\right\}_{j\in\mathbb{N}}$ be  a path to infinity and the corresponding subsequence of indices of vertices in $G$ along which $\left(H',G'\right)$ is obtained (as in Definition~1.2.2), and let $\left\{\mathcal{I}_r^{(j)}\right\}_{r,j\in\mathbb{N}}$ and $\left\{\mathcal{I}'_r\right\}_{r,j\in\mathbb{N}}$ be the corresponding sequences of isomorphisms. Denote the corresponding indexing of vertices by $\eta_r^{(j)}$ and $\eta_r'$.
First, we define $u_j=v_{n_j}$, and use the sequence $\{u_n\}_{n\in\mathbb{N}}$ instead of $\{v_n\}_{n\in\mathbb{N}}$. In this case we will say that the $\mathcal{R}$-limit is obtained along the sequence of vertices $\left\{u_j\right\}_{j\in\mathbb{N}}$.
Define another mapping $f_{n,r}: B_r\left(u_n\right)\to B_r\left(v_0'\right)$ by
\[
f_{n,r}=\eta_r'^{-1}\circ\eta_r^{(n)},
\]
and let $\mathcal{I}_{f_{n,r}}$ be the corresponding unitary operator.
Also, recall the definition,
\[ H_{u,r} = H|_{B_r(u)}. \]
Then we can write the condition for an $\mathcal{R}$-limit \eqref{def:rlimdef} in an equivalent way:
\begin{equation}\label{eq:altRlimdef}
\forall r\in\mathbb{N},\,\,\left\Vert \mathcal{I}_{f_{n,r}}H_{u_{n},r}\mathcal{I}_{f_{n,r}}^{-1}-H'_{v_{0}',r}\right\Vert \to 0.
\end{equation}
Notice that the set $\{f_{n,r}\}_{n,r\in\mathbb{N}}$ is (and required to be) coherent in the sense that for any fixed $n$ and $r'>r$, the actions of $f_{n,r'}$ and $f_{n,r}$ on $B_r(u_n)$ are identical.

Additionally we introduce the notion of limit of a sequence of rooted graphs, as follows:
\begin{definition} \label{def:graphsconv}
Let $\{G_k,u_k\}_{k\in\mathbb{N}}$ be a sequence of rooted graphs, and let $\left\{f_{k,r}:B_r(u_k)\to B_r\left(v_0'\right)\right\}_{k,r\in\mathbb{N}}$ be a set of (coherent) maps.
We say that the sequence $\{G_k,u_k\}_{k\in\mathbb{N}}$ converge to a rooted graph $(G',v_0')$ (with respect to the set of maps $\left\{f_{k,r}\right\}_{k,r\in\mathbb{N}}$) if, for any $R>0$ there exists $K\in\mathbb{N}$, such that for any $k>K$
\[
f_{k,r}\left(B_R({u}_k)\right)=B_R(v_0').
\]
\end{definition}

\subsection{General observations}

The proof of Theorem~4.1 requires some preliminary propositions for which we don't have to impose a restriction on the growth of the graph. The argument follows the proof of Theorem~(7.2.1) of Simon \cite{SimonSz}, which concerns operators on $\ell^{2}\left(\mathbb{N}\right)$, with special attention and adaptations to the more complex case of operators on graphs.

Let $G$ be a graph of bounded degree, and $H=\Delta+Q$ be a bounded Schr\"odinger operator on $\ell^2(G)$.

\begin{prop} \label{prop:RRsubsetR}
$\mathcal{R}\left(\mathcal{R}\left(H\right)\right)\subseteq\mathcal{R}\left(H\right)$.
\end{prop}
\begin{proof}
Assume:
\begin{enumerate}
\item[(1)] The operator $\left(H',G',v_0'\right)$ is an $\mathcal{R}$-limit of $\left(H,G,v_0\right)$
along the sequence $\left\{ u_{n}\right\} _{n=1}^{\infty}\subset G$
and the sequence of unitary maps $\left\{ f_{n,r}\right\} _{n,r\in\mathbb{N}}$.
\item[(2)] The operator $\left(\widetilde{H},\widetilde{G},\widetilde{v}_0\right)$ is
an $\mathcal{R}$-limit of $\left(H',G',v_0'\right)$ along the sequence
$\left\{ u_{n}'\right\} _{n=1}^{\infty}\subset G'$ and the sequence of unitary maps $\left\{ f_{n,r}'\right\} _{n,r\in\mathbb{N}}$.
\end{enumerate}
For any vertex $v$ (either in $G$ or $G'$ or $\widetilde{G}$) denote by $|v|$ the distance from  $v$ to the root (correspondingly $v_0$ or $v_0'$ or $\widetilde{v}_0$). Using the assumption (1) above, for any $n\in\mathbb{N}$ we can pick $k(n)$ so that $\left|u_{k(n)}\right|>2\left|u_n'\right|$ and
\[
\left\Vert I_{f_{k(n),R_n}}H_{u_{k(n)},R_n}I_{f_{k(n),R_n}}^{-1}-H'_{v_{0}',R_n}\right\Vert < \frac{1}{n},
\]
where $R_n=n+\left|u_n'\right|$ (and $\left|u_n'\right|=\text{dist}\left(u_n',v_0'\right)$). Denote for any $n\in\mathbb{N}$, $w_n=f_{R_n,k(n)}^{-1}\left(u_n'\right)$. Note that $w_n\in B_{\left|u_n'\right|}\left(u_{k(n)}\right)$, and thus $\left|w_n\right|>\left|u_n'\right|\to\infty$.
Let $R>0$ and $\varepsilon>0$. Using assumption (2), there exists $n_0\in\mathbb{N}$ such that for any $n>n_0$,
\[
\left\Vert I_{f_{n,R}'}H'_{u_n',R} I_{f_{n,R}'}^{-1}-\widetilde{H}{}_{\widetilde{v_{0}},R}\right\Vert < \varepsilon.
\]
We take $n>\max\left\{R,\frac{1}{\varepsilon},n_0 \right\}$. Then, since $B_R\left(w_n\right)\subset B_{R_n}\left(u_{k(n)}\right)$,
\begin{eqnarray*}
&\left\Vert I_{f_{n,R}'}I_{f_{k(n),R_n}}H_{w_n,R}I_{f_{k(n),R_n}}^{-1}I_{f_{n,R}'}^{-1}-\widetilde{H}{}_{\widetilde{v_{0}},R}\right\Vert \leq
\\
&\left\Vert I_{f_{n,R}'}I_{f_{k(n),R_n}}H_{w_n,R}I_{f_{k(n),R_n}}^{-1}I_{f_{n,R}'}^{-1}-I_{f_{n,R}'}H'_{u_n',R} I_{f_{n,R}'}^{-1}\right\Vert + \\
&+\left\Vert I_{f_{n,R}'}H'_{u_n',R} I_{f_{n,R}'}^{-1}-\widetilde{H}{}_{\widetilde{v_{0}},R}\right\Vert < \\
&\left\Vert I_{f_{k(n),R_n}}H_{u_{k(n)},R_n}I_{f_{k(n),R_n}}^{-1}-H'_{v_{0}',R_n}\right\Vert + \varepsilon < 2\varepsilon.
\end{eqnarray*}
Thus, $\widetilde{H}$ is an $\mathcal{R}$-limit of $H$ along the sequence $\left\{w_n\right\}_{n=1}^\infty$ and the sequence of unitary maps $\left\{g_{n,R}=f'_{n,R}\circ f_{k(n),R_n}\right\}_{n,R\in\mathbb{N}}$.

\end{proof}
\smallskip{}

\begin{prop}\label{prop:Rlimseq}
Assume $\left\{ \left(L_{k},G_k\right)\right\}_{k=1}^\infty \subset\mathcal{R}\left(H\right)$ is
a sequence of $\mathcal{R}$-limits of $\left(H,G\right)$, and assume there exists a sequence of generalized eigenfunctions $\varphi^{\left(k\right)}$
of $L_{k}$, satisfying $L_{k}\varphi^{\left(k\right)}=\lambda_{k}\varphi^{\left(k\right)}$,
and $\lambda_{k}\to\lambda'$. Assume further there exists
a sequence of vertices $u_{k}\in G_{k}$, such that
\[
\max_{u\in B_{k}\left(u_{k}\right)}\left|\varphi^{\left(k\right)}\left(u\right)\right|\leq C\cdot\left|\varphi^{\left(k\right)}\left(u_{k}\right)\right|\neq0,
\]
for some constant $C>0$. Then there exists an $\mathcal{R}$-limit
$\left(H',G'\right)\in\mathcal{R}\left(H\right)$, and a bounded nonzero
function $0\neq\varphi'\in\ell^{\infty}\left(G'\right)$,
satisfying $H'\varphi'=\lambda'\varphi'$.
\end{prop}

\begin{proof}
We first consider the sequence $\{G_k,u_k\}_{k\in\mathbb{N}}$ of rooted graphs.
Since the degree of $G$ is bounded, the set of graphs corresponding to $\mathcal{R}$-limits on $G$ is compact. Thus, there exists a subsequence $\left\{\widehat{G}_k,\widehat{u}_k\right\}_{k\in\mathbb{N}}\subseteq\{G_k,u_k\}_{k\in\mathbb{N}}$, a corresponding set of (coherent) maps $\left\{f_{k,R}\right\}_{k,R\in\mathbb{N}}$ and a rooted graph denoted by $(G',v_0')$, such that
\[
(\widehat{G}_k,\widehat{u}_k)\to(G',v_0')
\]
in the sense of Definition~\ref{def:graphsconv}.
Of course $f_{k,R}(\widehat{u}_k)=v_0'$.
Denote by $\left\{\widehat\varphi^{(k)}\right\}$ and $\left\{\widehat L_k\right\}$ the corresponding subsequences of  $\left\{\varphi^{(k)}\right\}$ and $\left\{L_k\right\}$. Define $\psi^{(k)}:B_k\left(v_0'\right)\to\mathbb{C}$ by
\[
\psi^{\left(k\right)}\left(u\right)=\frac{\widehat\varphi^{\left(k\right)}\left(f_{k,k}^{-1}u\right)}{\widehat\varphi^{\left(k\right)}\left(\widehat{u}_{k}\right).}
\]
Then,  $\psi^{\left(k\right)}\left(v_{0}'\right)=1$, and
\[
\sup_{u\in B_{k}\left(v_0'\right)}\left|\psi^{\left(k\right)}\left(u\right)\right|\leq C.
\]
Define further ${L}_k'=\mathcal{I}_{f_{k,k}}^{-1}\widehat L_{k}\mathcal{I}_{f_{k,k}}$. By compactness (and since $\psi^{(k)}$ is bounded in a neighbourhood of $v_0'$)  there
exists a sequence of indices $k(\ell):\mathbb{N}\to\mathbb{N}$ such that both $\psi^{\left(k(\ell)\right)}$ and $\widehat L_{k(\ell)}$ converge to limit objects $\varphi':G'\to\mathbb{C}$ and $L':\ell^2(G')\to\ell^2(G)$, in the sense that
\[
\left\Vert\left(\psi^{\left(k\left(\ell\right)\right)}-\varphi'\right)\Big\vert_{B_\ell\left(v_0'\right)}\right\Vert\to0,
\]
and
\[
\left\Vert \left(L_{k(\ell)}'\right)_{v_0',\ell} -\left(L'\right)_{v_0',\ell}\right\Vert \to0.
\]
Since $L_k'$ is given by an isomorphism of $\widehat{L}_k$ (which is an $\mathcal{R}$-limit of $H$) restricted to a ball, it is also a restriction of an $\mathcal{R}$-limit of $H$ (we can expand $f_{k,k}$ to an isomorphism of the full operator to get another $\mathcal{R}$-limit). Moreover $L'$ is also an $\mathcal{R}$-limit of $H$ as a limit of restrictions of $\mathcal{R}$-limits, by an argument similar to the argument in the proof of Proposition~\ref{prop:RRsubsetR}.
Now,
\begin{eqnarray*}
&\left\Vert L'\varphi'-\lambda\varphi'\right\Vert=\lim_{\ell\to\infty}\left\Vert\left(L'\varphi'-\lambda\varphi'\right)\Big\vert_{B_{\ell}\left(v_0'\right)}\right\Vert \leq \\
&\leq\lim_{\ell\to\infty}\bigg(\left\Vert\left(L'\varphi'-L'_{k(\ell)}\varphi'\right)\big\vert_{B_{\ell}\left(v_0'\right)}\right\Vert+
\left\Vert L'_{k(\ell)}\left(\varphi'-\psi^{(k(\ell))}\right)\big\vert_{B_\ell\left(v_0'\right)}\right\Vert+\\
&+\left\Vert\left(L'_{k(\ell)}\psi^{(k(\ell))}-\lambda_{k(\ell)}\psi^{(k(\ell))}\right)\big\vert_{B_{\ell}\left(v_0'\right)}\right\Vert+
\left\Vert\left(\lambda_{k(\ell)}\psi^{(k(\ell))}-\lambda_{k(\ell)}\varphi'\right)\big\vert_{B_{\ell}\left(v_0'\right)}\right\Vert+\\
&\left\Vert\left(\lambda_{k(\ell)}\varphi'-\lambda\varphi'\right)\big\vert_{B_{\ell}\left(v_0'\right)}\right\Vert\bigg)\to0,
\end{eqnarray*}
and thus $L'\varphi=\lambda\varphi'$. Additionally $\left\Vert \varphi'\right\Vert _{\infty}\leq C$,
so $\lambda\in\sigma_{\infty}\left(L'\right)$.
\end{proof}
\smallskip{}

\begin{corollary*}
The set $\bigcup_{L\in\mathcal{R}(H)}\sigma_\infty(L)$ is closed.
\end{corollary*}
\begin{proof}
Assume $\left\{\lambda_n\right\}_{n\in\mathbb{N}}\subset\bigcup_{L\in\mathcal{R}(H)}\sigma_\infty(L)$ and $\lambda_n\to\lambda'$. For any $n\in\mathbb{N}$, let $\left(L_n,G_n\right)\in \mathcal{R}(H)$ and $0\neq\varphi_n\in \ell^\infty(G_n)$ such that $(L_n-\lambda_n)\varphi_n=0$. Denote $u_n\in G_n$ so that
$\varphi_n(u_n)\geq\frac{\left\Vert \varphi_n\right\Vert_\infty}{2}.$
The conditions of Proposition~\ref{prop:Rlimseq} are satisfied with these sequences (and $C=2$) and we get that $\lambda'\in \bigcup_{L\in\mathcal{R}(H)}\sigma_\infty(L)$.
\end{proof}

The next two propositions rely on Proposition~\ref{prop:Rlimseq}. We show that under certain conditions the existence of a generalized eigenfunction of $H$ results in the existence of a bounded generalized eigenfunction of some $\mathcal{R}$-limit of $H$. The first proposition treats bounded generalized eigenfunctions, while in the second proposition they are unbounded.

\begin{prop} \label{prop:bounded}
Assume $\varphi\neq0$ is a bounded generalized eigenfunction of $H$, satisfying $H\varphi=\lambda\varphi$. Then one of the following is satisfied:
\begin{enumerate}
\item[(i)] There exists an $\mathcal{R}$-limit $(H',G')$ of $H$, and a bounded nonzero
function $0\neq\varphi'\in\ell^{\infty}\left(G'\right)$, solving
$H'\varphi'=\lambda\varphi'$.
\item[(ii)] There exist constants $\gamma>1,C>0$ so that $\forall u\in G$, $|\varphi\left(u\right)|\leq C\cdot \gamma^{-\left|u\right|}$.
\end{enumerate}
\end{prop}

\begin{prop} \label{prop:unbounded}
Assume $\varphi$ is an unbounded generalized eigenfunction of $H$, satisfying $H\varphi=\lambda\varphi$. Then one of the following is satisfied:
\begin{enumerate}
\item[(i)] There exists an $\mathcal{R}$-limit $(H',G')$ of $H$, and a bounded nonzero
function $0\neq\varphi'\in\ell^{\infty}\left(G'\right)$, solving
$H'\varphi'=\lambda\varphi'$.
\item[(ii)] There exists a constant $\gamma>1$, so that, $\forall C>0$, $\exists u\in G$, satisfying $|\varphi\left(u\right)|\geq C\cdot \gamma^{\left|u\right|}$.
\end{enumerate}
\end{prop}

We start with the proof of Proposition~\ref{prop:bounded}, and first present a lemma which will be useful.

\begin{lemma} \label{lem:bounded}
Let $\{a_n\}_{n=1}^\infty$ be a sequence of positive numbers and let $r>1$. Assume
\[\lim _{n\to\infty} a_n=0.\]
Then either there exists $C>0$ such that for any $n\in\mathbb{N}$
\begin{equation} \label{eq:expdec}
a_{n}<C r^{-n}.
\end{equation}
or there exists a subsequence $\left\{a_{n_k}\right\}_{k=1}^\infty$ such that for any $k\in\mathbb{N}$
\begin{equation} \label{eq:decsubseq}
a_{n_k}>a_{n_k+1} \text{ and } r a_{n_k}\geq a_{n_k-1}.
\end{equation}
\end{lemma}

We will prove this lemma after proving the proposition.

\begin{proof}[Proof of Proposition~\ref{prop:bounded}]
Fix $k\in\mathbb{N}\backslash\{1\}$ and define
\[
q_{m}=q_{m}^{\left(k\right)}=\max_{m\left(k-1\right)\leq\left|u\right|<mk}\left|\varphi\left(u\right)\right|.
\]
Choose a vertex satisfying $m\left(k-1\right)\leq\left|u_{m}\right|<m k$ and
\[
\left\vert\varphi\left(u_m \right)\right\vert=\max_{m\left(k-1\right)\leq\left|u\right|<m k}\left|\varphi\left(u\right)\right|=q_{m}
\]
and  denote it by $u_m$.

If $q_{m}\nrightarrow0$, let $\left\{q_{m_\ell}\right\}_{\ell\in\mathbb{N}}$ be a subsequence such that $q_{m_\ell}\to q'>0$.
Take $\ell_0\in\mathbb{N}$ so that $q_{m_\ell}>q'/2$ for every $\ell>\ell_0$. Then,
\[
\max_{u\in B_{\ell}\left(u_{m_\ell}\right)}\left|\varphi\left(u\right)\right|\leq \left\Vert \varphi \right\Vert_\infty \leq\frac{2\left\Vert \varphi \right\Vert_\infty}{q'}\cdot\left|\varphi\left(u_{m_\ell}\right)\right|\neq0.
\]
We can now repeat the compactness argument described in the proof of Proposition~\ref{prop:Rlimseq} to get an $\mathcal{R}$-limit (along a subsequence of $\{u_{m_\ell}\}_{l\in\mathbb{N}}$), and a corresponding generalized eigenfunction defined on it  which is bounded and nonzero.

Assume now that $\lim_{m\to\infty}q_{m}=0$,
and apply Lemma~\ref{lem:bounded} with this sequence and some fixed $r>1$ (independent of $k$). Notice that if for some $k\in\mathbb{N}\backslash\{1\}$, there exists $C>0$ such that for any $n\in\mathbb{N}$
\[
q_m^{(k)}<Cr^{-n},
\]
then, for any $u\in G$
\[
\left|\varphi(u)\right|<Cr^{-|u|/(k-1)},
\]
and $\varphi$ is exponentially decaying as in (ii) with $\gamma=r^{1/(k-1)}>1$.
Otherwise, 
we get (for any fixed $k\in\mathbb{N}$) a sequence $\left\{q_{m_i}\right\}_{i=1}^\infty$ satisfying
\[
q_{m_i}>q_{m_i+1} \text{ and } r q_{m_i}\geq q_{m_i-1}.
\]
Denote $\psi_i(u)=\frac{\varphi\left(u\right)}{\left\vert\varphi\left(u_{m_i}\right)\right\vert}$.
By compactness we can choose (for each $k$) a subsequence of $\left\{m_i\right\}_{i\in\mathbb{N}}$  (which we denote again by $\left\{m_i\right\}_{i\in\mathbb{N}}$) and a set of unitary maps $\left\{f_{i,R}\right\}_{i,R\in\mathbb{N}}$, so that both $H_{u_{m_i},i}$ and $\mathcal{I}_{f_{i,R}}\psi_i$ converge to an $\mathcal{R}$-limit $\left(L^{\left(k\right)},G_k, v_0^{(k)}\right)$ and a generalized eigenfunction $\varphi^{\left(k\right)}$ defined on $G_k$. Moreover, $\varphi^{\left(k\right)}$ satisfies both
\begin{equation*}
\max_{u\in B_{k}\left(v_0^{(k)}\right)}\left|\varphi^{\left(k\right)}\left(u\right)\right|\leq r \cdot\left|\varphi^{\left(k\right)}\left(v_0^{(k)}\right)\right|,
\end{equation*}
and $L^{\left(k\right)}\varphi^{\left(k\right)}=\lambda\varphi^{\left(k\right)}$.
We now use {Proposition~\ref{prop:Rlimseq}} with the sequences $\left\{L^{(k)}\right\}_{k\in\mathbb{N}}$, $\left\{\varphi^{(k)}\right\}_{k\in\mathbb{N}}$ and $\lambda_k\equiv\lambda$ to complete the proof and get a bounded generalized eigenfunction of some $\mathcal{R}$-limit of $H$.
\end{proof}

\begin{proof}[Proof of Lemma~\ref{lem:bounded}]
Define the set,
\[
N_1=\left\{ n\in\mathbb{N}\,|\,a_{n}\geq\max_{k>n}a_{k}\right\}.
\]
Since $a_n\to0$, $N_1$ is infinite. Define further,
\[
N_2=N_1\cap\left\{ n\in\mathbb{N}\,|\,a_{n-1}\leq r a_{n}\right\}.
\]
If $N_2$ is infinite we get a subsequence satisfying \eqref{eq:decsubseq}. Otherwise, there exists $n_0\in\mathbb{N}$ so that for any $n_0<n\in N_1$
\[
a_n < r a_n < a_{n-1}.
\]
Thus also $a_{n-1}\in N_1$, and we get that $a_{n-2}>r a_{n-1}>r^2 a_n$. Recursively for any natural $k< n-n_0$, we get that $n-k>n_0\in N_1$ and
\[
a_{n-k}> r^k a_n.
\]
Since $N_2$ is finite we conclude that $\left\{a_n\right\}_{n\in\mathbb{N}}$ satisfies exponential decay \eqref{eq:expdec},
\[
a_n < a_{n_0} r^{n_0} r^{-n}.
\]
\end{proof}

Next we prove Proposition~\ref{prop:unbounded}, and again begin by presenting a useful lemma.

\begin{lemma} \label{lem:unbounded}
Let $\{a_n\}_{n=1}^\infty$ be an unbounded sequence of real numbers. Assume the existence of $r>1,\ C>0$ such that $\forall n\in\mathbb{N}$,
\begin{equation}
\label{eq:expbound}
a_n<C r^n.
\end{equation}
Then, for any $\gamma>r$ there exists a subsequence $\left\{a_{n_k}\right\}_{k=1}^\infty$ such that for any $k\in\mathbb{N}$ both
\[
a_{n_k}>a_{n_k-1} \text{ and } \gamma a_{n_k}\geq a_{n_k+1}.
\]
\end{lemma}

\begin{proof}[Proof of Proposition~\ref{prop:unbounded}]
Fix $k\in\mathbb{N}$ and define for every $m\in\mathbb{N}$
\begin{equation*}
q_{m}=q_{m}^{\left(k\right)}=\max_{m\left(k-1\right)\leq\left|u\right|<mk}\left|\varphi\left(u\right)\right|.
\end{equation*}
Due to the assumption $q_{m}$ is unbounded. We proceed by assuming that (ii) is not satisfied and proving (i). By the assumption for all $\gamma>1$, $\exists C>0$ so that for all $u\in G$ we have $\left|\varphi(u)\right|<C\cdot \gamma^{|u|}$. Fix some $r>1$, and take $\gamma=r^{1/k}$. We get that, for every $k,m\in\mathbb{N}$,
\[q_m^{(k)}<C_{(k)}\cdot\gamma^{mk}=C_{(k)}\cdot r^m.\]
Fix $k$ and note that by Lemma~\ref{lem:unbounded} we get a subsequence $\left\{q_{m_i}\right\}_{i\in\mathbb{N}}$, on which, for every $\gamma'>r$,
\[
q_{m_i}>q_{m_i-1} \text{ and } \gamma' q_{m_i}\geq q_{m_i+1}.
\]

Choose a vertex satisfying $\left(m_i-1\right)k\leq\left|u_i\right|<m_i k$ and
\[
\left\vert\varphi\left(u_i\right)\right\vert=\max_{\left(m_i-1\right)k\leq\left|u\right|<m_i k}\left|\varphi\left(u\right)\right|\neq0
\]
and denote it by $u_i$.
Denote $\psi_i(u)=\frac{\varphi\left(u\right)}{\left\vert\varphi\left(u_i\right)\right\vert}$.
By compactness we can choose (for each $k$) sub-sequences of $\left\{m_i\right\}_{i\in\mathbb{N}}$ and $\left\{u_i\right\}_{i\in\mathbb{N}}$ (which we denote again by $\left\{m_i\right\}_{i\in\mathbb{N}}$ and $\left\{u_i\right\}_{i\in\mathbb{N}}$) and a set of unitary maps $\left\{f_{i,R}\right\}_{i,R\in\mathbb{N}}$,
so that both $H_{u_i,i}$ and $\mathcal{I}_{f_{i,R}}\psi_i$ converge to an $\mathcal{R}$-limit $\left(L^{\left(k\right)},G_k,v_0^{(k)}\right)$ and a generalized eigenfunction $\varphi^{\left(k\right)}$ defined on $G_k$. Moreover, $\varphi^{\left(k\right)}$ satisfies,
\begin{equation*}
\max_{u\in B_{k}\left(v_0^{(k)}\right)}\left|\varphi^{\left(k\right)}\left(u\right)\right|\leq \gamma'\cdot\left|\varphi^{\left(k\right)}\left(\widetilde{v}_0\right)\right|,
\end{equation*}
We now use {Proposition~\ref{prop:Rlimseq}} with the sequences $\left\{L^{(k)}\right\}_{k\in\mathbb{N}}$, $\left\{\varphi^{(k)}\right\}_{k\in\mathbb{N}}$ and $\lambda_k\equiv\lambda$ to complete the proof and get a bounded generalized eigenfunction of some $\mathcal{R}$-limit of $H$.
\end{proof}

\begin{proof}[Proof of Lemma~\ref{lem:unbounded}]
Define,
\[
N_1=\left\{ n\in\mathbb{N}\,|\,a_{n}\geq\max_{m\leq n}a_{m}\right\} .
\]
Since $a_{n}$ is unbounded,  $\left|N_1\right|=\infty$.
Define further,
\[
N_2=N_1\cap\left\{ n\in\mathbb{N}\,|\,a_{n+1}\leq \gamma a_{n}\right\}.
\]
If $\left|N_2\right|<\infty$, there exists an $n_{0}$ so that for
every $n\geq n_{0}$, $a_{n+1}> \gamma a_{n}$ and in particular $a_{n+1}\geq a_n \geq max_{m\leq n} a_m$. Thus $n+1\in N_1$ and $a_{n+2}>\gamma a_{n+1}>\gamma^2 a_n$, thus also $n+2\in N_1$.
Inductively, for every $n>n_{0}$, $n\in N_1$ and $a_{n}>\gamma^{n-n_{0}}a_{n_{0}}$
contradicting the assumption that $a_n<C r^n$. Thus $\left|N_2\right|=\infty$.
\end{proof}

\subsection{Proof of Theorem~\ref{thm:sigmainfty}}

The proof of Theorem~\ref{thm:sigmainfty} is based on Propositions \ref{prop:bounded} and \ref{prop:unbounded}. Given $\lambda\in\sigma_{\text ess}(H)$ we use the reverse Shnol's property (Theorem~\ref{thm:revSch}) to obtain a sequence of orthogonal generalized eigenfunctions $\left\{\varphi^{(n)}\right\}_{n=1}^\infty$ satisfying
\begin{equation} \label{eq:gefseq}
H\varphi^{(n)}=\lambda_n\varphi^{(n)},
\end{equation}
and $\lambda_n\to\lambda$.
Next, we will have to use the growth property of the graph $G$ to treat the second cases in Propositions~\ref{prop:bounded},~\ref{prop:unbounded} and complete the proof of the theorem.

\begin{proof}[Proof of Theorem~\ref{thm:sigmainfty}]
Let $\lambda\in\sigma_{\text ess}(H)$. Then, either
\begin{itemize}
\item[(a)] There exists a sequence $\left\{\lambda_n\right\}_{n=1}^\infty\subset\sigma(H)$ so that $\lambda_n\to\lambda$. In this case we get from Theorem~\ref{thm:revSch} a sequence of generalized eigenfunctions satisfying both \eqref{eq:gefseq} and \eqref{eq:graphefGrowth}.
\item[(b)] It is an eigenvalue of infinite multiplicity. Thus we have a sequence of orthonormal eigenfunctions satisfying \eqref{eq:gefseq} with $\lambda_n\equiv\lambda$.
\end{itemize}

Notice that by Proposition~\ref{prop:Rlimseq} it is enough to show the existence of a sequence of bounded generalized eigenfunctions of $\mathcal{R}$-limits of $H$ with corresponding ``eigenvalues" converging to $\lambda$.

We begin with case (a) and assume that infinitely many $\varphi^{(n)}$ are unbounded. Since the graph $G$ is of sub-exponential growth, and by \eqref{eq:graphefGrowth}, we get that for any $n\in\mathbb{N}$ also $\varphi^{(n)}$ is sub-exponentially bounded. Thus for any such $\varphi^{(n)}$ the situation (ii) of Proposition~\ref{prop:unbounded} is impossible. Consequently we get a corresponding sequence of bounded generalized eigenfunctions of $\mathcal{R}$-limits of $H$ with "eigenvalues" converging to $\lambda$. Thus, by Proposition~\ref{prop:Rlimseq} there exists an $\mathcal{R}$-limit $(H',G')$ of $H$ so that $\lambda\in\sigma_{\infty}(H')$.

Otherwise, either if only a finite number of $\varphi^{(n)}$ are unbounded or in case (b),  $\varphi^{(n)}\in\ell^\infty(G)$ for any $n>n_0$ and we can use Proposition~\ref{prop:bounded}. If the situation (i) of this proposition occurs infinity many times we get (again using Proposition~\ref{prop:Rlimseq}) a bounded generalized eigenfunction of some $\mathcal{R}$-limit.
The last possible case is: the circumstance (ii) of Proposition~\ref{prop:bounded} is satisfied for any ${n>n_0}$. In this case, since $G$ is of sub-exponential growth, $\varphi^{(n)}$ is an eigenfunction, and we can assume $\left\Vert\varphi^{(n)}\right\Vert_2=1$.
For any $n,m\in\mathbb{N}$ denote
\[p^{(n)}_m=\max_{|u|=m}{\left\vert\varphi^{(n)}(u)\right\vert},\]
and choose a vertex $u^{(n)}_m$ such that $\left\vert u^{(n)}_m\right\vert=m$ and $\left\vert\varphi^{(n)}\left( u^{(n)}_m\right)\right\vert=p^{(n)}_m$. Since $\varphi^{(n)}\in\ell^2(G)$, $\lim_{m\to\infty}p^{(n)}_m=0$ so there exists some $m_n\in\mathbb{N}$ such that $p^{(n)}_{m_n}=\left\Vert\varphi^{(n)}\right\Vert_\infty$, and additionally $p^{(n)}_{m_n}>p^{(n)}_{m}$ for every $m>m_n$.

Now, if $m_n\to\infty$ we can repeat the compactness argument with $\left\{\nicefrac{\varphi^{(n)}}{p_{m_n}^{(n)}}\right\}_{n\in\mathbb{N}}$ along $\left\{u_{m_n}^{(n)}\right\}_{n\in\mathbb{N}}$ to get an $\mathcal{R}$-limit $(H',G')$ with $\lambda\in\sigma_\infty(H')$.
Otherwise we can assume that ${M=\sup_n m_n<\infty}$. We denote
\[C_n=\max_{m\leq M} p_m^{(n)}=\left\Vert\varphi^{(n)}\right\Vert_\infty.\]
Notice that the sequence $\left\{\varphi^{(n)}\right\}$ is a sequence of $\left|\lambda-\lambda_n\right|$-approximate eigenfunctions for $\lambda$. Thus, since $\lambda\in\sigma_{\text ess}(H)$, the eigenfunctions $\varphi^{(n)}$ should converge weekly to zero and consequently $C_n\to0$.
We can now repeat the argument of the proof of Proposition~\ref{prop:bounded} to conclude the existence of an $\mathcal{R}$-limit as required, unless there exist some $M_1, N_1\in\mathbb{N}$ and $\gamma>1$ so that for any $n>N_1, m>M_1$
\[
p_m^{(n)}\leq C_n \gamma^{-(m-M_1)}.
\]
In this case we get by the sub-exponential growth of $G$ that
\[
\left\Vert\varphi^{(n)}\right\Vert_2\leq\sum_{m\geq M_1}\left|p^{(n)}_m\right|^2 S_{v_0}(m)+ M_1 C_n \to 0,
\]
in contradiction to ${\left\Vert\varphi^{(n)}\right\Vert_2=1}$.
\end{proof}

\section[A growth condition for generalized eigenfunctions]{A growth condition for generalized eigenfunctions on graphs}\label{sec:RSproperty}

This section is devoted to investigating the possible generalization of what we call reverse Shnol's Theorem to graphs.
We shall prove Theorem~\ref{thm:revSch}, which we repeat here:

\begin{namedtheorem}[Theorem~\ref{thm:revSch}] Let $(G,v_0)$ be a rooted $($infinite$)$ graph of bounded degree and $H$ a bounded Schr\"odinger operator on $\ell^{2}\left(G\right)$. Assume $\omega\in\ell^2(G)$ is real and positive $($i.e.\ $\omega(v)>0\,\, \forall v\in G)$. Let $\mu$ be a spectral measure for $H$. Then for $\mu$-a.e.\ $\lambda\in\sigma\left(H\right)$ there exists a generalized eigenfunction $\varphi=\varphi\left(v\right)$ satisfying $H\varphi=\lambda\varphi$,
and additionally
\[\varphi(\cdot)\omega(\cdot)\in\ell^2(G).\]
\end{namedtheorem}

We proceed with the proof. We begin by citing a proposition from Poerschke-Stolz-Weidmann \cite{PoStWe89} that will be essential for the proof.
The setting includes a separable Hilbert space $\mathcal{H}$ and a selfadjoint operator $T$ in $\mathcal{H}$ with $T\geq 1$. Denote by $D(T)$ the domain of $T$, and by $\mathcal{H}_+(T)$ the Hilbert space $D(T)$ with the inner product $\langle u,v\rangle_+=\langle Tu,Tv\rangle$. Define further $\mathcal{H}_-(T)$ to be the completion of $\mathcal{H}$ with the inner product $\langle f,g\rangle_-=\langle T^{-1}f,T^{-1}g\rangle$ (defined since $T\geq 1$ and thus $T^{-1}>0$).

Given a selfadjoint operator $K$ on $\mathcal{H}$, and a spectral measure $\mu$ for $K$, then (see e.g., \cite[Lemma~2]{PoStWe89}) there exists a $\mu$-\emph{spectral representation} for $K$, 
\[
U=(U_j):\mathcal{H}\to\oplus_{j=1}^N L^2(\mathcal{M}_j,d\mu),
\]
\[
U\psi=(U_j\psi)_{j=1,\ldots,N},
\]
where $N\in\mathbb{N}\cup\{\infty\}$, the sets $\mathcal{M}_j\subseteq\mathbb{R}$ are $\mu$-measurable with $\mathcal{M}_{j+1}\subset \mathcal{M}_j$ and such that the spectral multiplicity of $K$ on $\mathcal{M}_{j}\backslash \mathcal{M}_{j+1}$ is $j$ (in the case that $\mu\left(\mathcal{M}_{j}\backslash \mathcal{M}_{j+1}\right)>0$, including $j=0$ with $\mathcal{M}_0=\mathbb{R}$). Moreover, 
\[
UKU^{-1}=M_{\text id},
\] 
where $M_{\text id}$ is the multiplication operator with the function ${\text id}(x)=x$ (and $M_g h= g\cdot h$).

\begin{prop}[{\cite[Theorem 1]{PoStWe89}}]\label{prop:PSW}
Let $K$ be a selfadjoint operator in $\mathcal{H}$ and $\mu$ a spectral measure for $K$. Let $U$ be a $\mu$-spectral representation of $K$. Suppose there is a bounded continuous function $\gamma:\mathbb{R}\to\mathbb{C}$ with $|\gamma|>0$ on $\sigma(K)$ such that $\gamma(K)T^{-1}$ is a Hilbert-Schmidt operator. Define, $\mathcal{E}=\left\{f\in \mathcal{H}_+(T):\, Kf\in \mathcal{H}_+(T)\right\}$. Then there exist $\mu$-measurable functions $\varphi_j: \mathcal{M}_j\to \mathcal{H}_-(T),\, j=1,2,\ldots$, such that \\
\begin{itemize}
\item[a)] $(U_j f)(\lambda)=\langle\varphi_j(\lambda),f\rangle$ for $f\in\mathcal{H}_+(T)$ and $\mu$-a.e.\ $\lambda\in \mathcal{M}_j$.
\item[b)] $\langle\varphi_j(\lambda),Kf\rangle=\lambda\langle\varphi_j(\lambda),f\rangle$ for $f\in\mathcal{E}$ and $\mu$-a.e.\ $\lambda\in \mathcal{M}_j$.
\end{itemize}
\end{prop}

\begin{proof}[Proof of Theorem~\ref{thm:revSch}]
Let $\Cc_c(G)\subseteq\ell^2(G)$ denote the linear subspace of functions with finite support.
Define the operator $T:\Cc_c(G)\to\Cc_c(G)$ by $u\mapsto \Vert\omega\Vert_{\infty}\cdot \omega^{-1}\cdot u$. Define further $\|u\|_T:= \sqrt{\|u\|^2+\|Tu\|^2}$, and the associated closed operator on $\overline{\Cc_c(G)}^{\|\cdot\|_T}$ for which we use the same symbol $T$.
Since $\omega^{-1}\geq \frac{1}{\Vert\omega\Vert_{\infty}}$, we have $T\geq1$.
Thus $T^{-1}>0$ exists and is defined on $\ell^2(G)$.
Moreover, we claim that $T$ is selfadjoint. Since $T$ is symmetric and densely defined on $\mathcal{H}$ (e.g.\ on compactly supported functions) it is enough to show that $D(T)=D(T^*)$. Indeed, by Cauchy-Schwartz inequality
\[
|(T\psi,\phi)|^2\leq \Vert T\phi\Vert^2 \,\Vert \psi\Vert^2,
\]
and thus any $\phi\in D(T)$ is also in $D(T^*)$. On the converse, if $\phi\notin D(T)$ the linear functional $\psi\to (T\psi,\phi)$ is not bounded and thus also $\phi\notin D(T^*)$.

We will apply Proposition~\ref{prop:PSW} with the space $\mathcal{H}=\ell^2(G)$ with our Schr\"odinger operator shifted $K=H+\Vert H\Vert +1$, and the operator $T$ defined above.
We have $\Hh_+=D(T)=\overline{\Cc_c(G)}^{\|\cdot\|_T}$  as sets and $\Hh_-:=\overline{\ell^2(G)}^{\|\cdot\|_-}$ where $\|u\|_-:= \|T^{-1}u\|$, and
$$
\Hh_+\subset \ell^2(G)\subset \Hh_- \,.
$$

Let $U=(U_j)$ be a $\mu$-spectral representation of $K$ with measurable $\mathcal{M}_j\subseteq\mathbb{R}$ for $1\leq j\leq N\in\mathbb{N}\cup\{\infty\}$ satisfying $\mathcal{M}_j\supseteq \mathcal{M}_{j+1}$, see \cite{PoStWe89,BdmSt03} for details. Specifically, $U=(U_j):\ell^2(G)\to \oplus_{j=1}^N L^2(\mathcal{M}_j,d\mu)$ satisfies $U g(H) = M_g U$ for every measurable function $g$.

As consequence of Lemma~\ref{lem-HS}, given below, and Proposition~\ref{prop:PSW} (with $\gamma(x)=x$) there exist $\mu$-measurable functions $\varphi_j:\mathcal{M}_j\to\Hh_-$ for $1\leq j\leq N$ such that
\begin{equation}\label{eq:Kpsilambdapsi}
\langle K\psi, \varphi_j(\lambda)\rangle
	= \lambda \langle \psi,\varphi_j(\lambda)\rangle
\end{equation}
for $\mu$-a.e.\ $\lambda\in \mathcal{M}_j$ and all $\psi\in\Ee:=\{ \psi'\in D(K)\cap \Hh_+ \;:\; K\psi'\in\Hh_+\}$.

Note that for every $\lambda\in\sigma(K)$ and $\delta>0$, $\mu\left(\lambda-\delta,\lambda+\delta\right)>0$. Thus the spectral multiplicity of $K$ on $\left(\lambda-\delta,\lambda+\delta\right)$ is greater than zero. Hence $\left(\lambda-\delta,\lambda+\delta\right)\cap \mathcal{M}_1 \neq \emptyset$, and $\lambda\in \overline{\mathcal{M}_1}$. Thus 
\[\sigma(K)=\overline{\mathcal{M}_1}.\]

Thus, the case $j=1$ of \eqref{eq:Kpsilambdapsi} will be sufficient for us. 
As a consequence of Lemma~\ref{lem-OpCor} given below we have that
\[
\langle K\psi, \varphi_1(\lambda)\rangle
	= \lambda \langle \psi,\varphi_1(\lambda)\rangle
\]
for all $\psi\in\Cc_c(G)$  and $\mu$-a.e.\ $\lambda\in\mathcal{M}_1$. Let $A\subseteq \mathcal{M}_1$ be of full measure ($\mu(A)=1$) such that the previous identity holds. Since $\Cc_c(G)\subseteq \ell^2(G)$ is dense, $\varphi_1(\lambda)$ is a generalized eigenfunction of $K$ for $\lambda\in A$. We know that $\varphi_1(\lambda)\in\Hh_-$ for all $\lambda\in A$. Recall that $\Hh_-:=\overline{\ell^2(G)}^{\|\cdot\|_-}$ where $\|u\|_-:= \|T^{-1}u\|$. Hence, $\|T^{-1}\varphi_1(\lambda)\|<\infty$ for $\lambda\in A$. Since
\[
\|T^{-1}\varphi_1(\lambda)\|^2 = \sum_{v\in V} \left(\varphi_1(\lambda)(v)\omega(v)\right)^2\,,
\]
this means $\varphi_1(\lambda)(v) \omega(v)\in\ell^2(G)$ for all $\lambda\in A$. Since $\ell^2\subseteq\ell^{\infty}$ we get that
 there is a constant $c(\lambda)>0$ such that $|\varphi_1(\lambda)(v)|\leq c(\lambda) \omega(v)^{-1}$ for all $\lambda\in A$. Since $\mu(A)=1$, we have proven the desired result.
\end{proof}

Using the notations of Theorem~\ref{thm:revSch}, the following lemmas hold.  
\begin{lemma}
\label{lem-HS}
The operator $K T^{-1}$ is a Hilbert-Schmidt operator.
\end{lemma}

\begin{lemma}
\label{lem-OpCor}
We have $\Cc_c(G)\subseteq \Ee$.
\end{lemma}

\begin{proof}[Proof of Lemma~\ref{lem-HS}]
It is immediate to see that $T^{-1}$ is a Hilbert-Schmidt operator (it is an integral operator and its kernel is square integrable as $\omega$ is square integrable).
Since $H$ is bounded and the Hilbert-Schmidt operators form an ideal in the bounded operators on $\ell^2(G)$, we get that $K T^{-1}$ is a Hilbert-Schmidt operator.
\end{proof}

\begin{proof}[Proof of lemma~\ref{lem-OpCor}]
Since $K$ is a bounded operator we have $D(K)=\ell^2(G)$. Furthermore, for $\psi'\in\Cc_c(G)$, we have $K\psi'\in\Cc_c(G)\subseteq D(G)=\Hh_+$. Thus, $\Cc_c(G)\subseteq \Ee$ follows.
\end{proof}

\begin{remark}
Since $\overline{\mathcal{M}_1}=\sigma(K)$ we conclude that $\sigma(K)$ is the closure of the set of eigenvalues corresponding to generalized eigenfunctions for $K$ satisfying \eqref{eq:phiomegainell2}.
\end{remark}

\begin{lemma} \label{lem-l2}
Given an infinite rooted graph $(G,v_0)$, denote
\[\omega_G(v):=\frac{1}{\sqrt{S_{v_0}\left(\left|v\right|\right)} (|v|+1)}.\]
Then $\omega_G\in\ell^2(G)$.
\end{lemma}
The proof is straightforward.

\begin{corollary}
For a.e.\ $\lambda\in\sigma\left(H\right)$ there exists a corresponding generalized eigenfunction satisfying for any $v\in G$,
\begin{equation*} \label{eq:graphefGrowthagain}
\left|\varphi\left(v\right)\right|\leq \left(|v|+1\right)\cdot \sqrt{S_{v_0}\left(\left|v\right|\right)}.
\end{equation*}
\end{corollary}

For example on a $d$-regular tree we get a generalized eigenfunction satisfying
for any $v\in T$,
\begin{equation*} \label{eq:treeGrowth}
\left|\varphi\left(v\right)\right|\leq C\cdot\left(d-1\right)^{\left|v\right|/2}\left(|v|+1\right),
\end{equation*}
where $C>0$ is a constant.

On a graph $G$ of sub-exponential growth we get a generalized eigenfunction satisfying for all $\gamma>1$ and $v\in G$,
\begin{equation*} \label{eq:treeGrowth}
\left|\varphi\left(v\right)\right|\leq C\cdot\gamma^{\left|v\right|/2}\left(|v|+1\right),
\end{equation*}
where $C>0$ is a constant.
This growth rate is also sub-exponential.

\section{Examples}
We shall use the characterization
\begin{equation}\label{eq:sigmaesscharagain}
\bigcup_{L\text{ is an \ensuremath{\mathcal{R}}-limit of \ensuremath{H}}}\sigma\left(L\right)=\sigma_{\text{ess}}\left(H\right),
\end{equation}
which we prove here for graphs of uniform sub-exponential growth, to calculate $\sigma_{\text ess}(H)$ in some examples.

\subsection{Variations of $\mathbb{Z}^n$}

Let $1<n\in\mathbb{N}$. We shall construct a graph which we denote by $Z_{n\times n}$ by the following procedure:

\begin{figure}[ht!]
\centering
\includegraphics[width=100mm]{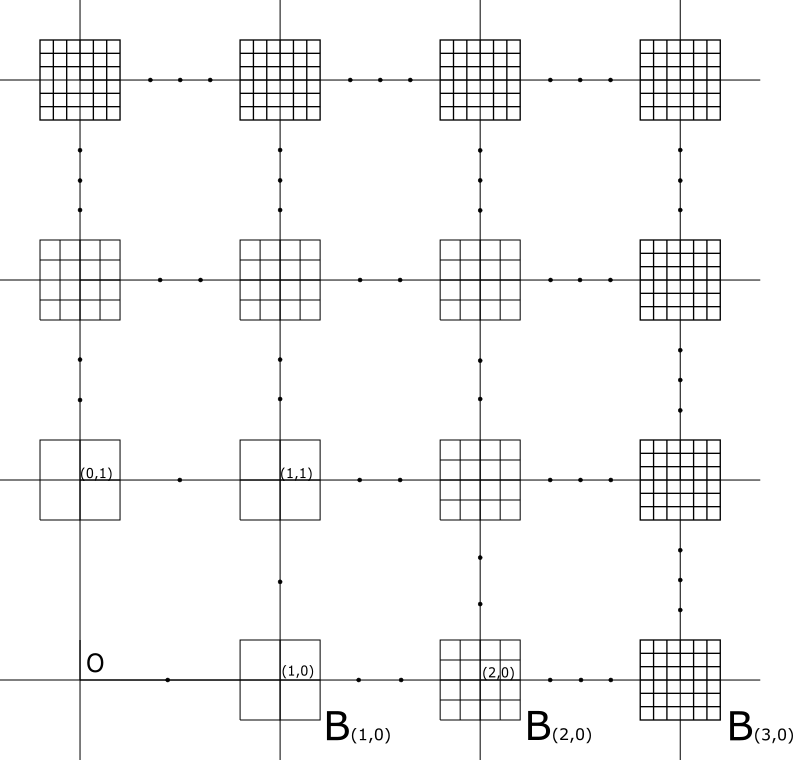}\\
\caption{The graph $Z_{2\times2}$. \label{figure:Z_2times2}}
\end{figure}

\begin{itemize}
\item Denote by $B^n_L$ the box of side length $2L+1$ contained in $\mathbb{Z}^n$, i.e.\ the vertex set is
\[
V\left(B^n_L\right)=[-L,L]^n=\left\{x=(x_1,\ldots,x_n)\in \mathbb{Z}^n\,\Big\vert\, |x_k|\leq L\,\,\forall k\in1,\ldots n \right\}.
\]
\item For each point $x\in\mathbb{Z}^n$ we shall associate the graph $B_x^n\equiv B_{\Vert x \Vert_\infty}^n$.
\item We connect each adjacent pair of boxes $B_x^n$ and $B_{x+e_j}^n$ by a line. The connection is done between the center points of the corresponding boundary surfaces and includes a sequence of vertices and edges of length $\max\left(\Vert x\Vert_\infty,\Vert x+e_j\Vert_\infty\right)$.
\end{itemize}

For example (a portion of) the graph $Z_{2\times 2}$ is drawn in Figure~\ref{figure:Z_2times2}.
Consider the adjacency operator $A$ on $Z_{n\times n}$. The $\mathcal{R}$-limits of $A$ are again the adjacency operator, but this time defined on the following graphs (see Figure~\ref{figure:rlimofZ22} for the case $n=2$):
\begin{enumerate}
\item The line $\mathbb{Z}$.
\item The grid $\mathbb{Z}^n$.
\item For each $1\leq k\leq n$ the sub-grid
\[
\mathbb{Z}^n_{{\scriptscriptstyle(\geq0)}^k}=\left\{ x=(x_1,\ldots,x_n)\in\mathbb{Z}^n\Big\vert x_1,\ldots,x_k\geq 0\right\}.
\]
\item A graph which we denote by
$\widetilde{\mathbb{Z}}^n_{\scriptscriptstyle \geq0}$ which is an half grid $\mathbb{Z}^n_{\scriptscriptstyle \geq0}$ connected to the half line (origin to origin).
\end{enumerate}

\begin{figure}[ht!]
\centering
\includegraphics[width=130mm]{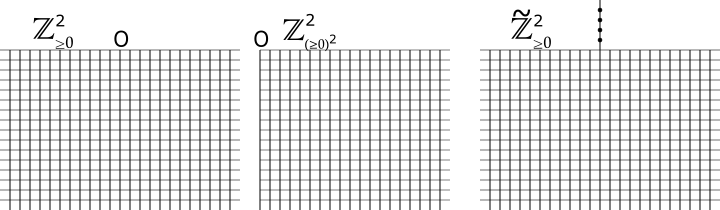}\\
\caption{Some of the $\mathcal{R}$-limits of $Z_{2\times2}$ (beside $\mathbb{Z}$ and $\mathbb{Z}^2$). \label{figure:rlimofZ22}}
\end{figure}

The spectrum of $A_{\mathbb{Z}}$ is $[-2,2]$. For $\mathbb{Z}^n$ and for the sub-grids $\mathbb{Z}^n_{{\scriptscriptstyle(\geq0)}^k}$ the spectrum is $[-2n,2n]$ (the spectrum is contained in $[-2n,2n]$ since for each of these grids we have that the spectral radius  $\rho\left(A_{\mathbb{Z}^n}\right)\leq\left\Vert A_{\mathbb{Z}^n}\right\Vert =2n$. On the other hand, for each $\lambda\in[-2n,2n]$ we can write approximate eigenfunctions by truncations of a function of the form $e^{ik_1 x_1+\ldots+\ldots +ik_n x_n}$, where $k_1,\ldots,k_n\in \mathbb{R}$ satisfying $\lambda=2\cos k_1 +\ldots+\cos k_n$).

As for $\widetilde{A}=A_{\widetilde{\mathbb{Z}}^n_{\scriptscriptstyle \geq0}}$, notice first that this operator is a finite rank perturbation of the operator $A_0=A_{\mathbb{N}}\oplus A_{\mathbb{Z}^n{\scriptscriptstyle \geq0}}$, thus $\sigma_{\text ess}(\widetilde{A})=\sigma_{\text ess}\left(A_0\right)=[-2n,2n]$. On the other hand $\rho(\widetilde A)\leq\Vert \widetilde A\Vert =2n$ and thus also $\sigma(\widetilde A)=[-2n,2n]$.

Thus we conclude from \eqref{eq:sigmaesscharagain} that
\[\sigma_{\text ess}\left(A\right)=[-2n,2n].\]
Moreover, since the spectral radius on $Z_{n\times n}$ is again $\leq 2n$, we get also that $\sigma\left(A\right)=[-2n,2n]$.

\begin{remark}
The spectrum of the adjacency operator on the box $[1,L]^n$ is composed of the set of eigenvalues
\[
\sigma\left(A_{[1,L]^n}\right)=\left\{\sum_{i=1}^n 2 \cos \frac{\pi k_i}{L+1}\, \bigg | \, i=1,\ldots,n,\, k_i=1,\ldots,L\right\}. \]
Thus, it is not surprising we get $\sigma\left(A_{Z_{n\times n}}\right)=[-2n,2n]$, which is the limit of the spectra of boxes as $L\to\infty$ (see e.g.\ \cite[Theorem 4.12]{MoharWoess}).
\end{remark}

\begin{remark}
More generally, let $G$ be a graph such that $A_{\mathbb{Z}^n}$ appears as an $\mathcal{R}$-limit of the adjacency operator $A_G$. If the vertex degree of $G$ is bounded by $2n$ then we can always conclude that
\[
\sigma(A_G)=\sigma_{\text ess}(A_G)=[-2n,2n].
\]
Indeed:
\begin{enumerate}
\item Since the vertex degree is bounded by $2n$ we get that $\rho(A_G)\leq\Vert A_G\Vert \leq 2n$.
\item On the other hand, since $\mathbb{Z}^n$ appears as an $\mathcal{R}$-limit of  $A_G$ we get from $\eqref{eq:sigmaesscharagain}$ that $[-2n,2n]=\sigma\left(A_{\mathbb{Z}^n}\right)\subseteq\sigma_{\text ess}(A_G)$.
\end{enumerate}
\end{remark}

\subsection{Sparse trees with sparse cycles} \label{sec:example2}
We start with a rooted tree $(T,v_0)$ satisfying:
\begin{enumerate}
\item $T$ is spherically homogeneous tree of bounded degree, i.e.\ any vertex $v\in T$ is connected to $\kappa(|v|)$ vertices of distance $|v|+1$ from the root, where $\kappa:\mathbb{N}\to\mathbb{N}$ is a (bounded) function.
\item $T$ is sparse, i.e.\ $\kappa$ satisfies:
\[
\kappa(n)=
\begin{cases} k_n \ \ \ \ n\in {L_n} \\ 1 \ \ \ \ \ \ \text{otherwise}, \end{cases}
\]
where $\left\{k_n\right\}_{n\in\mathbb{N}}$ is a bounded sequence and $\left\{L_n\right\}_{n\in\mathbb{N}}$ is a monotonic sequence satisfying $\lim_{n\to\infty} L_{n+1}-L_n=\infty$ (following e.g.\ \cite{Breuer}).
\end{enumerate}

We construct an example of a graph $G$ which is sparse and spherically homogeneous.
This graph can be seen as a sparse spherically homogeneous tree $T$ with additional edges. The added edges form cycles at distances $C_n$ from the root, where $\{C_n\}_{n\in\mathbb{N}}$ is a sequence satisfying (see Figure~\ref{figure:sparsesym}):
\[
L_n>C_n\geq L_{n-1}, \ \
L_n -C_n \to \infty, \ \text{and} \
C_n - L_{n-1}\to \infty.
\]
More formally, recall the definition
\[
\mathcal{S}(n)=\left\{v\in T\,\Big|\, |v|=n\right\}
\]
for the set vertices of distance $n$ from the root (and let $S(n)=|\mathcal{S}(n)|$).
We shall enumerate the vertices of $\mathcal{S}(n)$ in a natural way, i.e.\ write $\mathcal{S}(n)=(v_1^{(n)},\ldots,v_N^{(n)})$ where $v_1^{(n)}$ is an arbitrary vertex in $\mathcal{S}(n)$ and $v_k^{(n)}$ is defined inductively to be a vertex $v\in \mathcal{S}(n)\big\backslash\{v_1^{(n)},\ldots,v_{k-1}^{(n)}\}$ such that $\text{dist}(v_{k-1}^{(n)},v)$ is minimal. Now, the graph $G$ is defined with $V(G)=V(T)$ and
\[
E(G)=E(T)\cup \left(\bigcup_{n\in\mathbb{N}} E_n\right),
\]
where,
\[
E_n=\left(\bigcup_{k=1}^{S(C_n)-1} \left(v_k^{(C_n)},v_{k+1}^{(C_n)}\right)\right)\cup \left(v_{S(C_n)}^{(C_n)},v_1^{(n)}\right).
\]

\begin{figure}[ht!]
\centering
\includegraphics[width=110mm]{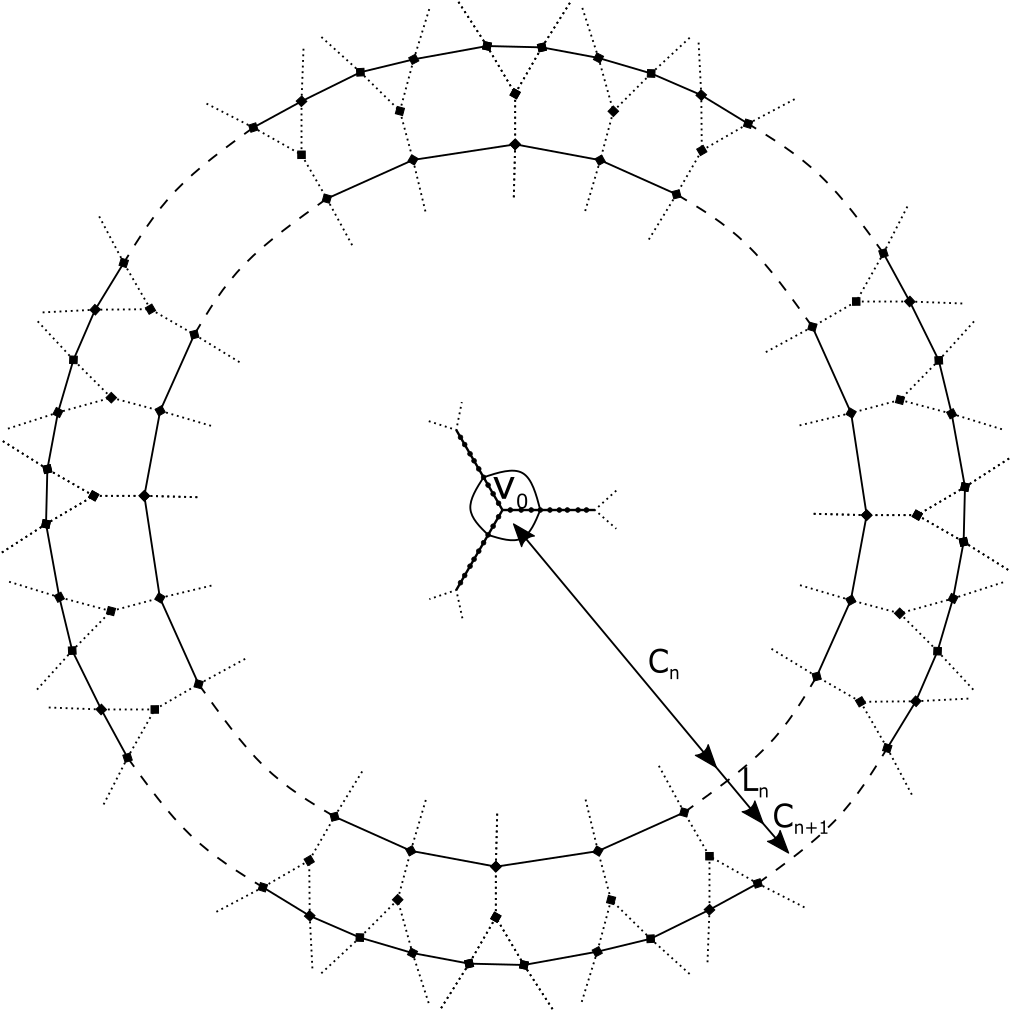}\\
\caption{An example of a sparse tree with sparse cycles. \label{figure:sparsesym}}
\end{figure}

For the sake of this example we take $k_n\equiv2$, and consider the adjacency operator $A=A_G$ defined on this graph.
We choose $\left\{L_n\right\}$ such that the growth rate of $G$ is uniform polynomial, e.g.\ $L_n=2^n$.
The possible $\mathcal{R}$-limits of $A$ are the adjacency operators on the following graphs (see Figure~\ref{figure:rlimsparsesym}):
\begin{itemize}
\item The line $\mathbb{Z}$.
\item A two sided infinite comb graph, denoted by $\mathcal{C}G$ and defined by
\begin{eqnarray*}
V(\mathcal{C}G)=&\left\{v=(k,l)\,\Big|\, k,l\in\mathbb{Z}\right\}, \\
E(\mathcal{C}G)=&\left\{\Big((k,l),(k,l+1)\Big)\,\Big|\, k,l\in\mathbb{Z}\right\}\cup\\ &\cup\left\{\Big((k,0),(k+1,0)\Big)\,\Big|\, k\in\mathbb{Z}\right\}.
\end{eqnarray*}
\item A star graph which is composed of $3$-copies of $\mathbb{N}$ glued together at $0$, denoted by $IS_3$.
\end{itemize}

\begin{figure}[ht!]
\centering
\includegraphics[width=100mm]{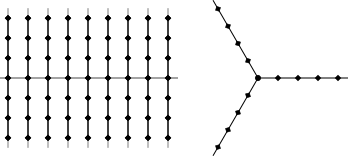}\\
\caption{Some of the $\mathcal{R}$-limits of the sparse symmetric graph (beside $\mathbb{Z}$): the infinite comb graph (left) and the star graph $IS_3$ (right). \label{figure:rlimsparsesym}}
\end{figure}

\begin{remark}
If we allow general (bounded) $\left\{k_n\right\}_{n\in\mathbb{N}}$ we will get additional $\mathcal{R} $-limits.
Denote by $\mathcal{K}$ the set of values which appear infinitely many times in $\{k_n\}_{n\in\mathbb{N}}$. For any $k\in \mathcal{K}$ we will get the graph $IS_{k+1}$  which is composed of $k+1$-copies of $\mathbb{N}$ glued together at $0$.
\end{remark}

\begin{lemma} \label{lem:CGspectra} $\sigma(A_{\mathcal{C}G})=[-2\sqrt{2},2\sqrt{2}]$
\end{lemma}

\begin{lemma} \label{lem:SGspectra} $\sigma\left(A_{IS_k}\right)=[-2,2]\cup\left\{\frac{k}{\sqrt{k-1}}\right\}$
\end{lemma}

As a conclusion of \eqref{eq:sigmaesscharagain} we get that
\[
\sigma_{\text ess}(G)=[-2\sqrt{2},2\sqrt{2}].
\]

\begin{proof}[Proof of Lemma \ref{lem:CGspectra}]
We shall use the periodicity of the problem to calculate the spectrum of $A_{\mathcal{C}G}$ (following e.g.\ \cite[Chapter 5]{SimonSz} and \cite[Chapter XIII.16]{ReedSimon4}). 

$A_{\mathcal{C}G}$ is unitary equivalent to a direct integral of the operators
\[
A^{(\theta)}=A_{\mathbb Z}+2\cos\theta \delta_0.
\]
To show this, define a Fourier transform operator
\[
\mathcal{F}:\ell^2(\mathcal{C}G)\to L^2\left(\partial \mathbb{D},\frac{d\theta}{2\pi}; \ell^2(\mathbb{Z})\right)
\]
by ($l\in\mathbb{Z}$)
\[
\left(\mathcal{F}\psi\right)(\theta,l)=\sum_{k=-\infty}^{\infty} \psi(k,l) e^{-ik\theta},
\]
where we define it first for $\psi\in\ell^1$ and extend to $\ell^2$ using
\[
\sum_{l\in\mathbb{Z}}\int_{\partial\mathbb{D}}\left\Vert\mathcal{F}\psi(\cdot,l)\right\Vert^2\frac{d\theta}{2\pi} = \sum_{k,l\in\mathbb{Z}}\left|\psi(k,l)\right|^2.
\]
The inverse of $\mathcal{F}$
\[
\mathcal{F}^{-1}:L^2\left(\partial \mathbb{D},\frac{d\theta}{2\pi}; \ell^2(\mathbb{Z})\right)\to \ell^2(\mathcal{C}g)
\]
is defined by
\[
(\mathcal{F}^{-1}f)(k,l)=\int e^{i k\theta}f(\theta,l)\frac{d\theta}{2\pi}.
\]

It now follows that
\[\left[(\mathcal{F}A\mathcal{F}^{-1})f\right](\theta,l)=\left(A^{(\theta)}f\right)(\theta,l).
\]
Indeed, given $f\in L^2\left(\partial \mathbb{D}; \ell^2(\mathbb{Z})\right)$ we write
\begin{eqnarray*}
&(A\mathcal{F}^{-1}f)(k,l)=\\
&\int\frac{d\theta'}{2\pi}e^{ik\theta'}\left(f(\theta',l-1)+f(\theta',l+1)+\left(e^{-i\theta'}+e^{i\theta'}\right)f(\theta',l)\delta_0(l)\right).
\end{eqnarray*}
Thus,
\begin{eqnarray*}
&\left(\mathcal{F}A\mathcal{F}^{-1}\right)f (\theta,l)=\\
&\sum_{k\in\mathbb{Z}}\int\frac{d\theta'}{2\pi}e^{ik(\theta'-\theta)}\left(f(\theta',l-1)+f(\theta',l+1)+2\cos(\theta')f(\theta',l)\delta_0(l)\right)\\
&=f(\theta,l-1)+f(\theta,l+1)+2\cos(\theta)f(\theta,l)\delta_0(l)=\left(A^{(\theta)}f\right)(\theta,l).
\end{eqnarray*}

Since $\mathcal{F}$ is unitary we have shown the unitary equivalence claimed above. Thus, in order to compute the spectrum of $A_{\mathcal{C}G}$ we need to compute the spectrum of the direct integral. Denote $R=\left(A^{(\theta)}-z\right)^{-1}$, $R_0(z)=\left(A_{\mathbb Z}-z\right)^{-1}$.
By the basic formula of rank-one perturbations we get that
\[
m(z)=\langle\delta_0,R(z)\delta_0\rangle= \frac{m_0(z)}{1+2\cos\theta m_0(z)},
\]
where
\[
m_0(z)=\langle\delta_0,R_0(z)\delta_0\rangle=\frac{1}{\sqrt{z^2-4}}
\]
is the Borel transform corresponding to the adjacency operator on $\mathbb{Z}$ (see e.g\ \cite{SimonSz}).
We know that $\sigma(A_{\mathbb Z})=\sigma_{\textrm{ess}}(A_{\mathbb Z})=[-2,2]$ and therefore $[-2,2]\subseteq \sigma\left(A^{(\theta)}\right)$ for all $\theta$. We can get additional points in the spectrum of $A^{(\theta)}$ if  $1+2\cos\theta m_0(z)$ vanishes. Thus
\[
\sqrt{z^2-4}=-2\cos\theta,
\]
so
\[
z_{\pm}= \pm2\sqrt{1+\cos^2\theta}.
\]
The choice of the $\pm$-branch of the root is determined by requiring that $m(z)=-\nicefrac{1}{z}+O\left(\nicefrac{1}{z^2}\right)$. Notice that the vector $\delta_0$ is not cyclic for $A^{(\theta)}$. However, since the difference $A^{(\theta)}-A_{\mathbb Z}=2\cos\theta(\delta_0,\cdot)\delta_0$ is of rank one we can't get additional points in the spectrum of $A^{(\theta)}$ beside the one found. 
Using continuity in $\theta$ to construct approximate eigenfunctions and integrating over $\theta$ we get that
\[
\sigma\left(A_{\mathcal{C}G}\right)=[-2\sqrt{2},2\sqrt{2}],
\]
where, since the spectrum of $A_{\mathcal{C}G}$ is symmetric both the points $z_{\pm}$ are in the spectrum.
\end{proof}

\begin{proof}[Proof of Lemma \ref{lem:SGspectra}]
We shall calculate the spectrum of the adjacency operator on the graph $IS_k$.
This graph is spherically homogeneous, and thus the operator $A_{IS_k}$ is unitary equivalent to a direct sum of one dimensional Jacobi operators.
Indeed, according to Theorem~2.4 of \cite{Breuer}, 
\[
A_{IS_k}\cong J_0\oplus A_{\mathbb N}\oplus\ldots\oplus A_{\mathbb N},
\]
where $J_0$ is a Jacobi matrix with parameters 
\[
a_n=\begin{cases}
\sqrt{k} & n=1 \\
1 & n>1,
\end{cases}
\]
\[b_n\equiv0,\]
and $A_{\mathbb N}$ appears in the direct sum $(k-1)$-times.
Since $J_0$ is a finite rank perturbation of $A_{\mathbb N}$ we have that $\sigma_{\text ess}(J_0)=\sigma_{\text ess}(A_{\mathbb N})=\sigma(A_{\mathbb N})=[-2,2]$.
Thus we should only calculate the discrete spectrum of $J_0$. By coefficient stripping (Theorem~3.2.4 in \cite{SimonSz}) we get the following relation for the $m$-function of $J_0$
\[
m(z)=\frac{-1}{z+k m_{\mathbb N}(z)},
\]
with
\[
m(z)=\langle\delta_1,(J_0-z)^{-1}\delta_1\rangle.
\]
Notice that $\delta_1$ is a cyclic vector for $J_0$, thus additional points in the spectrum of $J_0$ exist only if
\begin{equation} \label{eq:mISksol}
z+k m_{\mathbb N}(z)=0.
\end{equation}
Using the known expression (see, e.g., \cite{SimonSz})
\begin{equation*}
m_{\mathbb{N}}\left(z\right)=\frac{-z+\sqrt{z^{2}-4}}{2}
\end{equation*}
we get that the only solution of \eqref{eq:mISksol} is $z_0={k\over\sqrt{k-1}}$.
Thus $z_0\in\sigma\left(A_{IS_k}\right)$, and
\[
\sigma\left(A_{IS_k}\right)=[-2,2]\cup\left\{\frac{k}{\sqrt{k-1}}\right\}.
\]

\end{proof}

\end{document}